\documentclass{amsart}
\usepackage[utf8]{inputenc}
\usepackage{amsfonts}
\usepackage{amsmath}
\usepackage{amssymb}
\usepackage{amsthm}
\usepackage{mathrsfs} 
\usepackage{enumerate}
\usepackage{xcolor}
\usepackage{rotating}
\usepackage{hyperref}
\usepackage{graphicx}

\usepackage[justification=centering]{caption}
\usepackage{tikz}
\usetikzlibrary{decorations.shapes}
\usepackage{lipsum}
\usetikzlibrary{arrows.meta,automata,quotes}
\theoremstyle{plain}
\newtheorem{thm}{Theorem}%[section]
\newtheorem{cor}[thm]{Corollary}
\newtheorem{lem}[thm]{Lemma}
\newtheorem*{lem*}{Lemma~4}
\newtheorem{prop}[thm]{Proposition}
\newtheorem{conj}[thm]{Conjecture}

\theoremstyle{definition}
\newtheorem{defi}[thm]{Definition}
\newtheorem{rem}[thm]{Remark}
\newtheorem{example}[thm]{Example}
\newtheorem{problem}[thm]{Problem}
\newtheorem{obs}[thm]{Observation}
\newcommand{\ignore}[1]{}
\newcommand{\N}{\ensuremath{\mathcal N}}
\renewcommand{\P}{\ensuremath{\mathcal P}}
\newcommand{\B}{\mathcal B}

\newcommand{\X}{\mathcal X}
\newcommand{\n}{\mathbb N}
\newcommand{\nz}{\mathbb N_0}

\newcommand{\x}{\boldsymbol x}
\newcommand{\y}{\boldsymbol y}
\newcommand{\s}{\boldsymbol s}
\renewcommand{\ge}{\geqslant}
\renewcommand{\le}{\leqslant}

\renewcommand{\leq}{\leqslant}
\renewcommand{\ngeq}{\ngeqslant}

\usepackage{algorithm,algpseudocode}
\newcounter{algsubstate}

 \usepackage{bbm,dsfont}

\title{Subtraction games in more than one dimension} 

\author{Urban Larsson}
\address{Urban Larsson, IIT Bombay, India}
\email{larsson@iitb.ac.in}
\author{Indrajit Saha}
\address{Indrajit Saha, Kyushu University, Japan}
\email{indrajit@inf.kyushu-u.ac.jp}
\author{Makoto  Yokoo}
\address{Makoto  Yokoo, Kyushu University, Japan}
\email{yokoo@inf.kyushu-u.ac.jp}
\date{\today}
\begin{document}
%%%%%%%%%%%%%%%%%%%%%%%%%%%%%%%%%%%%%%%%%%%%%%%%%%%%%%%%%%%
\begin{abstract}
\label{abstract}
This paper concerns two-player alternating play combinatorial games (Conway 1976) in the normal-play convention, i.e. last move wins. Specifically, we study impartial vector subtraction games on tuples of nonnegative integers (Golomb 1966), with finite subtraction sets.  In case of two move rulesets we find a complete solution, via a certain $\mathcal{P}$-to-$\mathcal{P}$ principle (where $\mathcal{P}$ means that the previous player wins). Namely $x \in \mathcal{P}$ if and only if $x +a +b \in \mathcal{P}$, where $a$ and $b$ are the two move options. Flammenkamp 1997 observed that, already in one dimension, rulesets with three moves can be hard to analyze, and still today his related conjecture remains open. Here, we solve instances of rulesets with three moves in two dimensions, and conjecture that they all have regular outcomes. Through several computer visualizations of outcomes of multi-move two-dimensional rulesets, we observe that they tend to partition the game board into periodic mosaics on very few regions/segments, which can depend on the number of moves in a ruleset. For example, we have found a five-move ruleset with an outcome segmentation into six semi-infinite slices. In this spirit, we develop a coloring automaton that generalizes the $\mathcal{P}$-to-$\mathcal{P}$ principle. Given an initial set of colored positions, it quickly paints the $\mathcal{P}$-positions in segments of the game board. Moreover, we prove that two-dimensional rulesets have row/column eventually periodic outcomes. We pose open problems on the generic hardness of two-dimensional rulesets; several regularity conjectures are provided, but we also conjecture that not all rulesets have regular outcomes. 

\end{abstract}

%%%%%%%%%%%%%%%%%%%%%%%%%%%%%%%%%%%%%%%%%%%%%%%%%%%%%%%%%%%%%%%%%%%%%%%%%%%%%%%%%%%%

\maketitle
%%%%%%%%%%%%%%%%%%%%%%%%%%%%%%%%%%%%%%%%%%%%%%%%%%%%%%%%%%%%%%%%%%%%%%%%%%
\section{Introduction}
\label{sec:intro}
Figure~\ref{fig:dilbox} depicts a typical win-loss tessellation of games studied in this paper. Another example is the following children game. 
\begin{figure}[htbp!]
\begin{center}
\includegraphics[width=7.5cm]{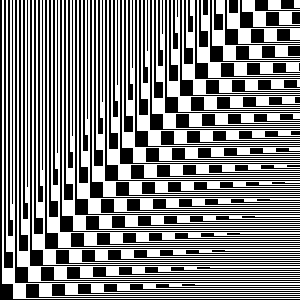}
\end{center}
\caption{``Diluted Boxes". The picture represents a win/loss tiling of a combinatorial game.%the game $S=\{(2,16),(13,1)\}$, size 600 by 600.
}
\label{fig:dilbox}
\end{figure}
%%%%%%%%%%%%%%%%%%%%%%%%%%%%%%%%%%%%%%%%%%%%%%%%%%%%%%%%%%%%%%%%%%%%%%%%%%

Squirrel plays a game with his best friend the Crow. On the table, there are two piles consisting of  $5$ peanuts and $6$ walnuts respectively. The players take turns to collect either $1$ peanut and $3$ walnuts or $2$ peanuts and $1$ walnut.
Suppose Crow is very hungry and that walnut is her favourite food. By her first move, she collects $1$ peanuts and $3$ walnuts. After this, the two piles consist of $4$ peanuts and $3$ walnuts. Now, when Squirrel copies the move of Crow then there is no further option for Crow. She is very disappointed. The next time she starts this game, she decides instead to grab $2$ peanuts and $1$ walnut. This leaves $3$ peanuts and $5$ walnuts on the table. Squirrel has a moment of realization: whatever he does, Crow will get the last move. He suggests to his friend that the next time they play, the player with the last move wins all the nuts, and that he gets to start. See Figure~\ref{fig:Crow} for how to win this game. 

We study a generalization of such games under the convention ``a player who cannot move loses'', which is called {\em normal-play}. Combinatorial games are two-player alternating-turn games with perfect and complete information, i.e., there is no hidden information and no chance move. 
  Popular recreational versions of combinatorial games are {\sc tic-tac-toe}, {\sc chess}, {\sc go}, {\sc checkers} and so on. Such games would admit infinite play if extra conditions were not imposed e.g. the ko-rule in {\sc go} and the fifty-move-draw convention in {\sc chess}. We will restrict attention to games where every sequence of play (not necessarily alternating) is finite. Moreover, we will restrict attention to {\em impartial} games, where, as in the ruleset {\sc squirrel \& crow}, the players have the same set of options, and each option is impartial. The prime example of such games is the ruleset  {\sc nim} \cite{bouton1901nim}. 
  
  A note on terminology used: the style {\sc ruleset} is used for a specific ruleset, but it is also used for a {\sc family of rulesets}. For example, we would write {\sc chess} but also {\sc board games}.

The perfect play {\em  outcome} of an impartial game is a loss or win depending on who starts \cite{berlekamp2004winning}. The starting player wins a game $G$ if $G\in\N$ (\N~ stands for the next player winning set), and otherwise $G\in \P$  (\P~ stands for the previous player winning set). Because of the normal play convention, any terminal position is in \P. As a variation of Zermelo's folk theorem \cite{zermelo1913anwendung}, the sets $\P$ and $\N$ partition the set of all game positions.

Let $\n=\{1,2,\ldots\}$ and $\nz=\n\cup\{0\}$ denote the positive and nonnegative integers respectively. The family of rulesets 
  {\sc finite subtraction} \cite{golomb1966mathematical, berlekamp2004winning} is another popular example of impartial games.\footnote{This class of rulesets is usually called ``subtraction games''. The word ``game'' is used widely without specific connotations, and sometimes it refers to a full ruleset. If we intend a specific game that people can play, then we may use the word position, starting position or ``game position''. Heap games are typically specified as  rulesets, and there are infinitely many possible starting positions.} Those are defined by a finite subtraction set $S\subset \n$, together with a position $x\in \nz$. By moving, a player subtracts their choice from the current position; the next position is $x-s$, for some $s\in S$, provided $s\le x$. If there is no such $s\in S$, then $x$ is a {\em terminal position}, and the current player loses.

 Let us define $d$-dimensional subtraction games, a.k.a the family of rulesets {\sc vector subtraction}. Let $d \in \n$.  The {\em game board} is $\B^d:=\nz^d $. When the dimension $d$ is understood, we may write $\B$ instead of $\B^d$. 
The standard partial order of integer vectors applies, $\boldsymbol x=(x_1,  \ldots,  x_d) \ge \boldsymbol y=(y_1, \ldots,  y_d)$
 if for all $i$, $x_i\ge y_i$. And $\boldsymbol x > \boldsymbol y$
 if $\x \ge \y$, and for some $i$, $x_i > y_i$. Addition, subtraction and scalar multiplication are component-wise, e.g. $\x +\y = (x_1+y_1,  \ldots,  x_d+y_d)$.

\begin{defi}[{\sc vector subtraction}~\cite{golomb1966mathematical,larsson2012operator}]
An instance of {\sc vector subtraction} is an impartial normal-play ruleset played on $d$-tuples of non-negative integers $\x\in \B^d$. A set of  $d$-tuples of nonnegative integers, $S\subseteq \B^d\setminus \{\boldsymbol 0\}$ defines the move options. A move consists in subtracting some $ \s \in  S$ from $\x$, for which  $\x-\s \ge \boldsymbol 0$. 
\end{defi}

{\sc vector subtraction} is included in the class of {\em invariant} games \cite{DuRi2010, LaHeFr2011} , in the sense that the validity of a move $\s$ does not depend on the current position $\x$ (provided that $\x-\s \ge \boldsymbol 0$).  

We study {\sc finite vector subtraction}, i.e. the instances for which $|S|<\infty$. 
Most parts of this paper concern  {\sc finite two-dimensional subtraction}. As an example of computations of the outcomes of such games, see  Figure~\ref{fig:dilbox} for the ruleset $S=\{(13,1),(2,16)\}$, and Figure~\ref{fig:Crow} for {\sc crow \& squirrel}. The black cells correspond to the \P-positions.

\begin{figure}[htbp!]
\begin{center}
\begin{tikzpicture}[scale = 0.5,
% [%%%%%%%%%%%%%%%%%%%%%%%%%%%%%%
         box/.style={rectangle,draw=gray!60!yellow,  thick, minimum size=.5 cm},
     ]%%%%%%%%%%%%%%%%%%%%%%%%%%%%%%

%\colorlet{yellow}{gray!60!yellow}
 %\draw[step=1cm,yellow, thick] (0, 0) grid (9,9); 

\node[box,fill=black] at (0,0){};  
\node[box,fill=black] at (0,1){};  
\node[box,fill=black] at (0,2){};  
\node[box,fill=black] at (0,3){};  
\node[box,fill=black] at (0,4){};  
\node[box,fill=black] at (0,5){};  
\node[box,fill=black] at (0,6){};  
\node[box,fill=black] at (0,7){};  
\node[box,fill=black] at (0,8){}; 
\node[box,fill=black] at (0,9){}; 
\node[box,fill=black] at (1,0){};  
\node[box,fill=black] at (2,0){};  
\node[box,fill=black] at (3,0){};  
\node[box,fill=black] at (4,0){};  
\node[box,fill=black] at (5,0){};  
\node[box,fill=black] at (6,0){};  
\node[box,fill=black] at (7,0){};  
\node[box,fill=black] at (8,0){}; 
\node[box,fill=black] at (9,0){}; 
\node[box,fill=black] at (1,1){}; 
\node[box,fill=black] at (1,2){};
\node[box,fill=black] at (3,4){};
\node[box,fill=black] at (3,5){};
\node[box,fill=black] at (3,6){};
\node[box,fill=black] at (3,7){};
\node[box,fill=black] at (3,8){};
\node[box,fill=black] at (3,9){};
\node[box,fill=black] at (4,4){};
\node[box,fill=black] at (4,5){};
\node[box,fill=black] at (4,6){};

\node[box,fill=black] at (4,2){};
\node[box,fill=black] at (5,2){};
\node[box,fill=black] at (6,2){};
\node[box,fill=black] at (7,2){};
\node[box,fill=black] at (8,2){};
\node[box,fill=black] at (9,2){};

\node[box,fill=black] at (5,4){};
\node[box,fill=black] at (6,4){};
\node[box,fill=black] at (7,4){};
\node[box,fill=black] at (8,4){};
\node[box,fill=black] at (9,4){};

\node[box,fill=black] at (7,6){};
\node[box,fill=black] at (8,6){};
\node[box,fill=black] at (9,6){};
\node[box,fill=black] at (6,8){};
\node[box,fill=black] at (6,9){};
\node[box,fill=black] at (7,9){};
\node[box,fill=black] at (7,8){};
\node[box,fill=black] at (8,8){};
\node[box,fill=black] at (9,8){};

 \foreach \x in {0,1,...,9}{
     \foreach \y in {0,1,...,9}
         \node[box] at (\x,\y){};
 }

% \draw (-.7, -0.7) node {$0$};
% \draw (-0.8, 9.5) node {$9$};
% %\draw (0, 0) node {$0$};
% \draw (9.5, -0.8) node {$9$};

% \draw (5, -1.5) node {Peanuts};
% \draw (-1.5, 5) node {Walnuts};

\draw (0.1 ,-1) node {$0$};
\draw (9, -1) node {$9$};

\draw (-1, 0.1) node {$0$};
\draw (-1, 9) node {$9$};

\draw (5, -1.5) node {Peanuts};
\draw (-2.5, 5) node {Walnuts};
\draw (5,6) node {$\boldsymbol x$};
\end{tikzpicture}
\end{center}
\caption{The picture illustrates the initial outcomes of the ruleset {\sc crow \& squirrel}, i.e. $S =\{(2,1), (1,3)\}$. Here $\boldsymbol x=(5,6)$ represents the given starting position in the second paragraph. The black cells correspond to the losing positions for the current player, a.k.a. the \P-positions.}
\label{fig:Crow}
\end{figure}
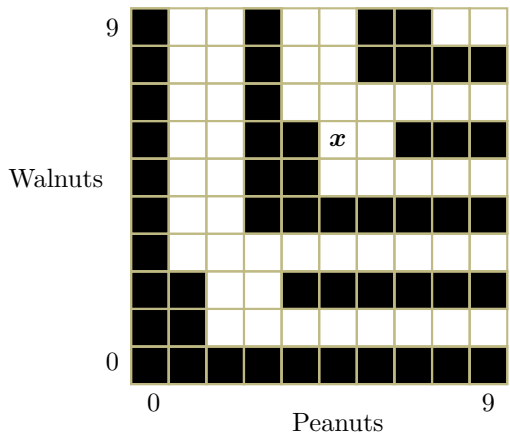

In this paper, we provide a complete characterization of the outcomes of {\sc two-move vector subtraction} (i.e. whenever $|S|=2$ in any dimension). The outcome patterns are much more varied than in the case of one dimension, but still tractable, and the solution is surprisingly simple, given the number of possible relations between the move parameters. It turns out that a case-by-case analysis is not necessary;  instead, we will invoke a certain `\P-to-\P\ principle'  and a `translation rule', where the outcome patterns are inherited from trivial {\sc one-move}. Let us begin by giving here a proof of the aforementioned principle.

\begin{lem}[{\sc two-move} \P-to-\P]
\label{lem:PtoP}
Consider  {\sc two-move vector subtraction}  $S= \{ \s_1, \s_2\}$. Then  $\x \in \P$  if and only if $\x+\s_1 +\s_2\in \P$. 
\end{lem}
\begin{proof}
 Suppose that $\x \in \P$. Then $\{\x+\s_1, \x+\s_2\}\subset \N$. The set of options of the position $\x+\s_1 +\s_2$ is the same set. Hence $\x+\s_1 +\s_2\in \P$.
 
 Suppose $\x+\s_1 +\s_2\in \P$. The set of options is $\{\x+\s_1, \x+\s_2\}\subset \N$. Each of these positions has a \P-position as an option. Hence, if $(\x+\s_1)-\s_2\in\N$ or $(\x+\s_2)-\s_1\in\N$, then this forces $\x\in\P$.   Otherwise $\{\x+\s_1-\s_2,\x+\s_2-\s_1\}\subset \P$. This implies that $\{\x-\s_2,\x-\s_1\}\subset \N$. Hence, $\x\in\P$.
\end{proof}

{\sc additive three-move} concerns rulesets where the third move is the sum of the first two moves. We study a few instances in two dimensions when variants of the \P-to-\P~ update rule continues to hold, and we demonstrate how this idea generalizes to a certain {\em coloring automaton} (Section~\ref{sec:outcomesegment}).
 
The outcomes of any instance of {\sc one-dimensional finite subtraction} are ultimately periodic \cite{golomb1966mathematical}. We prove that such results still hold in two dimensions when we restrict the view to rows or columns: {\sc two-dimensional finite subtraction} is eventually row and column periodic. In contrast, the general problem of `periodicity' for {\sc finite vector subtraction} is much harder, and indeed, a recent result \cite{larsson2013impartial, larsson2013heaps} proves Turing completeness of a slight extension of such games. In a follow up \cite{Gurvich}, 
Gurvich et al. show that {\sc finite  vector subtraction} is EXP-complete.

Instances of {\sc infinite vector subtraction} are well-studied. Popular examples are {\sc nim} ($S=\{a\s\mid \sum_{i} s_i=1,a\in \n\}$) and {\sc wythoff nim} \cite{wythoff1907modification} ($S=\{(a,0),(0,a),(a,a)\mid a\in \n\}$) and their many variations; see for example \cite{duchene2019wythoff} for a survey on a multitude of variations of {\sc wythoff nim}. {\sc nim} has become famous because of its generic capacity in encoding any impartial normal play game, via the famous Sprague-Grundy theory. 

In contrast, not many papers study {\sc finite vector subtraction}. A notable exception is \cite{abuku2019combination}, where the family of rulesets  {\sc cyclic nimhoff} \cite{fraenkel1991nimhoff} is generalized to the setting of {\sc finite two-dimensional subtraction}. The authors provide a  closed form for the nim-values whenever the component has a  so-called $h$-stair  structure (Theorem~2.1 and Definition~2.1 of their paper).\footnote{The $h$-stair definition originates in Siegel's Master's Thesis in 2005 \cite{siegel2005finite}.} From many examples, we observe that the rulesets considered in \cite{abuku2019combination} have exactly one {\em outcome segment}. 

Here, we consider more general two-dimensional rulesets for which one segment is very rare. An {\em outcome segment} is a  special type of geometry on the outcomes, with a distinct regular outcome pattern (a distinguishing tessellation/mosaic), and an outcome segmentation is a finite union of all such outcome segments, for a given ruleset. Intuitively, the segments are separated by thick lines. See Figure~\ref{fig:symasym} for some examples of this phenomenon. If a ruleset admits an outcome segmentation then we can sometimes describe the outcomes by an initial coloring and a finite set of {\em update rules}/{\em a coloring automaton} similar to the situation for eventual periodicity for {\sc finite one dimension},  where the update rule of each \P-position would be the period length, discounted for any preperiod. We discuss this development in Section~\ref{sec:outcomesegment}.

\begin{figure}[htbp!]
 \begin{center}
 \includegraphics[width=6cm]{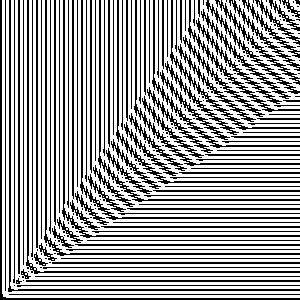}
 \includegraphics[width=6cm]{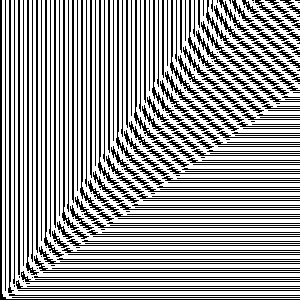}
 \end{center}
 \caption{A symmetric and an asymmetric 4-segmentation respectively: the initial 300 by 300 outcomes of {\sc three move} $S =\{(2,6),(6,2), (3,3)\}$ (left) and $S =\{(2,6),(3,3),(6,1)\}$ (right).}
 \label{fig:symasym}
 \end{figure}

 Guo and Miller \cite{guo2011lattice, guo2013algorithms} consider a vast generalization of {\sc subtraction} called {\em lattice games}. 
 %A lattice game consist of game board $\mathcal{B}$ and a ruleset $S$. 
 %Miller and Guo \cite{guo2009potential,guo2011lattice}
 They conjectured \cite{guo2009potential} that every lattice game has a {\em rational strategy}; roughly its set of \P-positions is finitely generated via a rational generating function. Fink \cite{fink2012lattice} provides a counter example, and proves instead (in two ways) that the proposed set of games is Turing complete. In our context, `rational strategies' would correspond to a variation of the notion of `outcome segments'. But their setting is different and the analogy is not exact. By viewing their games as classical heap games, they require bounded heap sizes, while the number of heaps can be arbitrarily large. In our setting, the number of heaps is bounded, while the individual heaps can be arbitrarily large. See Weimerskirch \cite{mikeW} for a nice discussion on this topic. We will return to this discussion briefly in Section~\ref{sec:picopen} where we pose some related open problems. We have examples in two dimensions, where, for a given ruleset, the set of $\P$-positions does not appear to constitute a finite disjoint union of individually regular patterns. See Figures~\ref{fig:chaos} and \ref{fig:chaos2}.
%%%%%%%%%%%%%%%%%%%%%%%%%%%%%%%%%%%%%%%%%%%%%
%%%%%%%%%%%%%%%%%%%%%%%%
\begin{rem}\label{rem:outcome geometry}
More intuition: An {\em outcome geometry} describes a larger picture of the outcomes, such as lines, or line segments that separate regions with distinguished {\em outcome patterns}, %(mosaics), 
or perhaps with no \P-positions \cite{friedman2019geometric}.  When we describe an outcome geometry there is no requirement to describe any outcome patterns along the lines or within the disjoint regions. The outcome patterns provide the usual information on how to win a game, but this is a secondary issue in terms of outcome geometry. An outcome segmentation is an example of an outcome geometry that we introduce here. %A solution of the outcome geometry indicates a current   outcome pattern. 
See also Subsection~\ref{sec:geom} for some background literature. 
\end{rem}

 % \begin{rem}\label{rem:NtoN}
 %     By Lemma~\ref{lem:PtoPd},   if $S= \{ \s_1, \s_2\}$, then $\x \in \N$  if and only if $\x+\s_1 +\s_2\in \N$. 
 % \end{rem}

\begin{rem}\label{rem: notation}
Regarding ruleset terminology, we sometimes omit words such as ``finite'' and ``vector'', and write, for example, {\sc subtraction} instead of {\sc finite vector subtraction}.  Similarly, we may omit ``subtraction'' in {\sc three-move subtraction} and simply write {\sc three-move}, and so on. The surrounding context would decide if we study one-dimensional games or some class of {\sc vector subtraction}. Instead of the family of rulesets {\sc two-dimensional vector subtraction} we may simply write {\sc two-dimension}. 
\end{rem}

%\subsection{Organization}
%\label{subsec:Organization}
The rest of the paper is organized as follows. Section~\ref{sec:literature} provides an overview of work in one dimension, except Subsection~\ref{sec:geom}, which discusses outcome geometry. 

Section~\ref{Sec:rowperiodicity} proves row/column eventual periodicity for generic {\sc two-dimensional finite subtraction}.

Section~\ref{sec:outcometable} introduces beginners to our subject via some {\sc one-dimension} examples; here, we also indicate some connections to {\sc two-dimension}. 

All other sections, except Section~\ref{sec:anydimension}, concern play in two dimensions. 

In Section~\ref{sec:2dresemble1d}, we study two-dimensional rulesets that inherit the outcome structures from one-dimensional ditto. In particular, we find some instances where adjoining a move preserves those inherited \P-position structures.

We solve {\sc two-move vector subtraction}:  Section~\ref{sec:twomove} focuses on two dimensions, and takes a geometric approach,  and Section~\ref{sec:anydimension} provides a \P/\N\ ~decision problem approach,  valid in any dimension. 

Section~\ref{sec:three_move_game}  generalizes ideas developed for {\sc two-move} and investigates their validity for {\sc three-move}. We find generalized \P-to-\P~ update rules and develop a theory about additive rulesets. 

This leads us to  Section~\ref{sec:outcomesegment}, where we introduce the concept of outcome segmentation via a defined coloring-scheme/automaton.
 By using these techniques, we solve some instances of {\sc three-move}. 

Finally, in Section~\ref{sec:picopen}, we display %some 
experimental results and pose open problems and conjectures regarding outcome segmentation, and the complexity of generic {\sc two-dimensional finite subtraction}.  
\section{Literature review}\label{sec:literature}
Together with {\sc nim} and {\sc wythoff nim}, {\sc finite subtraction} is probably one of the most popular combinatorial game rulesets, starting with the children game ``21'', where players are allowed to remove one or two beans from a heap of 21 beans. This normal-play game is a second player win, and is a brilliant educational tool for those who never encountered such games before. The beginner usually gets the offer to decide who plays first, and in practice what we see, often they mistakingly choose to start. Let us review some of the {\sc subtraction} literature.

\subsection{{\sc subtraction} mostly in one dimension}\label{sec:1dsubtr}

%%%%%%%%%%% Golomb paper%%%%%%%%%%%%%%%%

In the 1960s, Golomb  \cite{golomb1966mathematical} considers a variant of {\sc vector subtraction}, where a player wins by moving to the $\boldsymbol{0}$-vector, and establishes a draw-free condition in this setting.\footnote{We feel that it is more natural to adapt the convention ``who cannot move loses''. In this way, there cannot be draw games, and the theory develops more easily, as in \cite{larsson2012operator,larsson2013heaps}.} 
He explores the one-dimensional ruleset {\sc take-a-square}, i.e. $S=\{1,4,9,\ldots \}$, but does not provide any theoretical result on this ruleset: we quote, ``A detailed study of the properties of the losing set may well be as difficult as the study of the distribution of the prime numbers". He proves that if a ruleset $S$ contains arbitrarily large gaps then its set of \P-positions is infinite. Lastly, he defines a nonlinear shift register model which determines the winning and losing positions. This construction proves the ultimate periodicity of any finite one-dimensional subtraction game. 
Eppstein  \cite{eppstein2018faster}  continues the study of {\sc take-a-square} and provides an experimental justification that the set of \P-positions  has a ``density comparable to that of the densest known square-difference-free sets". He also studies complexity problems for {\sc subtraction}, and provides a fast algorithm for computing nim-values, provided a slow growth (as in {\sc take-a-square}). Both Golomb and Eppstein contemplate on why {\sc take-a-square} is so much more complex than the subtraction game derived from the Moser–de Bruijn sequence (all numbers with 0,1 in base 4 expansion). These rulesets have the same asymptotic density, but the \P-positions of  {\sc take-a-square} are more dense: $\sim n^{0.7}$ compared with $\sqrt{n}$ \P-positions for {\sc moser–de bruijn}, up to position $n$.

As mentioned, {\sc finite subtraction} on a single heap has ultimately periodic outcomes/nim-values. The method of proof gives an exponential bound on preperiod and period length in the largest element of the subtraction set. However, to our best knowledge,  exponential period lengths have not yet been rigorously established (see Flammencamp's thesis \cite{Flammenkamp_1997} for experimental results). Let us review some more recent developments in the literature on the quest of preperiod/period length of {\sc finite subtraction}. 
Even if we restrict our attention to {\sc three-move}, preperiod and period length are not yet fully understood.

In the early 1980s, Berlekamp et al. \cite[vol. 1, chap. 4, p. 83]{berlekamp2004winning} explore {\sc finite subtraction} and compute the nim-sequences for all rulesets with maximum subtraction $7$.\footnote{The nim-value of an impartial game can be computed recursively via a minimal exclusive algorithm \cite{siegel2013combinatorial} on the nim-values of its options. Any \P-position, such as the terminal positions, has nim-value $0$. Nim-sequences are a refinement of outcome sequences because they reveal the winning strategies in disjunctive sum play of impartial normal-play games. We will not require any of this material in this work, since we do not study game addition.} They note that the (ultimate) period length is the sum of two elements from the subtraction set, for all but one ruleset, namely $S=\{2,5,7\}$ (which has period length  $22$). In  \cite[vol. 3, chap. 5, p. 529]{berlekamp2004winning}, they examine {\sc additive subtraction} of the form $S = \{a, b, a+b\}$, and provide a complete statement for the nim-sequence whenever $b=2ha-r$, for some $0<r<a$, $h\in\n$. Specifically, the period length is $2b+r$. They highlight Ferguson's pairing property to derive the nim-value 1 from the 0s in the harder case when instead $b=2ha+r$ for some $0<r<a$, $h\in\n$. In this case, the period length is $a(2b+r)$. (They omit proofs of these results.) Ferguson's pairing property states that $\mathcal{G}(n)=1$ if and only if $\mathcal{G}(n-a)=0$, where $a$ represents the smallest number in the subtraction set.  
Austin \cite{austin1976impartial} analyses subtraction games in his Master's thesis. He demonstrates that a subtraction game has no preperiod if there exists a $p$ such that $p-s\in S$ whenever $s\in S$. In several instances, his results are similar to those in Winning Ways; he studies additive rulesets and also rulesets of the form $S=\{a,b, 2b-a\}$. 

Alth{\"o}fer  and B{\"u}ltermann  \cite{althofer1995superlinear} prove some results on the nim-values of {\sc three-move}  $S= \{a,2a+1, 3a+1\}$, $a \in \n$. This ruleset has a quadratic polynomial period length in the parameter $a$ and no preperiod.  They consider {\sc four-move} of the form $S=\{a,4a,12a+1,16a+1\}$ with $1\le a \le 26$, and prove pure periodicity; the period length is a cubic polynomial in the parameter $a$. They wonder if the {\sc five-move} ruleset   
$S=\{a,8a,30a+1,37a+1,38a+1\}$, $a \in \n$, has super-polynomial period length. 

Flammenkamp \cite{Flammenkamp_1997} expands on \cite{althofer1995superlinear}. He searches for long periods in {\sc finite subtraction} via extensive computations.  He shows that {\sc three-move} and other rulsets, given $\max S=s_3$ follow various fractal-looking behaviors on the classification of preperiods and periods, whether they are of the form $s_1+s_2$, $s_1+s_3$, $s_2+s_3$, or something else. He shows that very few have other period lengths, and that about half of the rulesets have period length $s_1+s_3$. Flammenkamp is not convinced by the mentioned {\sc five-move} problem from  \cite{althofer1995superlinear}. He suggests instead various {\sc four-move} and {\sc five-move} rulesets, where he, through extensive computations, provides heuristics that support ``exponential'' rather than ``polynomial'' eventual period lengths, in some cases via a ruleset parameter or in other cases via a ``record holder'' ruleset in terms of $\max S$. When it comes to {\sc four-move} he uses the second method to point at a tendency towards exponential vs. monomial behavior. He does this by testing the period lengths of record holders in the range $80\le \max S\le 235$, whether they can be lower bounded by an exponential expression of the form $2^{\alpha\max{S}}$ for some $ \alpha > 0$, or if they can be upper bounded by a monomial of the form $\max S^\beta$, for some $\beta>1$. As far as his computation goes, it seems that perhaps both are wrong.  Namely $\alpha\sim 0.2$, but with a decreasing tendency, whereas an upper bound $\beta$ seems even more unlikely, because, by indexing with $\max S$,  $\beta_{80}\approx 2.8$ while $\beta_{235}$ has increased to $\approx 5.2$. Motivated by the former, Flammenkamp makes more experiments. He completes a table with record holders for any finite ruleset with $\max S\le 30$, and finds a distinguishing property of members of record holder sets. They contain many elements of the form: both $\max S-s$ and $s$ belong to $S$.  Let us define this important property: $S$ is {\em max-symmetric} if, for all $1\le s\le \max S$, if $\max S-s \in S$, then $s\in S$. By restricting the experiments to max-symmetric sets, he finds that most record holders have size 5. By computing record holders among such sets, for all $61 \le \max S\le 117$, he finds that the period lengths tend to be $2^{\alpha\max S}$, for $\alpha\approx 0.3$. This certainly points towards the existence of subtraction sets with exponential period lengths.

Cairns and Ho \cite{cairns2010ultimately} introduce {\sc ultimately bipartite subtraction}. A ruleset is called bipartite if the ultimate period length is two. They first prove a necessary and sufficient condition for a ruleset to be bipartite. Secondly, they prove that if a ruleset is ultimately bipartite, then, for sufficiently large heap sizes, it is an \N-position if and only if the heap size is odd.  Ho  \cite{ho2015expansion} studies the periodicity of {\sc three-move subtraction} and discusses how to adjoin moves without changing the nim-value sequences.  
He poses a conjecture that ``The subtraction set of an ultimately bipartite game is non-expandable".

In  \cite{ward2016conjecture}, Ward considers {\sc three-move subtraction} $S=\{a,b,c\}$, where $a<b<c$, and conjecture a precise characterization of the nim-value periods. The conjecture is stated in two different regimes. When $S$ is additive, i.e. $c= a+b$, the period is at most a quadratic polynomial in  $a$ and $b$. This part is stated as a theorem in \cite{Flammenkamp_1997}, but concerning only outcomes, both \cite{austin1976impartial} and \cite{berlekamp2004winning} have at least partial solutions to this part. Both \cite{Flammenkamp_1997} and \cite{ward2016conjecture} conjecture that when $c \neq a+b$, the period length has seven possibilities; in essence the period length is a divisor of the sum of two elements of the subtraction set. 
Coleman et al.  \cite{coleman2020department} consider {\sc two-move subtraction}. They describe necessary and sufficient conditions for obtaining a given nim-value at a position. %Shenxing  
Zhang  \cite{zhang2021linearity} studies how the period of the nim-sequence changes when one adjoins a move to a given ruleset. The subtraction game literature and periodicity-related questions expand to partizan subtraction games (where the players have different subtraction sets) \cite{fraenkel1987partizan}. Duchene et al. \cite{duchene2022partizan} show that computing the preperiod length of the outcome sequence is NP-hard.  They compute the period length when each player's subtraction set is of size at most two.

As mentioned, finite subtraction games have eventually periodic nim-sequences. What can one say about infinite subtraction games? Of course, classical {\sc nim} is a simple example with an aperiodic and unbounded nim-sequence. In \cite{siegel2005finite} Siegel studies subtraction games where every removal but a finite set is allowed and proves the arithmetic periodicity of nim-values.\footnote{A sequence $(a_n)$ is arithmetic-periodic if there is a period $p>0$ and a saltus $s > 0$ such that for all $n > 0$ , $a_{n + p} = a_n + s$.}  Researchers have been interested in the question of whether there is any subtraction game with an aperiodic but bounded nim-sequence. Fox \cite{fox2014aperiodic} found a positive answer,  using a technique from combinatorics on words,  via a ruleset with an aperiodic nim-sequence bounded by 3, but the  description of the subtraction set and the proofs are complicated. Inspired by this, Fox and Larsson \cite{FoxLar} provides a simpler game and analysis with the same properties. They consider the subtraction set $S= \{ F_{2n+1} -1 \}_{n\in\n} =\{1,4,12, \cdots \}$ where $F_k$
is the $k$th Fibonacci number with $F_0=0$, $F_1 =1$.

\subsection{Other settings} 
Cohensius et al. \cite{cohensius2019cumulative}  solve instances of scoring-play  subtraction games, referred to as {\sc cumulative subtraction}. Two players alternatively remove stones from a pile that adds or removes points to a common score (depending on who played), and where any removal is restricted to a standard finite subtraction set $S$. They prove that the zero-sum outcomes in optimal play of such games are eventually periodic and the period is $2\max S$, and they find an exact characterization of the perfect play outcomes whenever $|S|=2$. So, the general problem of periodicity is a lot simpler for scoring-play than normal-play, while the instance with exactly two moves is somewhat more varied in scoring-play.

Larsson  \cite{larsson2012operator} considers   invariant subtraction games (i.e. {\sc vector subtraction}). %The game is defined by the set of  invariant moves $d$-tuples of  non-negative integers and a starting position. 
He defines the $\star$-operator of invariant subtraction games, and proves that any sequence of invariant subtraction games, generated by the  $\star\star$-operator, converges point-wise. The $\star$-operator defines the new subtraction game from the set of non-zero \P-positions of the old subtraction game. Thus, the $\star\star$-operator defines a game, by iterating $\star$ twice. He also provides a necessary and sufficient condition for duality, i.e. for $S= (S^{\star})^{\star}$ to hold. Note that the $\star$-operator can be applied to both finite and infinite subtraction sets. 

Carvalho et al. \cite{carvalho2020combinatorics} consider {\sc jenga}, a  popular recreational game. They compute nim-values of {\sc jenga} by viewing them as two-dimensional addition games, similar to the ones studied in \cite{larsson2013heaps}. For a study of the game values (a refinement of outcomes) of a finite two-dimensional partizan subtraction game, see  \cite{carvalho2012recursive}. We know of only one one-dimensional infinite partizan subtraction ruleset in the literature, namely  {\sc wythoff partizan subtraction} \cite{LarssonUrbanNeil}. Here, the players have complementary subtraction sets (via the complementary sequences of {\sc wythoff nim}'s set of \P-positions $S_L=\{1,3,4,6,\ldots\}$ and $S_R=\{2,5,7,10,\ldots \}$). The authors study the canonical form and reduced canonical form of a game  \cite{siegel2013combinatorial} and relate these concepts to various Fibonacci-type properties including methods from combinatorics on words; the authors achieve a complete solution of that game in terms of canonical form numbers and reduced canonical form switches. 

\subsection{Outcome geometry}\label{sec:geom}
Recent research on {\sc two-dimensional subtraction} has exposed a fair amount of geometry on the outcomes. This is certainly true for a {\sc generalized diagonal wythoff nim} \cite{Larsson2012GDWN,Larsson2014split} and in more generality (not necessarily symmetric rulesets) for the family {\sc linear nimhoff} \cite{friedman2019geometric}. The idea of those papers is to adjoin not single moves, but entire `lines' of moves to the original move-lines of {\sc nim} or {\sc wythoff nim}. The surprise is that the \P-positions appear to reorganize along various combinations of diffuse lines, new and/or old, with with distinct slopes and asymptotic densities as projected on the x-axis (e.g. {\sc nim} has one \P-line of slope 1 and density 1; {\sc wythoff nim} has two \P-lines of slopes $\phi=\frac{\sqrt{5}+1}{2}$ and $\phi^{-1}$ with densities $\phi^{-1}$ and $\phi^{-2}$ respectively). The regions (segments) between those \P-lines are visually, i.e. experimentally, void of \P-positions.  By making such assumptions, and by using a method from physics, called {\em renormalization}, a system of equations is formulated on the respective slopes and densities of the conjectured lines. The renormalization approach gives very good estimates of the conjectured geometry of games, as discussed in \cite{friedman2019geometric}. They also discuss when new move-lines do not affect the slopes and densities of existing \P-lines. That work has been an inspiration to this work, although the method used is different. 

Larsson et al. \cite{larsson2014maharaja}  introduce {\sc  maharaja nim}, which combines the moves of the Queen and Knight of Chess.  
They prove that for  {\sc  maharaja nim} the \P-positions remain close to those of {\sc wythoff nim}; they reorganize  within a bounded distance of the lines with slopes  $\phi$ or  $\phi^{-1}$. More generally they wonder if this global outcome geometry remains the same when any finite number of subtractions is adjoined to (and/or removed from) those of {\sc wythoff nim}. In \cite{duchene2009another} they prove that if one removes a single move from those of {\sc Wythoff Nim} then the \P-positions cannot remain exactly the same. But they did not further the research question into the direction of local vs. global behavior. They did not  phrase questions about {\em robustness/fragility} of the outcome geometry of a ruleset; here we borrow terminology from another recent work, closely related to combinatorial game theory~\cite{yakovurban}.

Larsson \cite{larsson2011blocking} introduces {\sc blocking wythoff nim}; for a given parameter $k$, at each stage of play, at most $k-1$ options may be blocked by the opponent; after the move any blocking is forgotten. He finds the winning strategies when $k=2,3$. Cook et al. \cite{Cook2017} continue by exploring the geometry of this game, and they find a cellular automaton that emulates a certain blocking-refinement of the outcomes. They find a surprising amount of self-organization in this system; the global geometry of the refined outcomes remains the same for all large blocking parameters. The system ``self-organizes itself into 11 visually distinguishable regions with 14 borders (half-lines or line-segments) and 6 junctions''.

%%%%%%%%%%%%%%%%% New Literature ends  here%%%%%%%%%%
\section{Eventual row periodicity}
\label{Sec:rowperiodicity}
The most basic and most general result for {\sc one-dimensional finite subtraction} concerns eventual periodicity of the outcomes. In two dimensions, we are able to prove a similar result, but only row- and  column-wise. 

\begin{defi}[Row Periodicity]
Consider a ruleset $S$ in  {\sc two-dimensional subtraction}. 
\begin{enumerate}[a)]
\item Consider row $j\in \nz$.  The outcomes are row $j$ periodic  if  $\exists$ $T > 0$ such that  $\forall i \ge 0$, $ o(i,j)= o(i+T,j)$. 
 \item  The outcomes are row periodic if  $\forall j \ge 0$  the outcomes are row $j$ periodic.
\end{enumerate}
\end{defi}
In practice, row periodicity is very rare.

\begin{defi}[Row Eventual Periodicity]
\label{def_eventualperiodic}
Consider a ruleset $S$ in {\sc two-dimensional subtraction}.
\begin{enumerate}[a)]
    \item  Consider row $j\in \nz$. The outcomes are row $j$ eventually periodic if there exists a column  $i^{'}\in \nz $  and a $T >0$ such that $\forall i \ge i'$
    $
    o(i,j) = o(i+T, j).
    $  
    \item The smallest such $i'\in \nz $ is the preperiod length of row $j$.
    \item The outcomes are row eventually  periodic if  $\forall j \ge 0$ the outcomes are row $j$  eventually periodic.
\end{enumerate}
\end{defi}

\begin{thm}[Row-Column Eventual Periodicity]
\label{thm:rowperiod2d}
Consider {\sc two-dimensional finite subtraction}. The outcomes are row (column) eventually periodic.
\end{thm}
\begin{proof}
Let $S=\{(a_i,b_i)\}_{1\le i \le k}$ be the subtraction set, and let $A=\max_{i} \{a_i\}$.

Fix a row $y$, and define the window $W(x,y) = (o(x,y), o(x+1,y), \cdots , o(x+A-1,y))$. The outcome function $o(x,y) \in \{\N, \P\}$. Hence there are at most $2^A$ possible outcome patterns within the window range. 

We prove the theorem using induction.  

Base case: the outcomes of row $0$  are eventually row periodic. We first prove this statement. 

Since $2^A$ is a finite number, there exists a smallest $x$ and an $\alpha$ such that $W(x+\alpha,0) = W(x,0)$. 
That is, the outcome patterns are repeated in the window. By computing the next outcome, because of the choice of window size, we find that $W(x+i+\alpha,0) = W(x+i,0)$, for all $i\ge 0$. Hence, the repetition occurs before $A 2^A$, and the period length $p_0$ is at most $A2^A$.

Hypothesis: Fix a row $y'$ and assume that the outcomes are eventually row periodic, for all rows $0\le y < y'$.
Let the period lengths for the rows $0, 1, \cdots, y'-1$ be  $p_0, p_1, \cdots, p_{y'-1}$, respectively.  We aim to prove that the $y'$-th row is eventually row periodic. 
%\begin{defi}
%\label{def:superperiod}

Let $B=\max_{i} \{b_i\}$ be the maximum of the $y$-coordinates of the subtraction set. Define the superperiod length 
\begin{equation}
\label{eq:superperiod}
\pi_{y'} = \prod_{i=y'-B}^{y'-1}p_i.
\end{equation}
%\end{defi}
For any sufficiently large $x$, consider the windows $W(x+m\pi_{y'},y')$, where $m \in \nz$. Since the number of possible outcome patterns within these windows is finite, there will be a repetition, for say $m'$ and $m''$ respectively. When this repetition occurs, we use that $x$ is sufficiently large, meaning that the lower rows have become periodic. It follows that row $y'$ is eventually periodic, because $o(x+m'\pi_{y'}+1,y')=o(x+m''\pi_{y'}+1,y')$ (the argument is the same as in the base case). 
\end{proof}

Theorem~\ref{thm:rowperiod2d} is proved for impartial rulesets. But, similar to {\sc one-dimension}, the proof technique is general, and it continues to hold for nim-values, partizan rulesets, and more. 

In going from full generality to particular ruleset analysis, it is worthwhile reviewing some elementary  examples of {\sc one-dimensional  finite subtraction}; sometimes this bears fruit in {\sc two-dimensional  finite subtraction}. This is the content of the next two sections.  
\section{Does it matter if  $b> 2a$?} %\texorpdfstring{$s_2>2s_1?$}?}
\label{sec:outcometable}
This section concerns elementary examples of {\sc one-dimension}, and could be skipped by readers well aquainted with the subject. Here we compute the initial outcomes for some examples of {\sc two-move} $S=\{s_1,s_2\}=\{a,b\}$ and {\sc additive three-move} $S=\{a,b,a+b\}$, $a< b$. Whenever possible, we relate the outcome patterns to Euclidean division, with $b=qa+r, 0\le r<a$. Moreover, we will point at some connections to {\sc two-dimension} studied further down this paper. 
  
We study the regime $b \le 2a$ via the rulesets $S_2= \{5,8\}$ and $S_4=\{2,3,5\}$  and the regime $b> 2a$ via the rulesets  $S_3=\{3, 8\}$ and $S_5=\{2,5,7\}$. Moreover we let $S_1=\{3\}$, and we denote the outcome functions corresponding to rulesets $S_i$ as $o_i$ where $i\in \{1,2,3,4,5\}$.

When $x< a=3$ there are only $\P$-positions for the ruleset $S_1$, since there is no move to play. For $a \le x  < 6=2a$ there are only $\N$-positions, since in this case $o_1(x-3)=\P$. The rest of the positions repeat this pattern since those three $\N$-positions act as if there was no move at all. Let us list the first few outcomes:
\vspace{2 mm}  
\begin{center}
{\small
\begin{tabular}{|c|c|c|c|c|c|c|c|c|c|c|c|c|c|c|c|c|c|c|}
\hline
 $x$ & 0 & 1 & 2 & 3 & 4 & 5 & 6 & 7 & 8 & 9 & 10 & 11 & 12 &13 &14 & 15 & 16& 17\\ \hline
$o_1$ & \P  & \P & \P & \N & \N & \N & \P & \P & \P & \N & \N  & \N & \P & \P &\P& \N &\N &\N \\ \hline
\end{tabular}}
\end{center}\vspace{2 mm}
%Consider the ruleset  $S_2$.
In general, the outcomes of {\sc one-move} correspond to {\sc two-move} of the form $S=\{a,qa\}$ for an odd $q$ (here $r=0$).

Further, in the case of $S_2$, $o_2(x) = \P$ if $x < a=5$ and $o_2(x) =\N $ if $5=a \le x < a+b=13$, and, due to Lemma~\ref{lem:PtoP},  the rest of the positions repeat this pattern.\vspace{2 mm} 
{\small
\begin{center}
\begin{tabular}{|c|c|c|c|c|c|c|c|c|c|c|c|c|c|c|c|c|c|c|}
\hline
 $x$ & 0 & 1 & 2 & 3 & 4 & 5 & 6 & 7 & 8 & 9 & 10 & 11 & 12 &13 &14 & 15 & 16 & 17\\ \hline
$o_2$ & \P  & \P & \P & \P & \P & \N & \N & \N & \N & \N & \N  & \N & \N & \P &\P& \P & \P &\P\\ \hline
\end{tabular}
\end{center}}\vspace{2 mm}

In this case, $8 = b\le 2a=10$, and the outcomes of $S_2$ resemble those of $S_1$; first, a bunch of \P-positions, then a bunch of \N-positions, and so on, where the sizes, $a$ and $b$ respectively of those respective bunches remain forever the same. Although this is a visually obvious observation, in this paper we take an approach that generalizes better: the outcome patterns of {\sc one-move} reigns until {\sc two-move} enters the battlefield.  

The second regime is here exampled by the ruleset $S_3$,  $o_3(x) =\P$ if $x < a=3$ , $o_3(x) =\N $ if $3 \le x <6$,  $o_3(x) =\P $ if $6 \le x <8$ and  $o_3(x) =\N $ if $8 \le x <11$. The rest of the positions repeat this pattern due to  Lemma~\ref{lem:PtoP}.\vspace{2 mm} 
{\small
\begin{center}
\begin{tabular}{|c|c|c|c|c|c|c|c|c|c|c|c|c|c|c|c|c|c|c|c|}
\hline
 $x$ & 0 & 1 & 2 & 3 & 4 & 5 & 6 & 7 & 8 & 9 &10 &11 &12 &13 &14 & 15 
 & 16 &17 \\ \hline
$o_3$ & \P  & \P & \P & \N & \N & \N & \P & \P & \N & \N & \N & \P  & \P & \P & \N & \N & \N  & \P \\ \hline
\end{tabular}
\end{center}}\vspace{2 mm} 
In this case, $6 = 2a< b = 8$, and the outcome pattern is somewhat more interesting and varied depending on the result of the Euclidean division of $b$ by $a$. In particular, the fact that $2a=6\in\P$ has various implications for {\sc two-dimension}. 

In general, for {\sc two-move}, if $x<b$ then $x\in \P$ if and only if $0\le x<a \pmod{2a}$, if $b\le x<b+a$ then $x\in \N$, and then apply Lemma~\ref{lem:PtoP} (see also Theorem~\ref{thm:twomoveonedim}). 

We study {\sc additive three-move} in one dimension and find a particular example when Lemma~\ref{lem:PtoP} continues to hold. By this, we mean: the period length is the sum of two moves from the set $S$. 
%The ruleset  is an example of  {\sc additive three-move}. 
Let us compute a few initial outcomes of $S_4$. 
  \vspace{2 mm}
  
{\small
\begin{center}
\begin{tabular}{|c|c|c|c|c|c|c|c|c|c|c|c|c|c|c|c|c|c|c|}
\hline
$x$ & 0 & 1 & 2 & 3 & 4 & 5 & 6 & 7 & 8 & 9 & 10 & 11 & 12 &13 &14 & 15 & 16 & 17\\ \hline
$o_4$ & \P  & \P & \N & \N & \N & \N & \N & \P & \P & \N & \N  & \N & \N & \N &\P & \P & \N &\N \\ \hline
\end{tabular}
\end{center}}\vspace{2 mm}
 In this case, $4=2a \ge b=3$ and the outcome of $S_4$ resembles those of $S_2$; first, a bunch of \P-positions, then a bunch of \N-positions,  and so on, where the
sizes, $a$ and $a+b$ respectively of those respective bunches remain forever the same. 
 The outcome function  $o_4(x)= \P$ if  $0 \le x  < a=2$ and $o_4(x)= \N$ if  $2 \le x  <7=2a+b$ and the rest of the positions repeat this pattern.

\begin{thm}%[{\sc additive three-move one-dimension}]
\label{thm:threemoveonedim}
Suppose $b/2 \le a<b$. Then the outcomes of the ruleset $S=\{a,b, a+b\}$ are purely periodic, with period length $2a+b$; $o(x)=\P$ if $0\le x<a$, and $o(x)=\N$ if $a\le x < 2a+b$. 
\end{thm}
\begin{proof}
The outcome $o(x)=\P$ if $0\le x<a$, since there is no move to play. Now if  $a \le x < b$, because $b\le 2a$,  using the move $a$ we have $o(x-a)=\P$. Therefore, $o(x)= \N$ if $a \le x < b$. Similarly, for $b \le x < a+b$,  we find $o(x-b)= \P$ and for $a+b \le x < 2a+b$,  we find that $o(x-a-b)=\P$. Therefore, $o(x)= \N$ if $b \le x < 2a+b$. Altogether, we have  $o(x)=\N$ if $a\le x < 2a+b$. Since the number of consecutive \N-positions is the same as the largest move, this proves the statement.
\end{proof}
As a consequence, the \P-to-\P\ property from {\sc two-move} continues to hold in this particular case. 
\begin{cor}%[{\sc  additive three-move}  \P-to-\P]
\label{cor:threemove}
    Consider $S=\{a,b,a+b\}$ with  $b/2 \le a < b$. Then,
%\begin{enumerate}[a)]
%\item 
the position $x \in \P$ if and only if $x+2a+b \in \P$.
%\item the position $x \in \N$ if and only if $x+2a+b \in \N$.
%\end{enumerate}
\end{cor}
\begin{proof}
  This is immediate by  Theorem~\ref{thm:threemoveonedim}. %the proof follows.  
\end{proof}

The ruleset $S_5$ is another example of  {\sc additive three-move}.
Let us list the first few outcomes: 
\vspace{2 mm} 

\begin{center}
\begin{tabular}{|l|l|l|l|l|l|l|l|l|l|l|l|l|l|l|l|}
\hline
$x$    & 0  & 1  & 2  & 3  & 4  & 5  & 6  & 7  & 8  & 9  & 10 & 11 & 12 & 13 & 14 \\ \hline
$o(x)$ & $\P$  & $\P$  & $\N$  & $\N$  & $\P$  & $\N$  & $\N$ & $\N$  & $\N$  & $\N$  & $\P$  & $\N$  & $\N$  & $\P$  & $\P$  \\ \hline
$x$    & 15 & 16 & 17 & 18 & 19 & 20 & 21 & 22 & 23 & 24 & 25 & 26 & 27 & 28 & 29 \\ \hline
$o(x)$ & $\N$  & $\N$  & $\N$  & $\N$  & $\N$  & $\N$  & $\N$  & $\P$  & $\P$  & $\N$  & $\N$  & $\P$  & $\N$  & $\N$  & $\N$  \\ \hline
\end{tabular}
\end{center}
In this case, $4 = 2a< b = 5$, and the outcome pattern is more complex than in the regime $2a \ge b$; the outcomes are purely periodic, but the period length sums up to 22. In particular, every  \P-to-\P\ rule fails to hold. These elementary observations have implications for our study of play in more than one dimension.

In Section~\ref{sec:2dresemble1d}, Theorem~\ref{thm:twomoveonedim} finds an instance when {\sc two-dimension} resembles {\sc one-dimension} if and only if $b\le 2a$, namely if $S=\{(a,a),(b,b),(a,b)\}$. In Section~\ref{sec:twomove}, which concerns {\sc two-move two-dimension} of the form $S=\{(a,b),(c,d)\}$, if we let $c \leq 2a$ and $d \leq 2b$, the outcome patterns are straightforward and closely resemble those for {\sc one-move} (Figure~\ref{fig:Lshape}). In  Section~\ref{sec:outcomesegment} we study {\sc symmetric additive three-move}, that is rulesets of the form $\{(a,b),(b,a),(a+b,a+b)\}$. And in Proposition~\ref{prop:symadd} we find solutions whenever $b\le 2a$. See also Figure~\ref{fig:additive}. On the other hand,  Figure~\ref{fig:1551add2552} points toward more complexity whenever $b>2a$. Similarly, we have observed regular behavior for {\sc asymmetric additive}, a ruleset of the form $S=\{(a,b),(b,a+b), (a,a+b)\}$ in a case when $b\le2a$; see Proposition~\ref{prop:OS}. 

\section{Two-dimensional rulesets that resemble one-dimensional  ones}
\label{sec:2dresemble1d}
This section is a gentle technical introduction to our subject in that it takes a small non-trivial step, while using a classical {\sc one-dimension} result, going to {\sc two-dimension}. 

The outcomes of {\sc one-move} $S=\{a\}$ satisfy $x\in\P$ if and only if 
    \begin{align}\label{eq:onemoveonedim}
    x\in \{0,\ldots , a-1\} \pmod{2a}.
    \end{align}

Let us recap the well-known {\sc two-move one-dimension} result, as arrived at in the previous section. We will state it in a way that we will generalize later. 
For {\sc one-dimension two-move}, say $S=\{a,b\}$, the period length of outcomes is at most $a+b$, and there is no pre-period.  

We are not aware of the first appearance of a proof. For a modern approach, see \cite{coleman2020department}. 
We write $\x\in\P(S)$ or $\x\in\N(S)$ for the outcome of position $\x$, with ruleset $S$. Consider two rulesets $S_1$ and $S_2$. We write $o_i(\x)\in \{\P,\N\}$ for the outcome of $\x$ with rules as in $S_i$, $i\in\{1,2\}$. 

\begin{thm}[{\sc two-move one-dimension}]\label{thm:twomoveonedim}
Suppose $a<b$, with $S_1=\{a\}$ and $S_2=\{a,b\}$. Then the outcomes of the ruleset $S_2$ are purely periodic, with period length $a+b$, or $2a$ in case $2a\mid a+b$. Namely,  $o_1(x)=o_2(x)$ if $0\le x<b$, and $o_2(x)=\N$ if $b\le x<a+b$. 
\end{thm}
\begin{proof}
Base case: obviously $o_1(x)=o_2(x)$ if $0\le x<b$. And $o_2(x)=\N$ if $b\le x<a+b$, since in this case $o_2(x-b)=\P$. Now apply Lemma~\ref{lem:PtoP} with $d=1$.
\end{proof}
{\sc two-move} that resemble {\sc one-move} satisfies the condition $2a\ge b$, with notation as in Theorem~\ref{thm:twomoveonedim}. 

    \begin{cor}\label{cor:onedim}
    Consider $S = \{a,b\}$, with $a<b$. The outcomes are purely periodic of the form ``$x\in\P$ if and only if $x\in \{0,\ldots , a-1\} \pmod{a+b}$'', if and only if $ b\le 2a$. The position $2a\in \P$ if and only if $2a<b$.
     
    \end{cor}
    \begin{proof}
        Combine equation~ \eqref{eq:onemoveonedim} with Theorem~\ref{thm:twomoveonedim}.
    \end{proof}

Corollary~\ref{cor:onedim} will be useful to analyze some rulesets in {\sc two-dimension}.
\begin{defi}[Symmetric Ruleset]
\label{def:sym}
A ruleset $S\in \B^2$ is symmetric if, for all $(a,b)\in S$, $(b,a)\in S$. 
\end{defi}

\begin{defi}[Twin Ruleset]
\label{def:twin}
A ruleset $S\in\B^2$ is twin if, for all $(a,b)\in S$, $a=b$. 
%A sibling ruleset $S\in\B^2$ satisfies, for given $k,\ell\in \n$ for all $(a,b)\in S$, $ka=\ell b$.
\end{defi}
For example $S=\{(1,2),(2,1)\}$ is a symmetric ruleset and $S=\{(1,1),(2,2)\}$ is a twin ruleset. So, twin rulesets are trivially symmetric. And twin rulesets are derivations of one-dimensional games in the following sense. 
\begin{defi}[Ruleset Derivation]
\label{def: Ruleset Derivation}
Consider a pair of game boards $\B^1$ and $\B^2$ with corresponding dimensions $d_1<d_2$. Let $S_1=S_1(\B^1)$ and let $S_2=S_2(\B^2)$ be two rulesets. Then $S_2$ is an $f$-derivation of $S_1$ if there is a linear  surjective function $f:\B^2\rightarrow \B^1$ such that, for all positions $ \boldsymbol x\in \B^2$, $o_1(f(\boldsymbol x))=o_2(\boldsymbol x)$. 
\end{defi}

For example let $d_1=1$ and $d_2=2$,  with $S_2=\{(s,0)\mid s\in S_1\}$. Then $S_2$ is an $f$-derivation of $S_1$ for $f((x,y))=x$. As seen by this example, ruleset derivation can be much stronger than an outcome correspondence. In this example, the game trees are multiplied for the larger ruleset, and there is no interaction between them. But for our purpose, it suffices to interpret derivation as `outcome-derivation'. Indeed, this gives an opportunity to explore when adjoining moves to the second ruleset does not interfere with the outcomes of the first ruleset. It is also the correct notion in the context of {\em mechanism design} of combinatorial games. When does a given pattern of candidate outcomes in {\sc two-dimension} correspond to a {\sc one-dimension} derivation? When can it be modeled by a ruleset in {\sc two-dimension}? Which patterns are not possible to model as {\sc subtraction} outcomes? (See \cite{larsson2013impartial,MR4188748}.)

Figure~\ref{fig:1-derivation} is an example of a 2-segmentation that is a derivation from {\sc one-dimension}.  The following result generalizes this idea.
\begin{figure}[htbp!]
 \begin{center}
 \includegraphics[width=8cm]{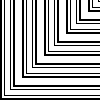}
 \end{center}
 \caption{%A symmetric 2-segmentation, which is derived from a one-dimensional game: 
 These are the  initial outcomes of {\sc three-move} $S =\{(2,2),(5,5),(7,7)\}$.}
 \label{fig:1-derivation}
 \end{figure}
\begin{lem}[Twin Ruleset Derivation]\label{lem:twinsetder}
    Consider $S_1$ in {\sc one-dimension}, and let $S_2=\{(s,s)\mid s\in S_1\}$. For all $(x,y)\in\B^2$, let 
\begin{equation*}
\label{Eqn_twin_Set}
	f((x,y))=
	\begin{cases}
		x, &\text{if }\  x\le y\\
		y , & \text{otherwise}.
	\end{cases}
\end{equation*}
 Then $S_2$ is an $f$-derivation of $S_1$.
\end{lem}
\begin{proof} %Suppose $o_1(x) = \P$. Then, for all $s\in S_1$, $o_1(x-s)= \N$. 
Suppose $x\le y$. For a base case take $x=0$. Observe that, for all $y\ge 0$, $o_2(0,y)=\P$, because $(0,y)$ is terminal since $S_2$ is twin. Fix an $ x' > 0$ and assume that for all $x< x'$,    $o_1(x) = \P$ if and only if for all $y\ge x$, $o_2(x,y)=\P$.\vspace{2 mm}

\noindent \textbf{Case 1:}  $x' \in \P$. We have to show that, for all $y \ge x'$, $o_2(x',y)=\P$. 
 Since  $x' \in \P$  then, for all $s \in S_1$, $o_1(x'-s)=\N$. By the induction hypothesis, for all $y\ge x'$, for all $(s,s) \in S_2$,  $o_2(x'-s,y-s)=\N$  (because $y-s \ge x'-s$). Hence, for all $y\ge x'$, $o_2(x',y)=\P$.\vspace{2 mm}

\noindent \textbf{Case 2:}  $x' \in \N$. We have to show that, for all $y\ge x' $, $o_2(x',y)=\N$. Since  $x' \in \N$  then there exists an  $s \in S_1$ such that $o_1(x'-s)=\P$.
By the induction hypothesis,  for any $y \ge x'$, $o_2(x'-s,y-s)=\P$  (because $y-s \ge x'-s$).  Hence  for all $y$, $o_2(x',y)=\N$.\vspace{2 mm}

The proof is analogous  for $x > y$.
\end{proof}

Sometimes (symmetric) expansions of twin rulesets remain derivations of {\sc one-dimension}.   
\begin{thm}[Symmetric Ruleset Expansion]
\label{thm:twinadjoin}
    For fixed $a,b\in \n$, with $a<b$ let $S=\{(a,a),(b,b)\}$. Then $\P(S)=\P(S\cup\{(a,b)\})=\P(S\cup\{(a,b),(b,a)\})$ if and only if $ b/2\le a$. 
\end{thm}
\begin{proof}
Let $S_1=\{(a,a),(b,b)\}$ and let $S_2= S_1 \cup \{(a,b)\}$. We begin by proving  $\P(S_1)=\P(S_2)$ in the case of $b/2 \le a < b $. The base case  is  $(0,0) \in \P(S_1) \cap \P(S_2)$. Fix $(x', y')$ and 
assume that, for  all positions $(x ,y)<(x', y')$,  $o_1(x , y)=  o_2(x , y)$.\vspace{2 mm}

\noindent \textbf{Case 1:} $(x', y') \in \P(S_1)$. We have to show that $ ( x', y')\in \P(S_2)$. That holds if $(x'-a, y'-b)\in \N(S_2)$. By the induction assumption, it suffices to prove that $(x'-a, y'-b)\in \N(S_1)$, which holds by using Lemma~\ref{lem:twinsetder}, $f$-derivation and the first part of Corollary~\ref{cor:onedim}, since $b\le 2a$. Namely, in case $y'\ge x'$, these results imply that $x'\in \{0,\ldots , a-1\} \pmod{a+b}$. And so $x'-a\not\in \{0,\ldots , a-1\} \pmod{a+b}$. Thus, if $y'-b\ge x'-a$, the $f$-derivation proves that $(x'-a,y'-b)\in \N$. Suppose therefore that $y'\ge x'$ with $x'\in \{0,\ldots , a-1\} \pmod{a+b}$, but $y'-b < x'-a$. We must prove that $y'-b\not\in \{0,\ldots , a-1\} \pmod{a+b}$. We get $x'-b\le y'-b < x'-a$. This means that $y'-b\in \{-b,\ldots , -1\}\equiv \{a,\ldots , a+b-1\} \pmod{a+b}$.\vspace{2 mm}

 \noindent \textbf{Case 2:} If $(x', y') \in \N(S_1)$,   we must prove that $ ( x', y')\in \N(S_2)$. That holds since either $(x'-a, y'-a)\in \P(S_2)$ or $(x'-b, y'-b)\in \P(S_2)$, by the induction assumption. \vspace{2 mm}

The equality $\P(S\cup\{(a,b)\})=\P(S\cup\{(a,b),(b,a)\})$ follows by an analogous argument.

Next, suppose that $b/2 > a$. Consider $\x=(2a,b)$. Then $\x\in\P(S_1)$, by combining the second part of Corollary~\ref{cor:onedim} with the $f$-derivation in Lemma~\ref{lem:twinsetder}. But $\x\in\N(S_2)$, since $\x-(a,b)\in\P(S_2)$. 
\end{proof}

\section{Two-move subtraction in two dimensions.}\label{sec:twomove}
As we saw  Figure~\ref{fig:dilbox}
%on the second page of this paper, 
already {\sc two-move} contains interesting patterns. However ``Diluted Boxes'' is by no means the unique pattern formation for {\sc two-move}. We will show a few more pictures in this section to guide us through the outcomes of {\sc two-move} via a geometric approach.  On the other hand, Section ~\ref{sec:anydimension}, which analyses {\sc two move} in any dimension,  focuses more on the \P/\N-decision problem and its complexity. Here we develop geometrical ideas through Figure~\ref{fig:cuttingprinciple}, and we will obtain detailed results without indulging in a tedious case-by-case analysis.
We begin by recalling the trivial she-loves-me-she-loves-me-not outcome structures of {\sc one-move}. 

\begin{lem}[{\sc one-move}]\label{lem:onemove}
 Let $a$ and $b$ be non-negative integers such that at least one is positive. Consider the subtraction game $S=\{(a,b)\}$. Then,
\begin{enumerate}
    \item [a)]  the position $(x,y)\in \P$ if and only if $(x+a, y+b)\in \N$;
  
    \item [b)]  the position $(x,y)\in \P$ if and only if $(x+2a, y+2b)\in \P$;
   
    \item [c)]  the position $(x,y)\in \N$ if and only if  $(x+2a, y+2b)\in \N$.
\end{enumerate}
\end{lem}

\begin{proof} 
Omitted. 
 \end{proof}

The following result will be used to derive the outcome patterns of  {\sc two-move}. For $a, b \in \nz $,  $a+b>0$, $n\in \nz $, let  the $n$th L-shape be $$L_n =   \left\{(x + 2na,y)\mid    0 \le x< a,  2nb\le y   \right\} \cup \left\{(x, y + 2nb)\mid    0 \le y< b , 2na\le x  \right\}. $$

\begin{lem}[{\sc one-move} Structure]\label{lem:onemovestruct} 
  Consider any {\sc one-move} ruleset $S= \{(a,b)\}$. 
%\begin{enumerate}[a)]
    % \item If both $a,b>0$, then 
    Then $(x,y)\in \P$ if and only if  $(x,y) \in \bigcup_{n\in \nz}  L_n$. 
\end{lem}
 \begin{proof}
     Note the terminal \P-positions, $L_0 =   \{(x,y)\mid    0 \le x< a,  0\le y   \} \bigcup \{(x, y )\mid    0 \le y< b , 0\le x  \}$, and apply Lemma~\ref{lem:onemove}.
 \end{proof}
 \begin{rem}
The L-shapes $L_n$ for any $n \in \nz$ collapse to parallel strips, horizontal or vertical, whenever $a=0$ or $b=0$ respectively. See Figures~\ref{fig:horizontal} and \ref{fig:Lshape}. 
 \end{rem}

The following result is proved in more generality in Section~\ref{sec:intro}  (Lemma~\ref{lem:PtoP}).
%and so we omit a proof here.  
 \begin{lem*}[$\P$-to-$\P$ in two dimensions]
 %\label{lem:PTOP2D}
Let $S= \{ (a,b), (c,d)\}$.  Then $(x,y) \in \P$  if and only if   $(x+a+c,  y+b+d)\in  \P$.

\end{lem*}

Due to Lemma~\ref{lem:PtoP} we define a ruleset dependent translation. 
%\In{Lemma ~\ref{lem:PtoP} is not itemized form.}

 \begin{defi}[Translation Set]\label{def:PtoP} 
 Let $S=\{(a,b),(c,d)\}$ and let $X\subset \B^2$. Then $T(X)=T_S(X):=\{(x+a+c,y+b+d) \mid (x,y)\in X\}$.
 \end{defi}
 We may iterate this set translation.

\begin{obs}\label{obs:T}
Note that, by Lemma~\ref{lem:PtoP}, for all $n\in\nz$, $X\subset\P$, if and only if $T^n(X)=\{(x+n(a+c), y+n(b+d)) \mid (x,y)\in X\}\subset \P$.
\end{obs}

 \tikzset{decorate sep/.style 2 args={decorate,decoration={shape backgrounds,shape=circle,shape size=#1,shape sep=#2}}}

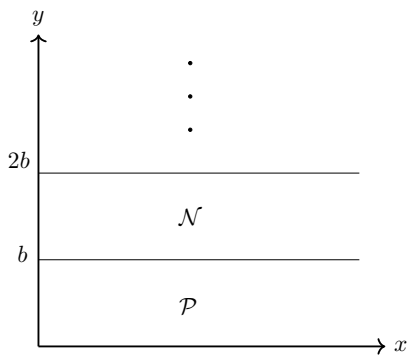
\begin{figure}[htbp!]
\begin{center}
%\hspace{-3.2 cm}
\resizebox{5.3 cm}{!}{
\begin{tikzpicture}[scale=1.3]
\draw[->][thick] (0,0)--(4,0) node[right]{$x$};
\draw[->][thick] (0,0)--(0,3.6) node[above]{$y$};
\put (-9,36) {$b$};
\put (-13,76) {$2b$};
\put (60,14) {$\P$};
\put (60,52) {$\N$};
\draw (0,1)--(3.7,1);
\draw (0,2)--(3.7,2);
%\draw[dotted] (1.6,2.1)--(1.6,2.3);
\draw[decorate sep={0.5mm}{5mm},fill] (1.75,2.5)--(1.75,3.3);
\end{tikzpicture}}
\end{center}
%\vspace{-3cm}
\caption{The subtraction set $S=\{(0, b)\}$ generates parallel horizontal  strips of width $b$.
%\ur{The arrows need to be outside the other lines.} \In{I changed now.}
}
\label{fig:horizontal}
\end{figure}
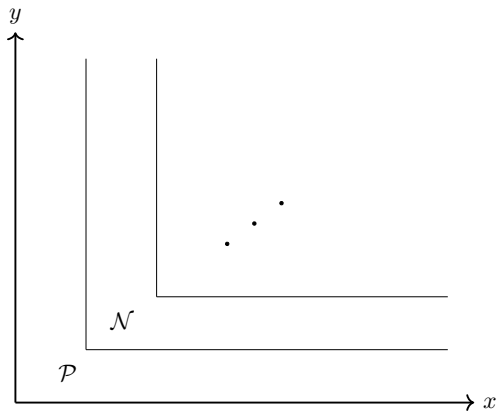
\begin{figure}[htbp!]
\begin{center}
%\hspace{-3.5 cm}
\resizebox{6.8 cm}{!}{
\begin{tikzpicture}[scale=1.3]
%\draw[dashed] (-3.0,-3.0);
\draw[thick, ->] (0,0)--(5.2,0) node[right]{$x$};
\draw[thick, ->] (0,0)--(0,4.2) node[above]{$y$};
\draw (0.8,0.6)--(0.8,3.9);
\draw (0.8,0.6)--(4.9,0.6);
\put (18, 9) {$\P$};
\draw (1.6,1.2)--(1.6,3.9);
\draw (1.6,1.2)--(4.9,1.2);
\put (40,31) {$\N$};
%\draw[dotted] (1.3,1.2)--(1.5,1.4);
\draw[decorate sep={0.5mm}{5mm},fill]  (2.4,1.8)--(3.2,2.4); %(1.35, 1.25);
\end{tikzpicture}}
\end{center}
%\vspace{-3cm}
\caption{The subtraction set $S=\{(a, b)\}$, $a,b>0$, generates $\mathrm L$-shaped outcome sets. The first $b$ rows and $a$ columns are all $\P$-positions. The following shifted $b$ rows and $a$ columns are all $\N$-positions. The next are all \P-position, and so on.
}\label{fig:Lshape}
\end{figure}

The \P-positions of {\sc two-move} can be learned from {\sc one-move} in a linear fashion. By symmetry it suffices to study, say the upper `half' of the game board, separated by a line of slope $\frac{b+d}{a+c}$.

\begin{defi}[Slope]
Consider $S=\{(a,b),(c,d)\}$. Then $\delta=\delta(S)=\frac{b+d}{a+c}$. 
\end{defi}

We can have at most two 0s in the set $\{a,b,c,d\}$. By symmetry, we may assume that, if two elements equal $0$ then either $a=c=0$ or $a=d=0$. These are cases with either one-dimensional derivations (parallel horizontal strips or rectangular checkerboard patterns). If there is exactly one $0$ in the set $\{a,b,c,d\}$, then we assume that $a=0$ (or later in Theorem~\ref{thm:2m2dim} that $b=0$ or $d=0$). For examples, %$S=\{(0,1),(2,4)\}$,  $S=\{(0,1),(4,2)\}$ or  $S=\{(0,4),(2,1)\}$; 
see Figure~\ref{fig:one0a}.%~and~\ref{fig:one0b}.  

\begin{figure}[htbp!]
\begin{center}
\includegraphics[width=3.3 cm]{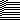}\hspace{7 mm}
\includegraphics[width=3.3 cm]{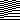}\hspace{7 mm}
\includegraphics[width=3.3 cm]{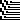}
\end{center}
%\vspace{-3cm}
\caption{These are the first 20 by 20 outcomes of the rulesets $S=\{(0,1),(2,4)\}$,  $S=\{(0,1),(4,2)\}$ and $S=\{(0,4),(2,1)\}$,  respectively. 
}\label{fig:one0a}
\end{figure}

Let us define three ruleset depending sets of positions, where we assume that $a\le c$. This is no restriction since rulesets are unordered by definition. Let $A=\{(x,y)\mid 0\le x<a, y\ge \delta x\}$, $B=\{(x,y)\mid a\le x<c, y\ge \delta x\}$ and $C=\{(x,y)\mid c\le x<c+a, y\ge \delta x\}$. See also Figure~\ref{fig:cuttingprinciple}. 

\begin{figure}[htbp!]
\begin{center}
%\hspace{-3.2 cm}
\begin{tikzpicture}[scale=1.3]
%\draw[dashed] (-3.0,-3.0);
\draw[->] (0,0)--(6,0) node[right]{$x$};
\draw[->] (0,0)--(0,6.1) node[above]{$y$};
\draw[dashed] (0.3,0.2)--(0.3,6);
  %(0.4,0.4)--(0.4,6)node[put]{$\P$};
\draw[thick] (1.2,0.8)--(1.2,6);
\draw[dashed] (1.5,1)--(1.5,6);
\draw[thick] (0,0)--(1.5,1);
\draw[dotted] (1.5,1)--(3.6,2.4);
\put (1,120) {$A$};
\put (22,120) {$B$};
\put (46,120) {$C$};
\put (67,120) {$D$};
\put (11,-10) {$a$}
\put (40,-10) {$c$}
\put (59,30) {$(a+c,b+d)$}
\end{tikzpicture}
\end{center}
%\vspace{-3cm}
\caption{A vertical {\em cut} at $c$ lets us treat the $A$ and $B$ regions as  {\sc one-move} $\{(a,b)\}$. The outcome regions $A, B$ and $C$ are calculated in Lemma~\ref{lem:AB} and Lemma~\ref{lem:C}, and region $D$ is a linear consequence of $A, B$ and $C$; see Theorem~\ref{thm:2m2dim}. 
}\label{fig:cuttingprinciple}
\end{figure}
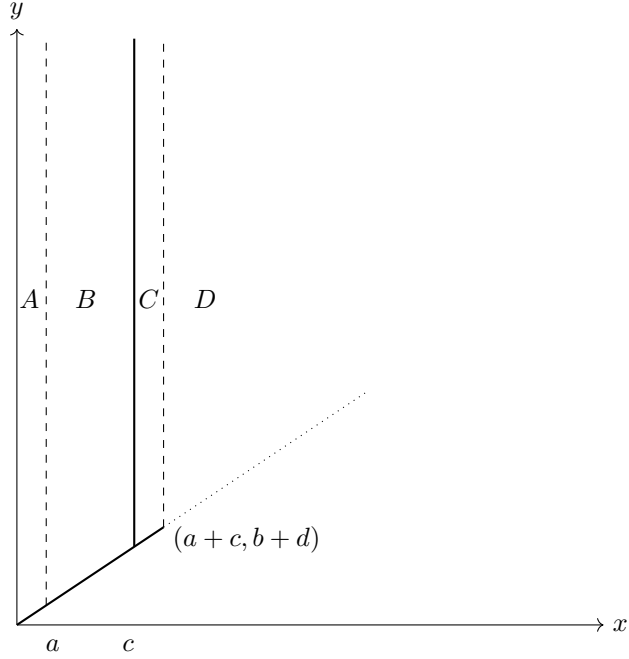
We invite the reader to identify the regions $A$, $B$ and $C$ in Figure~\ref{fig:dilbox}. The regions appear also in Figure~\ref{fig:one0a}, but, because  in those cases $a=0$, then  $A=C=\varnothing$.  Note also that if $a=c$, then $B=\varnothing$.

\begin{lem}[Regions $A$ and $B$]\label{lem:AB}
Suppose $c>0$. The outcomes of regions $A$ and $B$ are as in Lemma~\ref{lem:onemovestruct}.
\end{lem}
\begin{proof}
    Since $x<c$, then {\sc one-move} $\{(a,b)\}$  determines the outcomes.
\end{proof}

\begin{rem}
\label{rem:two_move_a,c=0}
If in Lemma~\ref{lem:AB}, $a=c=0$, then the statement is empty, and we will rely instead on the symmetric case to compute all the outcomes. Indeed $a=c=0$ forces both $b,d>0$, and the symmetric case of $\delta=\infty$ is slope $0$. Our method does not provide any computation if $\delta=\infty$ but it computes all outcomes if $\delta=0$. 
\end{rem}

Next, we determine the outcomes in region $C$. 

In case $a=0$, then $C=\varnothing$, so no computation is required.  Thus every case with a non-trivial $C$ region is covered in the following result. The intuition is given in Figure~\ref{fig:delta}.

\begin{lem}[Region $C$]\label{lem:C}
    %In Figure~\ref{fig:cuttingprinciple}, 
      Suppose that $0<a\le c$. Consider $(x,y)\in C$. Then $(x,y)\in \P$ if and only if $y<d$ and $2na\le x< (2n+1)a$, for some $n\in \nz$.
\end{lem}
\begin{proof}
If $y<d$, then Lemma~\ref{lem:onemovestruct} (a) still applies. Otherwise, by definition of the $C$ region, the move $(c,d)$ leads to a position in $A$, which is non-empty since $a>0$. Those are all \P-positions, since $c>0$. 
    
\end{proof}
Note that $d/c\le \delta$ (which in particular means that the move $(c,d)$ is valid in all of the $C$ region) if and only if $ad\le bc$. 

\begin{cor}[Region $C$ no \P-position]
\label{cor:C_region_noP}
Suppose $0<a\le c$. The region $C$ contains no \P-position if %and only if 
$ad\le bc$ or $(2n+1)a\le c<(2n+2)a$ for some $n\in \nz$.
\end{cor}

% New Proof
\begin{proof}
  Observe that if $ad\le bc$  then it is impossible to have $y< d$. Because  in the region $C$,  $x \ge c$,  $d/c \le \delta$ implies $y \ge x d/c\ge d$. Also, if $(2n+1)a\le c<(2n+2)a$ then $2na\le x<(2n+1)a$ cannot hold. Therefore, using Lemma~\ref{lem:C},  the region $C \subset \N$.
\end{proof}

\begin{defi}[Mirror Ruleset]
    Let $S=\{(a,b),(c,d)\}$. Then the mirror ruleset is $S^{-1}=\{(b,a),(d,c)\}$.
\end{defi}

For example, the rulesets in Figure~\ref{fig:one0b} are the mirrors of those in Figure~\ref{fig:one0a}.
\begin{figure}[htbp!]
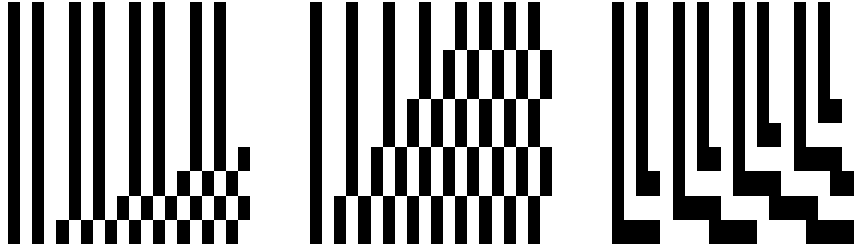

\begin{center}
\scalebox{1}[-1]{\rotatebox[origin=c]{-90}{\includegraphics[width=3.2cm]{012420by20.png}}}\hspace{7 mm}
\scalebox{1}[-1]{\rotatebox[origin=c]{-90}{\includegraphics[width=3.2cm]{014220by20.png}}}\hspace{7 mm}
\scalebox{1}[-1]{\rotatebox[origin=c]{-90}{\includegraphics[width=3.2cm]{210420by20.png}}}
\end{center}
%\vspace{-3cm}
\caption{These are the mirror rulesets of Figure~\ref{fig:one0a}; the first 20 by 20 outcomes of the rulesets $S=\{(1,0),(4,2)\}$,  $S=\{(1,0),(2,4)\}$ and $S=\{(1,2), (4,0)\}$,  respectively. 
}\label{fig:one0b}
\end{figure}  

The mirror rulesets gets the analogous sets $A^{-1}=\{(x,y)\mid 0\le x<b, y\le \delta x\}$, $B^{-1}=\{(x,y)\mid b\le x<d, y\le \delta x\}$ and $C^{-1}=\{(x,y)\mid d\le x<d+b, y\le \delta x\}$. 
 Let $\Gamma = A\cup B\cup C$ and let $\Gamma^{-1}=A^{-1}\cup B^{-1}\cup C^{-1}$. 
We conclude this section, by collecting its parts and fusing them into a main result. 
\begin{thm}[{\sc two-move}]\label{thm:2m2dim}
The position $(x,y)\in \P$ if and only if  $(x,y)\in T^k((\Gamma\cup \Gamma^{-1})\cap \P)$, for some $k$, %if $y\ge \delta x$, or
    %\item $(x,y)\in T^k((A^{-1}\cup B^{-1}\cup C^{-1})\cap \P)$, for some $k$, if $y\le \delta x$, 
where 
\begin{itemize}
\item $(x,y)\in \Gamma $ is a \P-position if and only if $2na\le x< (2n+1)a$, for some $n\in \nz$ and  $(x,y)\not\ge (c,d)$. 
%\item $(x,y)\in A\cup B$ is a \P-position if and only if $2ka\le x< (2k+1)a$, for some $k\in \nz$. And $(x,y)\in C$ is a \P-position if and only if $y<d$ and $2ka\le x< (2k+1)a$, for some $k\in \nz$. 
\item if $b\le d$,  $(x,y)\in  \Gamma^{-1}$ is a \P-position if and only if $2nb\le x< (2n+1)b$, for some $n\in \nz$, and $(x,y)\not\ge (c,d)$. 
\item if $d<b$, $(x,y)\in  \Gamma^{-1}$ is a \P-position if and only if $2nd\le x< (2n+1)d$, for some $n\in \nz$, and $(x,y)\not\ge (a,b)$. 

%$(x,y)\in A^{-1}\cup B^{-1}$ is a \P-position if and only if $2kd\le x< (2k+1)d$, for some $k\in \nz$. And $(x,y)\in C^{-1}$ is a \P-position if and only if $y<a$ and $2kd\le x< (2k+1)d$, for some $k\in \nz$. 
\end{itemize}
\end{thm}
 \begin{proof}
 Combine Lemma~\ref{lem:PtoP}, Lemma~\ref{lem:onemovestruct} and Observation~\ref{obs:T} with Lemmas~\ref{lem:AB} and~\ref{lem:C}.  
 \end{proof}
 %%%%%%%%%%%%%%%The ruleset {(4,1) (9,10)}. First 40 by 40 positions%% c=2a+1, ad>bc.       
\begin{figure}[htbp!]
\begin{center}
\begin{tikzpicture}[scale = 0.2]
 \foreach \x/\y in {9/9,10/9,9/8}
 \filldraw[color = red] (\x, \y) rectangle (\x+1, \y+1);% 
%\node at (14.95,14.95) {\includegraphics[width=256pt]{23_31_12__30.png}};
\colorlet{yellow}{gray!60!yellow}
\draw[step=1 cm,yellow] (0, 0) grid (40,40); 
%%%%%%%%%%%%%%%%%%
 \foreach \x/\y in 
{0/0, 0/1, 0/2, 0/3, 0/4, 0/5, 0/6, 0/7, 0/8, 0/9, 0/10, 0/11, 0/12, 0/13, 0/14, 0/15, 0/16, 0/17, 0/18, 0/19, 0/20, 0/21, 0/22, 0/23, 0/24, 0/25, 0/26, 0/27, 0/28, 0/29, 0/30, 0/31, 0/32, 0/33, 0/34, 0/35, 0/36, 0/37, 0/38, 0/39, 1/0, 1/1, 1/2, 1/3, 1/4, 1/5, 1/6, 1/7, 1/8, 1/9, 1/10, 1/11, 1/12, 1/13, 1/14, 1/15, 1/16, 1/17, 1/18, 1/19, 1/20, 1/21, 1/22, 1/23, 1/24, 1/25, 1/26, 1/27, 1/28, 1/29, 1/30, 1/31, 1/32, 1/33, 1/34, 1/35, 1/36, 1/37, 1/38, 1/39, 2/0, 2/1, 2/2, 2/3, 2/4, 2/5, 2/6, 2/7, 2/8, 2/9, 2/10, 2/11, 2/12, 2/13, 2/14, 2/15, 2/16, 2/17, 2/18, 2/19, 2/20, 2/21, 2/22, 2/23, 2/24, 2/25, 2/26, 2/27, 2/28, 2/29, 2/30, 2/31, 2/32, 2/33, 2/34, 2/35, 2/36, 2/37, 2/38, 2/39, 3/0, 3/1, 3/2, 3/3, 3/4, 3/5, 3/6, 3/7, 3/8, 3/9, 3/10, 3/11, 3/12, 3/13, 3/14, 3/15, 3/16, 3/17, 3/18, 3/19, 3/20, 3/21, 3/22, 3/23, 3/24, 3/25, 3/26, 3/27, 3/28, 3/29, 3/30, 3/31, 3/32, 3/33, 3/34, 3/35, 3/36, 3/37, 3/38, 3/39, 4/0, 5/0, 6/0, 7/0, 8/0, 8/2, 8/3, 8/4, 8/5, 8/6, 8/7, 8/8, 8/9, 8/10, 8/11, 8/12, 8/13, 8/14, 8/15, 8/16, 8/17, 8/18, 8/19, 8/20, 8/21, 8/22, 8/23, 8/24, 8/25, 8/26, 8/27, 8/28, 8/29, 8/30, 8/31, 8/32, 8/33, 8/34, 8/35, 8/36, 8/37, 8/38, 8/39, 9/0, 9/2, 9/3, 9/4, 9/5, 9/6, 9/7, 9/8, 9/9, 10/0, 10/2, 10/3, 10/4, 10/5, 10/6, 10/7, 10/8, 10/9, 11/0, 11/2, 11/3, 11/4, 11/5, 11/6, 11/7, 11/8, 11/9, 12/0, 12/2, 13/0, 13/2, 13/11, 13/12, 13/13, 13/14, 13/15, 13/16, 13/17, 13/18, 13/19, 13/20, 13/21, 13/22, 13/23, 13/24, 13/25, 13/26, 13/27, 13/28, 13/29, 13/30, 13/31, 13/32, 13/33, 13/34, 13/35, 13/36, 13/37, 13/38, 13/39, 14/0, 14/2, 14/11, 14/12, 14/13, 14/14, 14/15, 14/16, 14/17, 14/18, 14/19, 14/20, 14/21, 14/22, 14/23, 14/24, 14/25, 14/26, 14/27, 14/28, 14/29, 14/30, 14/31, 14/32, 14/33, 14/34, 14/35, 14/36, 14/37, 14/38, 14/39, 15/0, 15/2, 15/11, 15/12, 15/13, 15/14, 15/15, 15/16, 15/17, 15/18, 15/19, 15/20, 15/21, 15/22, 15/23, 15/24, 15/25, 15/26, 15/27, 15/28, 15/29, 15/30, 15/31, 15/32, 15/33, 15/34, 15/35, 15/36, 15/37, 15/38, 15/39, 16/0, 16/2, 16/4, 16/5, 16/6, 16/7, 16/8, 16/9, 16/11, 16/12, 16/13, 16/14, 16/15, 16/16, 16/17, 16/18, 16/19, 16/20, 16/21, 16/22, 16/23, 16/24, 16/25, 16/26, 16/27, 16/28, 16/29, 16/30, 16/31, 16/32, 16/33, 16/34, 16/35, 16/36, 16/37, 16/38, 16/39, 17/0, 17/2, 17/4, 17/5, 17/6, 17/7, 17/8, 17/9, 17/11, 18/0, 18/2, 18/4, 18/5, 18/6, 18/7, 18/8, 18/9, 18/11, 19/0, 19/2, 19/4, 19/5, 19/6, 19/7, 19/8, 19/9, 19/11, 20/0, 20/2, 20/4, 20/11, 21/0, 21/2, 21/4, 21/11, 21/13, 21/14, 21/15, 21/16, 21/17, 21/18, 21/19, 21/20, 21/21, 21/22, 21/23, 21/24, 21/25, 21/26, 21/27, 21/28, 21/29, 21/30, 21/31, 21/32, 21/33, 21/34, 21/35, 21/36, 21/37, 21/38, 21/39, 22/0, 22/2, 22/4, 22/11, 22/13, 22/14, 22/15, 22/16, 22/17, 22/18, 22/19, 22/20, 23/0, 23/2, 23/4, 23/11, 23/13, 23/14, 23/15, 23/16, 23/17, 23/18, 23/19, 23/20, 24/0, 24/2, 24/4, 24/6, 24/7, 24/8, 24/9, 24/11, 24/13, 24/14, 24/15, 24/16, 24/17, 24/18, 24/19, 24/20, 25/0, 25/2, 25/4, 25/6, 25/7, 25/8, 25/9, 25/11, 25/13, 26/0, 26/2, 26/4, 26/6, 26/7, 26/8, 26/9, 26/11, 26/13, 26/22, 26/23, 26/24, 26/25, 26/26, 26/27, 26/28, 26/29, 26/30, 26/31, 26/32, 26/33, 26/34, 26/35, 26/36, 26/37, 26/38, 26/39, 27/0, 27/2, 27/4, 27/6, 27/7, 27/8, 27/9, 27/11, 27/13, 27/22, 27/23, 27/24, 27/25, 27/26, 27/27, 27/28, 27/29, 27/30, 27/31, 27/32, 27/33, 27/34, 27/35, 27/36, 27/37, 27/38, 27/39, 28/0, 28/2, 28/4, 28/6, 28/11, 28/13, 28/22, 28/23, 28/24, 28/25, 28/26, 28/27, 28/28, 28/29, 28/30, 28/31, 28/32, 28/33, 28/34, 28/35, 28/36, 28/37, 28/38, 28/39, 29/0, 29/2, 29/4, 29/6, 29/11, 29/13, 29/15, 29/16, 29/17, 29/18, 29/19, 29/20, 29/22, 29/23, 29/24, 29/25, 29/26, 29/27, 29/28, 29/29, 29/30, 29/31, 29/32, 29/33, 29/34, 29/35, 29/36, 29/37, 29/38, 29/39, 30/0, 30/2, 30/4, 30/6, 30/11, 30/13, 30/15, 30/16, 30/17, 30/18, 30/19, 30/20, 30/22, 31/0, 31/2, 31/4, 31/6, 31/11, 31/13, 31/15, 31/16, 31/17, 31/18, 31/19, 31/20, 31/22, 32/0, 32/2, 32/4, 32/6, 32/8, 32/9, 32/11, 32/13, 32/15, 32/16, 32/17, 32/18, 32/19, 32/20, 32/22, 33/0, 33/2, 33/4, 33/6, 33/8, 33/9, 33/11, 33/13, 33/15, 33/22, 34/0, 34/2, 34/4, 34/6, 34/8, 34/9, 34/11, 34/13, 34/15, 34/22, 34/24, 34/25, 34/26, 34/27, 34/28, 34/29, 34/30, 34/31, 34/32, 34/33, 34/34, 34/35, 34/36, 34/37, 34/38, 34/39, 35/0, 35/2, 35/4, 35/6, 35/8, 35/9, 35/11, 35/13, 35/15, 35/22, 35/24, 35/25, 35/26, 35/27, 35/28, 35/29, 35/30, 35/31, 36/0, 36/2, 36/4, 36/6, 36/8, 36/11, 36/13, 36/15, 36/22, 36/24, 36/25, 36/26, 36/27, 36/28, 36/29, 36/30, 36/31, 37/0, 37/2, 37/4, 37/6, 37/8, 37/11, 37/13, 37/15, 37/17, 37/18, 37/19, 37/20, 37/22, 37/24, 37/25, 37/26, 37/27, 37/28, 37/29, 37/30, 37/31, 38/0, 38/2, 38/4, 38/6, 38/8, 38/11, 38/13, 38/15, 38/17, 38/18, 38/19, 38/20, 38/22, 38/24, 39/0, 39/2, 39/4, 39/6, 39/8, 39/11, 39/13, 39/15, 39/17, 39/18, 39/19, 39/20, 39/22, 39/24, 39/33, 39/34, 39/35, 39/36, 39/37, 39/38, 39/39}
\filldraw[color = black, opacity =.35] (\x, \y) rectangle (\x+1, \y+1);% 
\draw[thin] (0,0) -- (13,11);
\node at (11,20)[]{$C$};
\node at (9.5,-1)[]{$c$};
\node at (-1,10.5)[]{$d$};
\node at (4.5,-1)[]{$a$};
\end{tikzpicture}
\end{center}
\caption{The ruleset $S=\{(4,1), (9,10)\}$. This ruleset satisfies $bc<ad$ and $c=2a+1$. Hence there are \P-positions in the $C$-region as defined in Figure~\ref{fig:cuttingprinciple}. Namely $\{(x,y)\mid y/x\ge \delta, y<d\}=\{(8,9),(9,9),(10,9)\}$.}
\label{fig:delta}
\end{figure}
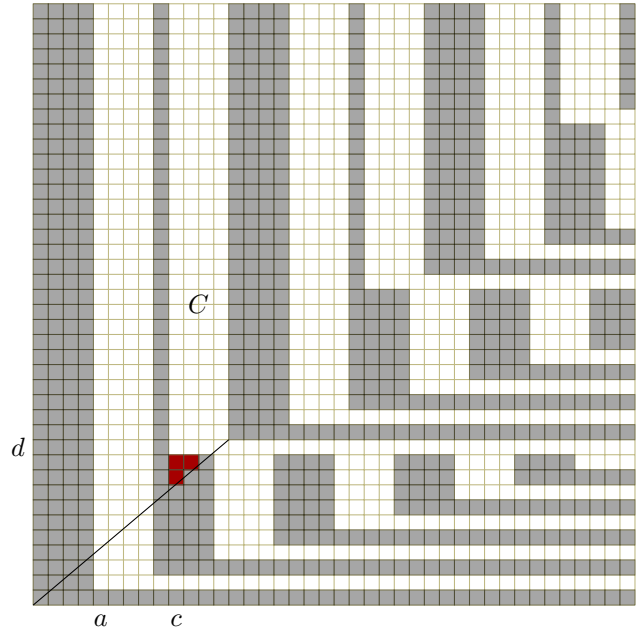

\begin{figure}[htbp!]
\begin{center}
\includegraphics[width=8cm]{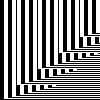}
\end{center}
%\vspace{-3cm}
\caption{The ruleset $S=\{(4,1),(17,15)\}$ corresponds to the case $c=4a+1$, and otherwise the idea is the same as in Figure~\ref{fig:delta}.} %\ur{we probablly do not need this one.}}\label{fig:411715} \In{Not required. The next figure is better.}
\label{fig:411715}
\end{figure}

\section{two-move vector subtraction}\label{sec:anydimension}%\label{sec:onemoveanydimension}
This section elaborates on the natural generalization of {\sc two-move}, from the previous section, to any dimension. This study focuses on the \P/\N-decision problem of a generic game position. Recall the `\P-to-\P\ principle' for {\sc two-move vector subtraction}, Lemma~\ref{lem:PtoP}:
\begin{lem*}[General \P-to-\P]
%\label{lem:PtoPd}
Consider  {\sc two-move vector subtraction}  $S= \{ \s_1, \s_2\}$. Then  $\x \in \P$  if and only if $\x+\s_1 +\s_2\in \P$. 
\end{lem*}

 \begin{rem}\label{rem:NtoN}
     By Lemma~\ref{lem:PtoP},   if $S= \{ \s_1, \s_2\}$, then $\x \in \N$  if and only if $\x+\s_1 +\s_2\in \N$. 
 \end{rem}

\begin{defi}[Translation Function]\label{def:trans}
    Consider a $d$-dimensional ruleset $S=\{\s_1 ,\s_2\}$ of size two. %, and a $(j,\alpha)$-cut. %For $k\in \nz$, let the $k^{\mathrm{th}}$ 
    Let the translation, of $\X\in \B$, be 
    $t(\x)=t_S(\x)=\x +\s_1+\s_2$.
\end{defi}
We will iterate this function: $t^k(\x)=t(t^{k-1}(\x))$, where $t^0(\x)=\x$. Thus, for all $k$, $t^k(\x) = \x +k(\s_1+\s_2)$. 
\begin{thm}[General {\sc two-move}]\label{thm:2mddim}
    Consider {\sc two-move vector subtraction} $S=\{\s_1,\s_2\}$. 
    Then $\x\in\P$ if and only if $\x-t^k(\x)\in\P$, where $k \in\nz$  is such that $\x-t^{k+1}(\x)\ngeq\boldsymbol 0$.
    
\end{thm}
\begin{proof}
This is immediate by Lemma~\ref{lem:PtoP}.
\end{proof}
Thus, the complexity of the $\P$/$\N$ decision problem for a generic position $\x\in\B$ reduces to computing the outcome of its smallest  representative, say $\x'$ in $\B$ $\pmod {\s_1+\s_2}$. This, in turn reduces to  one  out of two situations:
\begin{itemize}
\item[(i)] $0\le \x'\ngeq \{\s_i,\s_{3-i}\}$, for both $i\in\{1,2\}$;
\item[(ii)] $0\le \x'-\s_i\ngeq \{\s_i,\s_{3-i}\}$, for some $i\in\{1,2\}$;
\item[(iii)] find a $k\in\nz$ such that $(k+1)\s_i\ngeq\x'\ge k\s_i$, $i\in\{1,2\}$. 
\end{itemize}
In case (i), $\x'$ is a terminal \P-position, and so, by Theorem~\ref{thm:2mddim}, $\x\in\P$. 
In case (ii), there is a move to a terminal \P-position; hence $\x'\in \N$, and thus, by Theorem~\ref{thm:2mddim}, $\x\in \N$. If (i) does not hold, then test if she-loves-me-she-loves-me-not according to (ii) returns a terminal \P-position for some $i\in \{1,2\}$. Denote this procedure by Algorithm~1. Observe that a response to the move $\s_i$ with the move $\s_{3-i}$ is impossible by the choice of $\x$ as the smallest representative in $\B$ $\pmod {\s_1+\s_2}$.

So, the complexity is in essence the same as determining the remainder after division, but here in $d$ dimensions. %(whichever coordinate will have the smallest divisor will determine the outcome). 
   The standard na\"ive algorithm for the \P/\N\ membership problem is exponential in succinct input size \cite{fraenkel2004complexity}. 
   %\In{I think we can include a reference here}.\ur{Fraenkel has a paper about complexity for impartial games.} \In{okay!} 
   As a corollary, we get an improvement for the case of {\sc two-move vector subtraction}. 
\begin{cor}
  Algorithm~1 for the \P/\N\   membership problem of {\sc two-move vector subtraction} is linear in the succinct input size.
\end{cor}
\begin{proof}

See the paragraph after Theorem~\ref{thm:2mddim}.
\end{proof}

 \section{Three moves in two dimensions}\label{sec:three_move_game}
We consider special types of  {\sc three-move subtraction} and identify some properties of their outcomes. In particular, we are interested in: when does the \P-to-\P\ property from {\sc two-move subtraction}  continue to hold?\footnote{The nature of the results is `lemma', and we predict that they will be useful in future work if not used here.} One attractive family is the family of {\sc three-move additive} rulesets. 

\begin{defi}[{\sc additive subtraction}]\label{def:addsym}
The ruleset $S=\{(a,b),(c,d),(e,f)\}$ is additive if $e=a+c$ and $f=b+d$.
\end{defi}

Figure~\ref{fig:additive} depicts outcomes of the additive rulesets $S=\{(1,2),(2,1),(3,3)\}$ and $S=\{(1,2),(3,4),(4,6)\}$. See also Figures~\ref{fig:122133color} and \ref{fig:coloring123446}.

  \begin{figure}[htbp!]
  \begin{center}
\includegraphics[width=6cm]{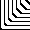}\hspace{.5 cm}  \includegraphics[width=6cm]{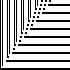}
  \end{center}
  \caption{The initial outcomes of {\sc additive subtraction} $S=\{(1,2),(2,1),(3,3)\}$ and $S=\{(1,2),(3,4),(4,6)\}$, respectively.}
  \label{fig:additive}
  \end{figure}
\begin{lem}[{\sc additive subtraction}]\label{lem:three_move_withsepecialstructure}
Let $S=\{(a,b),(c,d), (a+c,b+d)\}$. If  $(x+2a+c, y+2b+d)\in \P$ and $(x+2c+a, y+2d+b)\in \P$ then $(x,y) \in \P$.
 \end{lem}
\begin{proof}
The set of options of $(x+2a+c,y+2b+d)$ is 
$$\{(x+a+c,y+b+d),(x+a,y+b),(x+2a,y+2b)\}\subset \N.
$$  
The set of options of $(x+2c+a,y+2d+b)$ is 
$$\{(x+c+a,y+d+b),(x+c,y+d),(x+2c,y+2d)\}\subset \N.
$$ 
For each option, there is a move to a $\P$-position. The option $(x+a+c,y+b+d)\in\N$ occurs in both sets. But both $(x+a,y+b),(x+c,y+d)\in\N$. Hence the third move must lead to a $\P$-position. Hence $(x,y)\in \P$.
%The proof is immediate by the definition of $\P$ and $\N$  positions.
\end{proof}

\begin{rem}
 The reverse direction of Lemma~ \ref{lem:three_move_withsepecialstructure} is not true. Consider the ruleset $S =\{(1,1),(2,2),(3,3)\}$. Then $(0,0) \in \P$ but  $(5,5) \notin \P$. 
\end{rem}
The next lemma shows that, given some proviso, we can still have \P-to-\P\ update rules for {\sc additive subtraction}.
\begin{lem}[{\sc additive subtraction} \P-to-\P]\label{lem:three_move_PtoP}
Let $S=\{(a,b),(c,d), (a+c,b+d)\}$ and suppose $(x,y)\in\P$. 
\begin{enumerate}[(i)]
    \item If $(x+2a+2c, y+2b+2d)\in \N$ then $(x+2c+a, y+2d+b)\in \P$ or $(x+2a+c, y+2b+d)\in \P$;
    \item if $(x+2a+c, y+2b+d)\in \N$ then $(x+2a,y+2b)\in\P$;%\ur{The converse will be used.}
    \item if $(x+2c+a, y+2d+b)\in \N$ then  $(x+2c,y+2d)\in\P$;%\ur{The converse will be used.}
    \item if both $(x+2c+a, y+2d+b)\in \N$ and $(x+2a+c, y+2b+d)\in \N$ then $(x+2a+2c, y+2b+2d)\in \P$.%\ur{This is the converse of (i).}
\end{enumerate}
 \end{lem}
\begin{proof}
Let $S=\{(a,b),(c,d), (a+c,b+d)\}$. 
The set of positions with $(x,y)\in \P$ as an option is 
\begin{align}\label{eq:xopt}
\{(x+a+c,y+b+d),(x+a,y+b),(x+c,y+d)\}\subset \N.
\end{align} 
For (i), observe that one of the options must be in \P. But $(x+a+c,y+b+d)\in \N$. Hence $(x+2c+a, y+2d+b)\in \P$ or $(x+2a+c, y+2b+d)\in \P$.
For (ii) and (iii), observe that two of the options lead to \N-positions in \eqref{eq:xopt}. Hence the third must be in \P. 
For (iv), observe that all three options of  $(x+2a+2c, y+2b+2d)$ are \N-positions. 
\end{proof}

The following result is an adaptation from one-dimensional.
\begin{lem}
\label{Lemma_with_twin}
Let $S=\{(a,b),(a+c,b+c), (a+2c,b+2c)\}$. If $(x, y)\in \P$ then  $(x+2a+2c, y+2b+2c)\in \P$.
\end{lem}
\begin{proof}

  Let $S=\{(a,b),(a+c,b+c), (a+2c,b+2c)\}$ be the ruleset. The set of positions with $(x,y)\in \P$ as an option is 
$$
\{(x+a,y+b), (x+a+c,y+b+c), (x+a+2c, y+b+2c)\}\subset \N.
$$
This set coincides with the set of options of $(x+2a+2c, y+2b+2c)$. %is 
%$$
%\{ (x+a+2c, y+b+2c), (x+a+c, y+b+c), (x+a,y+b) \} \subset \N.
%$$
%Since all the $((x+a+2c, y+b+2c), %(x+a+c, y+b+c), (x+a,y+b))\in \N$, 
Therefore $ (x+2a+2c, y+2b+2c)\in \P$.
\end{proof}

Figure~\ref{fig:233112_100} gives an example of a ruleset with {\sc asymmetric additive} rules. 

\begin{defi}[{\sc asymmetric additive subtraction}]
    The ruleset $$S=\{(a,b),(c,d),(e,f)\}$$ is {\sc asymmetric additive} if $e=a+c$ and $d=b+f$.
\end{defi}

%$S=\{(a,c),(b,c+d), (a+b,d)\}$

\begin{lem}[{\sc asymmetric additive subtraction}]\label{lem:three_move_sum}

Let $S=\{(a,b),(c,b+d), (a+c,d)\}$ and suppose $(x,y)\in\P$. 
\begin{enumerate}[(i)]
\item If  $(x+2a+2c, y+2d) \in \N$  then   $(x+a+2c, y+2d-b) \in \P$ or  $(x+2a+c, y+d-b) \in \P$;
\item if $(x+a+c, y+2b+d) \in \N$ then $(x,y+2b) \in \P$;
\item if  $(x+2a+c, y+b+d) \in \N$ then $(x+2a,y) \in \P$;
\item  if both  $(x+a+2c, y+2d-b) \in \N$ and  $(x+2a+c, y+d-b) \in \N$ then $(x+2a+2c, y+2d) \in \P$.
\end{enumerate}
\end{lem}

\begin{proof}
Let  $S=\{(a,b),(c,b+d),(a+c,d)\}$. The set of positions with $(x,y)\in \P$ as an option is 
\begin{align}\label{eq:xopt_three}
\{(x+a,y+b),(x+c,y+b+d),(x+a+c,y+d)\}\subset \N.
\end{align} 

Consider the assumptions of item (i). The set of options of $(x+2a+2c, y+2d)\in\N$ is 
$\{(x+a+2c,y+2d-b),(x+2a+c,y+d-b),(x+a+c,y+d)\}.
$
By \eqref{eq:xopt_three}, hence $(x+a+2c, y+2d-b) \in \P$ or  $(x+2a+c, y+d-b) \in \P$.

Consider the assumption of item (ii).  The set of options of $(x+a+c , y+2b+d)\in\N$ is 
$\{(x+c,y+b+d),(x+a,y+b),(x,y+2b)\}$.   
By \eqref{eq:xopt_three}, $\{(x+c,y+b+d),(x+a,y+b)\} \subset \N$. Hence $(x,y+2b) \in \P$. 

For the assumption of item (iii), the set of options of $(x+2a+c, y+b+d)\in\N$ is 
$\{(x+a+c,y+d),(x+2a,y),(x+a,y+b)\}$.  
By \eqref{eq:xopt_three}, $(x+a+c,y+d), (x+a,y+b) \subset \N$. Hence  $(x+2a,y) \in \P$.

Consider the assumptions of item (iv). The set of options of $(x+2a+2c, y+2d)$ is $\{(x+a+2c,y+2d-b),(x+2a+c,y+d-b),(x+a+c,y+d)\}\subset \N$. 
By \eqref{eq:xopt_three}, hence $(x+2a+2c, y+2d) \in \P$.
\end{proof}

\begin{figure}[htbp!]
  \begin{center}
 \includegraphics[width=8cm]{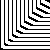}
  \end{center}
  \caption{ The initial outcomes of the {\sc asymmetric additive} ruleset $S=\{(1,2), (2,3), (3,1)\}$ (50 by 50 positions).}
  \label{fig:233112_100}
  \end{figure}

\section{Outcome segmentation}
\label{sec:outcomesegment}
When it comes to generic {\sc finite two-dimensional subtraction}, we will provide some thrilling open problems in Section~\ref{sec:picopen}, which requires a couple of technical definitions. The definitions are also used to compute \P-positions in some instances in this section.

An {\em outcome segment} is a special type of geometry on the outcomes. %that we define as follows. 

\begin{defi}[Coloring Scheme/Automaton]
   A subset $\Gamma $ of the first quadrant may permit a coloring scheme as follows. 
An update rule $u=(u_1,u_2)$ takes as input a colored position $(x,y,c)$, where $c\in C_u\subset C$, a given  set of colors and where $(x,y)\in\B$. The output is a colored position $(x+u_1,y+u_2, c')\in \Gamma \times C$. Thus, a single color can have several update rules attached to it. A legal coloring scheme $(X_0, U, C)$ gives at most one color to each position in $\Gamma$; here $X_0\subset \Gamma$ is a set of initially $C$-colored positions and $U$ is a finite set of update rules.  
\end{defi}

Thus, a coloring scheme could be regarded as a simple cellular automaton, where the color of each cell depends on at most one younger cell. For known Turing complete cellular automatons (such as a rule 110), updates depend on several (three) younger cells.

Note however that a coloring scheme permits a cell to be `updated' more than once provided it is given the same color.

\begin{defi}[Outcome Segment]
\label{def:outcomesegment}
   Consider a subset of the first quadrant $\Gamma=\{(x,y)\mid \alpha x +k<y <\beta x+m\}\subset \B^2$, for some fixed rationals $\alpha, \beta, k, m$, and a finite set of update rules $U$. Then $\Gamma$ is a  $U$-segment it permits a coloring scheme via the rules in $U$. Suppose that a $U$-segment $\Gamma$ partitions the \P- and \N-positions of a given ruleset, such that the colored positions correspond to the \P-positions. Then $\Gamma$ is an outcome segment.
   \end{defi}
%   \ur{I am using different terminology/notation for these concepts. We need to make decisions.}
\begin{example}[Outcome Segment]\label{ex:OS}
    %\ur{Here, we illustrate the perfect outcome segments for one-move games.}
    See the middle segment in Figure~\ref{fig:tiling_example}. Set the initial coloring as $\mathbf{red}=\{(0,0),(0,1),(0,2)\}$, $\mathbf{green}=\{(1,0),(1,1)\}$ and $\mathbf{blue}=\{(2,1)\}$. The update rule $\vdash$ ``adjoin'' is: 
    \begin{itemize}
        \item $(x,y)\in \mathbf{red}\rightarrow (x+5,y+4)\vdash  \mathbf{green}$,
        \item $(x,y)\in \mathbf{green}\rightarrow (x+3,y+5)\vdash \mathbf{red}$,
        \item $(x,y)\in \mathbf{gre  en}\cup \mathbf{blue}\rightarrow (x+5,y+4)\vdash \mathbf{blue}$.
    \end{itemize}
    In Proposition~\ref{prop:OS}, we prove that this coloring scheme corresponds exactly to the \P-positions of the middle segment of the {\sc asymmetric additive} ruleset $$S=\{(1,2),(2,3),(3,1)\}.$$
\end{example}

\begin{defi}[Outcome Segmentation]
\label{def:outcomesegmentation}
    Consider a ruleset $S$. A finite union of outcome segments $\cup_{i}\Gamma_i$ is an outcome segmentation of $S$ (or just a segmentation of $S$) if, for all $i,j$, $\Gamma_i\cap \Gamma_j=\varnothing$, and, for all sufficiently large game boards of $n$ positions, they contain at least $n-o(n)$ positions.  A segmentation is perfect if it contains all positions. A $k$-segmentation has $k$ outcome segments. 
\end{defi}
The number of positions is $n-o(n)$ instead of $n$  because the border regions can be non-trivial strips instead of lines.

   \begin{example}[Outcome Segmentation]\label{ex:titlingex}
       There are three segments in Figure~\ref{fig:tiling_example}. The middle segment is described in Example~\ref{ex:OS}. The update rule for the lower segment is $(5,4)$ and the initial set consists of all positions $(x,0)$ with $x\ge 2$. The upper segment has two alternating update rules $(4,4)$ and $(4,5)$, with $(4,4)$ applied to the initial set $\{(0,y)\mid y\ge 3\}$. We prove that altogether this is a perfect outcome segmentation of the game in Figure~\ref{fig:tiling_example}.
   \end{example}
   The terminology ``outcome segmentation" is motivated by multiple observations where segments are separated by lines and/or strips.  
\begin{example}\label{ex:lines}
Consider the two lines $f(x)=9x/8+2$ and $g(x)=4x/5-4/5$. The positions $(x,y)$ in the middle (lower, upper) segment satisfy $\lceil g(x)\rceil < y <  \lceil f(x)\rceil $ ($y\le \lceil g(x)\rceil $, $y\ge \lceil f(x)\rceil $). 
\end{example}

\begin{figure}[htbp!]
\begin{center}
\begin{tikzpicture}[scale = 0.5]
%\node at (14.95,14.95) {\includegraphics[width=256pt]{23_31_12__30.png}};
\colorlet{yellow}{gray!60!yellow}
\draw[step=1 cm,yellow, thick] (0, 0) grid (23,21); 
%%%%%%%%%%%%%%%%%%
 \foreach \x/\y in %{0/34, 4/34, 8/34, 12/34, 16/34, 20/34, 24/34, 28/34, 0/33, 4/33, 8/33, 12/33, 16/33, 20/33, 24/33, 28/33, 29/33, 0/32, 4/32, 8/32, 12/32, 16/32, 20/32, 24/32, 28/32, 29/32, 30/32, 31/32, 0/31, 4/31, 8/31, 12/31, 16/31, 20/31, 24/31, 29/31, 30/31, 31/31, 32/31, 33/31, 0/30, 4/30, 8/30, 12/30, 16/30, 20/30, 24/30, 31/30, 32/30, 33/30, 34/30, 0/29, 4/29, 8/29, 12/29, 16/29, 20/29, 24/29, 33/29, 34/29, 0/28, 4/28, 8/28, 12/28, 16/28, 20/28, 24/28, 25/28, 26/28, 0/27, 4/27, 8/27, 12/27, 16/27, 20/27, 24/27, 25/27, 26/27, 27/27, 28/27, 0/26, 4/26, 8/26, 12/26, 16/26, 20/26, 26/26, 27/26, 28/26, 29/26, 30/26, 
 %{0/25, 4/25, 8/25, 12/25, 16/25, 20/25, 28/25, 29/25, 30/25, 31/25, 32/25, 0/24, 4/24, 8/24, 12/24, 16/24, 20/24, 21/24, 30/24, 31/24, 32/24, 33/24, 34/24, 0/23, 4/23, 8/23, 12/23, 16/23, 20/23, 21/23, 22/23, 23/23, 
 %{0/22, 4/22, 8/22, 12/22, 16/22, 21/22, 22/22, 23/22, 24/22, 25/22, 0/21, 4/21, 8/21, 12/21, 16/21, 23/21, 24/21, 25/21, 26/21, 27/21, 
 {0/20, 4/20, 8/20, 12/20, 16/20, 0/19, 4/19, 8/19, 12/19, 16/19, 17/19, 18/19, 0/18, 4/18, 8/18, 12/18, 16/18, 17/18, 18/18, 19/18, 20/18, 0/17, 4/17, 8/17, 12/17, 18/17, 19/17, 20/17, 21/17, 22/17, 0/16, 4/16, 8/16, 12/16, 20/16, 21/16, 22/16, 0/15, 4/15, 8/15, 12/15, 13/15, 0/14, 4/14, 8/14, 12/14, 13/14, 14/14, 15/14, 0/13, 4/13, 8/13, 13/13, 14/13, 15/13, 16/13, 17/13, 0/12, 4/12, 8/12, 15/12, 16/12, 17/12, 18/12, 19/12, 20/12, 21/12, 22/12, 0/11, 4/11, 8/11, 0/10, 4/10, 8/10, 9/10, 10/10, 0/9, 4/9, 8/9, 9/9, 10/9, 11/9, 12/9, 0/8, 4/8, 10/8, 11/8, 12/8, 13/8, 14/8, 15/8, 16/8, 17/8, 18/8, 19/8, 20/8, 21/8, 22/8,  0/7, 4/7, 0/6, 4/6, 5/6, 0/5, 4/5, 5/5, 6/5, 7/5, 0/4, 5/4, 6/4, 7/4, 8/4, 9/4, 10/4, 11/4, 12/4, 13/4, 14/4, 15/4, 16/4, 17/4, 18/4, 19/4, 20/4, 21/4, 22/4, 0/3, 0/2, 0/1, 1/1, 2/1, 0/0, 1/0, 2/0, 3/0, 4/0, 5/0, 6/0, 7/0, 8/0, 9/0, 10/0, 11/0, 12/0, 13/0, 14/0, 15/0, 16/0, 17/0, 18/0, 19/0, 20/0, 21/0, 22/0}%, 23/0, 24/0, 25/0, 26/0, 27/0, 28/0, 29/0, 30/0, 31/0, 32/0, 33/0, 34/0} 
 \filldraw[color = black, opacity =.1] (\x, \y) rectangle (\x+1, \y+1);% 

%%%%%%%%%%%%%%%%%
\foreach \x/\y in {2/1,6/4,6/5,7/5,10/8,11/8,10/9,11/9,12/9,10/10,14/13,15/13,16/13,17/13,15/12,16/12,14/14,15/14,18/18,19/18,20/18,18/17,18/19,19/17,20/17,21/17,22/17,20/16,21/16}%,27/21,32/25}
 \filldraw[color = blue, opacity =.6] (\x+.15, \y+.15) rectangle (\x+.85, \y+.85);
 
 \foreach \x/\y in {0/0,0/1,0/2,4/5, 4/6, 8/11,8/10,8/9,12/15,12/14,16/20,16/19,16/18}%,24/29,24/28,24/27, 20/23,20/24}
 \filldraw[color = red, opacity =.9] (\x+.15, \y+.15) rectangle (\x+.85, \y+.85);
  
 \foreach \x/\y in {1/0,1/1, 5/6,5/5,5/4,9/10,9/9,13/15,13/14,13/13,17/18,17/19}%,21/24,29/33}
 \filldraw[color = green!80!black] (\x+.15, \y+.15) rectangle (\x+.85, \y+.85);

\draw (-0.7 ,0.35) node {$0$};
%\draw (4, -.5) node {$9$};

\draw (0.35, -0.7) node {$0$};
%\draw (-.5, 4.75) node {$9$};

%Npos-translation 3,2 from 1/0,1/1, 0/0,0/1,0/2,2/1,
\foreach \x/\y in {3/3,3/4,2/3,2/4,2/5,4/4
}
\filldraw[color = yellow] (\x+.2, \y+.2) rectangle (\x+.8, \y+.8);
%Npos-translation 1,2 from 1/0,1/1, 0/0,0/1,0/2,2/1,
\foreach \x/\y in {2/2,2/3,1/2,1/3,1/4,3/3
}
\filldraw[color = yellow!50!black] (\x+.3, \y+.3)  rectangle (\x+.7, \y+.7);

%Npos translation 3,2 from 6/4,6/5,7/5,5/4,5/5,5/6,4/5,4/6
\foreach \x/\y in {8/7,8/8,9/8,7/7,7/8,7/9,6/8,6/9
}
\filldraw[color = yellow] (\x+.2, \y+.2) rectangle (\x+.8, \y+.8);
%Npos translation 1,2 from 6/4,6/5,7/5,5/4,5/5,5/6,4/5,4/6
\foreach \x/\y in {7/6,7/7,8/7,6/6,6/7,6/8,5/7,5/8
}
\filldraw[color = yellow!50!black] (\x+.3, \y+.3)  rectangle (\x+.7, \y+.7);
%Npos middle segment
\draw (4.5,3.5) node {$ii$};
\draw (3.5,5.5) node {$iii$};
\draw (3.5,2.5) node {$i$};

%Npos lower segment
\draw (16.5,5.5) node {$iii$};
\draw (16.5,6.5) node {$i$};
\draw (16.5,7.5) node {$ii$};
%Npos upper segm
\draw (5.5,16.5) node {$i$};
\draw (6.5,16.5) node {$ii$};
\draw (7.5,16.5) node {$iii$};

 \draw[thin] (0,2) -- (16,20);
 \draw[thin] (1,0) -- (21,16);%(31,24);

\end{tikzpicture}
\end{center}
\caption{The picture illustrates the initial 23 by 21 outcomes of the asymmetric additive ruleset $S =\{(1,2),(2,3),(3,1)\}$, together with the \P-to-\P\ update rules for the middle segment.  At each level $k$, the green cells define the red cells at level $k+1$: \P-to-\P\ rule $(2,3)+(1,2)=(3,5)$. Note that, from each red cell, the move $(3,1)$ (which is not part of the \P-to-\P\ update) leads to an $\N$-position. 
All colors at level $k$ define the green and blue cells at level $k+1$: \P-to-\P\ rule $(2,3)+(3,1)=(5,4)$. Note that, from each green or blue cell, the move $(1,2)$ leads to an $\N$-position. For an explanation of the remaining coloring and numbering, see the proof of Proposition~\ref{prop:OS}.}
\label{fig:tiling_example}
\end{figure}
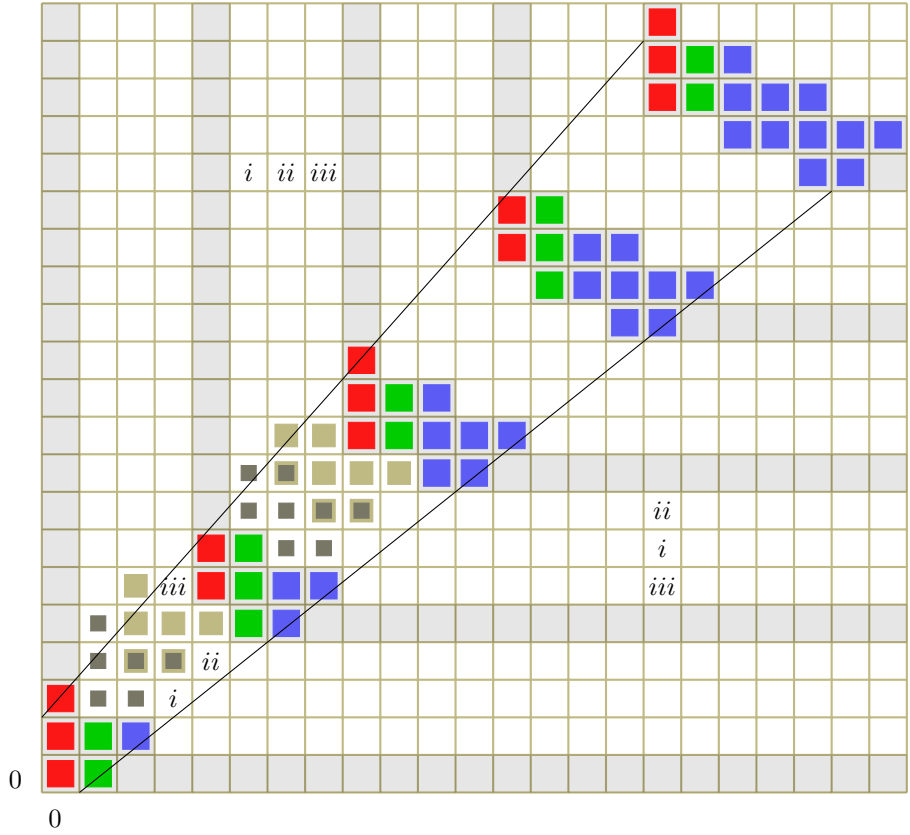

Let us analyze the proposed perfect outcome segmentation from Example~\ref{ex:titlingex}.  
\begin{prop}\label{prop:OS}
    Let  $S =\{(1,2),(2,3),(3,1)\}$. This game has a perfect outcome segmentation given by three sets of update rules together with three initial sets of colored positions.
\end{prop}
\begin{proof}
    The update rules for the two lower segments are obtained by \P-to-\P\ combinations from the ruleset. %The \P-to-\P rule for the upper segment is $(3,1)+(1,2)=(4,3)$. 

    The update rule for the lower segment $\Gamma_0$ is $(2,3)+(3,1) = (5,4)$. The initial coloring corresponds to the terminal \P-positions $\{(2,0),(3,0),\ldots\}$. Since no rule has y-coordinate greater than three, there is no colored position within this segment that can reach another colored position. Hence \P-to-\N property holds. Enumerate the moves by $i:(1,2)$, $ii:(2,3)$ and $iii:(3,1)$. In Figure~\ref{fig:tiling_example}, row-wise generic type $ii, i, iii$ moves have been inserted at the lower segment, for positions with $y$-coordinates $1,2,3\pmod 4$ respectively. Clearly, they reach colored positions. 
    Hence the lower segment is an outcome segment for this ruleset.  

    For the middle segment $\Gamma_1$, we must check that, starting from any colored non-terminal position, the missing rules from the \P-to-\P\ updates reach \N-positions.  Indeed, the initial coloring given in Example~\ref{ex:OS} corresponds to terminal \P-positions, and the caption of Figure~\ref{fig:tiling_example} verifies this: starting from a colored position in the middle segment, the missing rules are smaller than the \P-to-\P\ rules that define the coloring. 
    
    The \N-to-\P~ property of the middle segment is justified by the translations of the colored regions to the two nuances of brown cells: the dark-brown cells are translated by the move $(1,2)$ and the light-brown are translated by the move $(2,3)$. By adapting this brown-coloring scheme at each level, there will be 3 or 2 non-colored cells between each pair of red-green-blue colored regions. Those cells take the same move as described for the lower and upper segments. By decomposing each color update, for $k\ge 2$, in two copies of the shape of level $k-2$. The proof follows by induction, similar to the simpler case in Proposition~\ref{prop:symadd}. We omit further technical details.

    Consider the upper segment $\Gamma_2$. Claim: if we manually color all positions of the form $(8k, 3+9k)$,  $(4+8k,7+9k)$ and $(4+8k,7+9k)$, then the \P-to-\P\ rule $(3,1)+(1,2)=(4,3)$ colors exactly the remaining cells from the rule: ``The upper segment has two alternating update rules $(4,4)$ and $(4,5)$, with $(4,4)$ applied to the initial set $\{(0,y)\mid y\ge 3\}$''. Indeed, if we would apply the \P-to-\P\ rule to those manually colored positions, we would recolor exactly the upper red cells from the coloring of the middle segment. And it is then easy to verify that, starting from a colored position in the upper segment, the missing rule $(2,3)$ reaches \N-positions.  Similar to the lower segment, the upper segment has type $i, ii, iii$ moves from non-colored positions as indicated in the Figure~\ref{fig:tiling_example}, for positions with $x$-coordinates $1,2,3\pmod 4$ respectively. Hence, the upper segment defines correctly the outcomes of the ruleset $S$.

    Finally, by the observation in Example~\ref{ex:lines} it is clear that $\Gamma_0\cap\Gamma_1=\varnothing$, and $\Gamma_1\cap\Gamma_2=\varnothing$.
\end{proof}

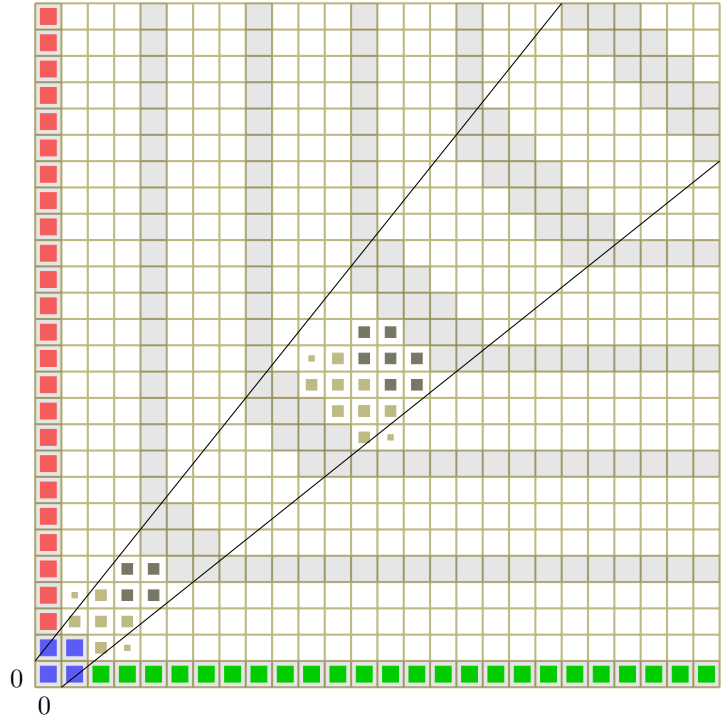
\begin{figure}[htbp!]
\begin{center}
\begin{tikzpicture}[scale = 0.35]
\colorlet{yellow}{gray!60!yellow}
\draw[step=1 cm,yellow, thick] (0, 0) grid (26,26); 
%%%%%%%%%%%%%122133
\foreach \x/\y in {
0/0, 0/1, 0/2, 0/3, 0/4, 0/5, 0/6, 0/7, 0/8, 0/9, 0/10, 0/11, 0/12, 0/13, 0/14, 0/15, 0/16, 0/17, 0/18, 0/19, 0/20, 0/21, 0/22, 0/23, 0/24, 0/25, %0/26, 0/27, 0/28, 0/29, 
1/0, 1/1, 2/0, 3/0, 4/0, 4/5, 4/6, 4/7, 4/8, 4/9, 4/10, 4/11, 4/12, 4/13, 4/14, 4/15, 4/16, 4/17, 4/18, 4/19, 4/20, 4/21, 4/22, 4/23, 4/24, 4/25, %4/26, 4/27, 4/28, 4/29, 
5/0, 5/4, 5/5, 5/6, 6/0, 6/4, 6/5, 7/0, 7/4, 8/0, 8/4, 8/10, 8/11, 8/12, 8/13, 8/14, 8/15, 8/16, 8/17, 8/18, 8/19, 8/20, 8/21, 8/22, 8/23, 8/24, 8/25, %8/26, 8/27, 8/28, 8/29, 
9/0, 9/4, 9/9, 9/10, 9/11, 10/0, 10/4, 10/8, 10/9, 10/10, 11/0, 11/4, 11/8, 11/9, 12/0, 12/4, 12/8, 12/15, 12/16, 12/17, 12/18, 12/19, 12/20, 12/21, 12/22, 12/23, 12/24, 12/25, %12/26, 12/27, 12/28, 12/29, 
13/0, 13/4, 13/8, 13/14, 13/15, 13/16, 14/0, 14/4, 14/8, 14/13, 14/14, 14/15, 15/0, 15/4, 15/8, 15/12, 15/13, 15/14, 16/0, 16/4, 16/8, 16/12, 16/13, 16/20, 16/21, 16/22, 16/23, 16/24, 16/25, %16/26, 16/27, 16/28, 16/29, 
17/0, 17/4, 17/8, 17/12, 17/19, 17/20, 17/21, 18/0, 18/4, 18/8, 18/12, 18/18, 18/19, 18/20, 19/0, 19/4, 19/8, 19/12, 19/17, 19/18, 19/19, 20/0, 20/4, 20/8, 20/12, 20/16, 20/17, 20/18, 20/25, %20/26, 20/27, 20/28, 20/29, 
21/0, 21/4, 21/8, 21/12, 21/16, 21/17, 21/24, 21/25, 22/0, 22/4, 22/8, 22/12, 22/16, 22/23, 22/24, 22/25, 23/0, 23/4, 23/8, 23/12, 23/16, 23/22, 23/23, 23/24, 24/0, 24/4, 24/8, 24/12, 24/16, 24/21, 24/22, 24/23, 25/0, 25/4, 25/8, 25/12, 25/16, 25/20, 25/21, 25/22}%, 26/0, 26/4, 26/8, 26/12, 26/16, 26/20, 26/21, 26/28, 26/29, 27/0, 27/4, 27/8, 27/12, 27/16, 27/20, 27/27, 27/28, 27/29, 28/0, 28/4, 28/8, 28/12, 28/16, 28/20, 28/26, 28/27, 28/28, 29/0, 29/4, 29/8, 29/12, 29/16, 29/20, 29/25, 29/26, 29/27}
\filldraw[color = black, opacity =.1] (\x, \y) rectangle (\x+1, \y+1);%
\draw (-0.7 ,0.35) node {$0$};
%\draw (4, -.5) node {$9$};

\draw (0.35, -0.7) node {$0$};
%\draw (-.5, 4.75) node {$9$};

\foreach \x/\y in {0/0,1/0,0/1,1/1}
 \filldraw[color = blue, opacity =.6] (\x+.2, \y+.2) rectangle (\x+.8, \y+.8);

\foreach \x/\y in {0/2,0/3,0/4,0/5,0/6,0/7,0/8,0/9,0/10,0/11,0/12,0/13,0/14,0/15,0/16,0/17,0/18,0/19,0/20,0/21,0/22,0/23,0/24,0/25}%,0/27,0/28,0/29}
 \filldraw[color = red, opacity =.6] (\x+.2, \y+.2) rectangle (\x+.8, \y+.8);

\foreach \x/\y in {2/0,3/0,4/0,5/0,6/0,7/0,8/0,9/0,10/0,11/0,12/0,13/0,14/0,15/0,16/0,17/0,18/0,19/0,20/0,21/0,22/0,23/0,24/0,25/0}%,27/0,28/0,29/0}
 \filldraw[color = green!80!black] (\x+.2, \y+.2) rectangle (\x+.8, \y+.8);

%Npos translation 1,2 and 2,1 from 0/0,0/1,1/0,1/1
\foreach \x/\y in {1/2,2/2,2/3,2/1,3/2
}
\filldraw[color = yellow] (\x+.3, \y+.3) rectangle (\x+.7, \y+.7);
\foreach \x/\y in {1/3,3/1
}
\filldraw[color = yellow] (\x+.4, \y+.4) rectangle (\x+.6, \y+.6);
%Npos translation 3,3 from 0/0,0/1,1/0,1/1
\foreach \x/\y in {3/3,3/4,4/3,4/4
}
\filldraw[color = yellow!50!black] (\x+.3, \y+.3)  rectangle (\x+.7, \y+.7);
%%%%%%%%%Here the inductive argument will be displayed
%Npos translation 9,12 and 10,11 from 0/0,0/1,1/0,1/1
\foreach \x/\y in {10/11,11/11,11/12,12/9,12/10,13/10,11/10,12/10,12/11
}
\filldraw[color = yellow] (\x+.3, \y+.3) rectangle (\x+.7, \y+.7);
\foreach \x/\y in {10/12,13/9
}
\filldraw[color = yellow] (\x+.4, \y+.4) rectangle (\x+.6, \y+.6);
%Npos translation 3,3 from 0/0,0/1,1/0,1/1
\foreach \x/\y in {12/12,12/13,13/12,13/13,14/12,13/11,14/11
}
\filldraw[color = yellow!50!black] (\x+.3, \y+.3)  rectangle (\x+.7, \y+.7);

 \draw[thin] (0,1) -- (20,26);
 \draw[thin] (1,0) -- (26,20);

\end{tikzpicture}
\end{center}
\caption{The picture illustrates the initial 26 by 26 outcomes of the symmetric additive game $S =\{(1,2),(2,1),(3,3)\}$, together with the initial coloring for the \P-to-\P\ update rules: blue updates to blue by using rules $(4,5)$ and $(5,4)$, red updates to red by using rule $(5,4)$ and green updates to green by using rule $(4,5)$. Hence, exactly all the gray cells will be colored. The yellow and brown cells are used in the proof of Proposition~\ref{prop:symadd}.}
\label{fig:122133color}
\end{figure}

We generalize the left-most ruleset in Figure~\ref{fig:additive}. 
Let us prove that the update rules in Figure~\ref{fig:122133color} are correct. The result holds in a somewhat more general statement, a symmetric restriction of Definition~\ref{def:addsym}. 
  
\begin{prop}[{\sc symmetric additive}]\label{prop:symadd}
    Consider symmetric additive rulesets of the form \\
    $S=\{(a,b),(b,a),(a+b,b+a)\}$,
    where $a\ge b/2$.  The set of $\P$-positions is given by the following updates. The set of terminal positions is $T_L\cup T_U\cup T_M$, where $T_L=\{(x,y)\mid y<a\}$, $T_U=\{(x,y)\mid x<a\}$ and $T_M=\{(x,y)\mid x< b, y< b\}$. Then, update the middle segment by the rules $r_1=(2a+b,2b+a)$ and $r_2=(2b+a,2a+b)$, by initializing with $T_M$. Update the lower (upper) segment via the rule $r_1$ ($r_2$), by initializing using the lower and upper terminal positions respectively.  
\end{prop}
\begin{proof}
We prove it for the case $S=\{(1,2),(2,1),(3,3)\}$. Similar to Section~\ref{sec:2dresemble1d}, one can check that the method of proof generalizes to $a<b\le 2a$. Clearly the initial coloring $\{(x,y) \mid x=0 \text{ or } y=0  \text{ or } x=y=1\}$ as in Figure~\ref{fig:additive} corresponds to the set of terminal \P-positions. Suppose that the updates are correct (i.e., colored if and only if a \P-position) for all positions $\x<\x'$. We must verify that $\x'\in\P$ if and only if cell $\x'$ has been colored. By symmetry, it suffices to justify {\bf green} and {\bf blue}. Suppose first that $\mathrm{color}(\x')=\mathbf{green}$. It was updated from {\bf green}, which is a combination of the moves $(2,1)$ and $(3,3)$. By induction, the cell $\x'-(5,4)+(4,2)=\x'-(1,2)\in \N$. Therefore, the converse of Lemma~\ref{lem:three_move_PtoP} (iii) implies that $\x'\in\P$. 

We claim that it suffices to prove that every position below the line of slope $4/5$ with offset $(1,0)$, and between the terminal green cells and the first green update, is an $\N$-position. Namely, induction can be used as follows: suppose $\x'$ is non-colored below the line of slope $4/5$ with offset $(1,0)$. Then, by induction $\x'-(5,4)\in \N$. Hence, there is a move  $\x'-(5,4)-\s\in \P$. This same move can be used from $\x'$, and clearly, $\x'-\s\in \mathbf{green}$, which by induction is in \P. 

To prove the claim, notice that the leftmost cells in that region are $(3,1)$, $(4,2)$, and $(5,3)$. And $(3,1)-(2,1)\in \P$, $(4,2)-(1,2)\in \P$ and $(5,3)-(3,3)\in\P$. Similarly, every position in the set $\{(x,y) \mid 0<y<4, x\ge y+2\}$ has a move to a green terminal \P-position.

To complete the proof, we next study the middle region. One can check manually that the first non-colored region is a subset of $\N$. The proof continues by induction. If the outcomes of the middle segment are correct until the $k$th update, then we translate the $\N$-positions between the $(k-1)$th and $k$th updates, by using the updates $(4,5)$ and $(5,4)$. These translated positions will be $\N$-positions using the same moves as on the lower level. We have illustrated a way of inductive proof that does not depend on the symmetry of the move set and hence can be used more generally. The yellow and brown cells in the lower left corner are \N-positions by a base case computation. The yellow and brown cells above the third update level of the coloring automaton are copy-pasted from the level just below and are \N-positions by induction. The remaining cells at that level (between colored cell levels) are all covered by the \N-positions from the base case, namely the terminal blue cells correspond to the `difference' of the inductive case and the level to be proved.  
\end{proof}

The word `difference', in the final paragraph of the proof, allows for some overlap, as is often the case of the coloring automaton. 

The cases where $a<b/2$ have more complexity as is evident by Figure~\ref{fig:1551add2552}; see Section~\ref{sec:picopen} for further observations.

\begin{figure}[htbp!]
\begin{center}
\includegraphics[width=5.5cm]{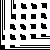}\hspace{5 mm}
\includegraphics[width=5.5cm]{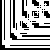}
\end{center}
\caption{These variations of {\sc symmetric additive three-move}, ($a=1, b=5$ and $a=2, b=5$ respectively) have more complexity than the case $a\ge b/2$.
}
\label{fig:1551add2552}
\end{figure}

%\subsection{An additive rule built on an arithmetic progression}
In view of Figure~\ref{fig:coloring123446}, we define a family of {\sc two-dimensional subtraction} that depends on just one parameter. 
\begin{defi}[{\sc arithmetic additive}]
Let $a\ge 1$. Then 
{\small $
S=\{(a,2a),(3a,4a),(4a,6a)\}
$ }is an {\sc arithmetic additive} ruleset. 
\end{defi}
We leave it as an open problem to verify that the coloring scheme updates in Figure~\ref{fig:coloring123446} correspond to the middle segment's \P-positions of {\sc arithmetic additive} $S=\{(1,2),(3,4),(4,6)\}$.  
\begin{figure}[htp!]
\begin{center}
\begin{tikzpicture}[scale = 0.22]
\colorlet{yellow}{gray!60!yellow}
\draw[step=1 cm,yellow] (0, 0) grid (50,50); 
%%%%%%%%%%%%%122133
\foreach \x/\y in {
0/0, 0/1, 0/2, 0/3, 0/4, 0/5, 0/6, 0/7, 0/8, 0/9, 0/10, 0/11, 0/12, 0/13, 0/14, 0/15, 0/16, 0/17, 0/18, 0/19, 0/20, 0/21, 0/22, 0/23, 0/24, 0/25, 0/26, 0/27, 0/28, 0/29, 0/30, 0/31, 0/32, 0/33, 0/34, 0/35, 0/36, 0/37, 0/38, 0/39, 0/40, 0/41, 0/42, 0/43, 0/44, 0/45, 0/46, 0/47, 0/48, 0/49, 1/0, 1/1, 2/0, 2/1, 2/4, 2/5, 2/6, 2/7, 2/8, 2/9, 2/10, 2/11, 2/12, 2/13, 2/14, 2/15, 2/16, 2/17, 2/18, 2/19, 2/20, 2/21, 2/22, 2/23, 2/24, 2/25, 2/26, 2/27, 2/28, 2/29, 2/30, 2/31, 2/32, 2/33, 2/34, 2/35, 2/36, 2/37, 2/38, 2/39, 2/40, 2/41, 2/42, 2/43, 2/44, 2/45, 2/46, 2/47, 2/48, 2/49, 3/0, 3/1, 4/0, 4/1, 5/0, 5/1, 6/0, 6/1, 6/8, 6/9, 7/0, 7/1, 7/8, 7/9, 7/12, 7/13, 7/14, 7/15, 7/16, 7/17, 7/18, 7/19, 7/20, 7/21, 7/22, 7/23, 7/24, 7/25, 7/26, 7/27, 7/28, 7/29, 7/30, 7/31, 7/32, 7/33, 7/34, 7/35, 7/36, 7/37, 7/38, 7/39, 7/40, 7/41, 7/42, 7/43, 7/44, 7/45, 7/46, 7/47, 7/48, 7/49, 8/0, 8/1, 8/8, 8/9, 8/12, 8/13, 9/0, 9/1, 9/8, 9/9, 9/16, 9/17, 9/18, 9/19, 9/20, 9/21, 9/22, 9/23, 9/24, 9/25, 9/26, 9/27, 9/28, 9/29, 9/30, 9/31, 9/32, 9/33, 9/34, 9/35, 9/36, 9/37, 9/38, 9/39, 9/40, 9/41, 9/42, 9/43, 9/44, 9/45, 9/46, 9/47, 9/48, 9/49, 10/0, 10/1, 10/8, 10/9, 11/0, 11/1, 11/8, 11/9, 12/0, 12/1, 12/8, 12/9, 12/16, 12/17, 13/0, 13/1, 13/8, 13/9, 13/16, 13/17, 13/20, 13/21, 14/0, 14/1, 14/8, 14/9, 14/16, 14/17, 14/20, 14/21, 14/24, 14/25, 14/26, 14/27, 14/28, 14/29, 14/30, 14/31, 14/32, 14/33, 14/34, 14/35, 14/36, 14/37, 14/38, 14/39, 14/40, 14/41, 14/42, 14/43, 14/44, 14/45, 14/46, 14/47, 14/48, 14/49, 15/0, 15/1, 15/8, 15/9, 15/16, 15/17, 15/24, 15/25, 16/0, 16/1, 16/8, 16/9, 16/16, 16/17, 16/28, 16/29, 16/30, 16/31, 16/32, 16/33, 16/34, 16/35, 16/36, 16/37, 16/38, 16/39, 16/40, 16/41, 16/42, 16/43, 16/44, 16/45, 16/46, 16/47, 16/48, 16/49, 17/0, 17/1, 17/8, 17/9, 17/16, 17/17, 18/0, 18/1, 18/8, 18/9, 18/16, 18/17, 18/24, 18/25, 19/0, 19/1, 19/8, 19/9, 19/16, 19/17, 19/24, 19/25, 19/28, 19/29, 20/0, 20/1, 20/8, 20/9, 20/16, 20/17, 20/24, 20/25, 20/28, 20/29, 20/32, 20/33, 21/0, 21/1, 21/8, 21/9, 21/16, 21/17, 21/24, 21/25, 21/32, 21/33, 21/36, 21/37, 21/38, 21/39, 21/40, 21/41, 21/42, 21/43, 21/44, 21/45, 21/46, 21/47, 21/48, 21/49, 22/0, 22/1, 22/8, 22/9, 22/16, 22/17, 22/24, 22/25, 22/36, 22/37, 23/0, 23/1, 23/8, 23/9, 23/16, 23/17, 23/24, 23/25, 23/40, 23/41, 23/42, 23/43, 23/44, 23/45, 23/46, 23/47, 23/48, 23/49, 24/0, 24/1, 24/8, 24/9, 24/16, 24/17, 24/24, 24/25, 24/32, 24/33, 25/0, 25/1, 25/8, 25/9, 25/16, 25/17, 25/24, 25/25, 25/32, 25/33, 25/36, 25/37, 26/0, 26/1, 26/8, 26/9, 26/16, 26/17, 26/24, 26/25, 26/32, 26/33, 26/36, 26/37, 26/40, 26/41, 27/0, 27/1, 27/8, 27/9, 27/16, 27/17, 27/24, 27/25, 27/32, 27/33, 27/40, 27/41, 27/44, 27/45, 28/0, 28/1, 28/8, 28/9, 28/16, 28/17, 28/24, 28/25, 28/32, 28/33, 28/44, 28/45, 28/48, 28/49, 29/0, 29/1, 29/8, 29/9, 29/16, 29/17, 29/24, 29/25, 29/32, 29/33, 29/48, 29/49, 30/0, 30/1, 30/8, 30/9, 30/16, 30/17, 30/24, 30/25, 30/32, 30/33, 30/40, 30/41, 31/0, 31/1, 31/8, 31/9, 31/16, 31/17, 31/24, 31/25, 31/32, 31/33, 31/40, 31/41, 31/44, 31/45, 32/0, 32/1, 32/8, 32/9, 32/16, 32/17, 32/24, 32/25, 32/32, 32/33, 32/40, 32/41, 32/44, 32/45, 32/48, 32/49, 33/0, 33/1, 33/8, 33/9, 33/16, 33/17, 33/24, 33/25, 33/32, 33/33, 33/40, 33/41, 33/48, 33/49, 34/0, 34/1, 34/8, 34/9, 34/16, 34/17, 34/24, 34/25, 34/32, 34/33, 34/40, 34/41, 35/0, 35/1, 35/8, 35/9, 35/16, 35/17, 35/24, 35/25, 35/32, 35/33, 35/40, 35/41, 36/0, 36/1, 36/8, 36/9, 36/16, 36/17, 36/24, 36/25, 36/32, 36/33, 36/40, 36/41, 36/48, 36/49, 37/0, 37/1, 37/8, 37/9, 37/16, 37/17, 37/24, 37/25, 37/32, 37/33, 37/40, 37/41, 37/48, 37/49, 38/0, 38/1, 38/8, 38/9, 38/16, 38/17, 38/24, 38/25, 38/32, 38/33, 38/40, 38/41, 38/48, 38/49, 39/0, 39/1, 39/8, 39/9, 39/16, 39/17, 39/24, 39/25, 39/32, 39/33, 39/40, 39/41, 39/48, 39/49, 40/0, 40/1, 40/8, 40/9, 40/16, 40/17, 40/24, 40/25, 40/32, 40/33, 40/40, 40/41, 40/48, 40/49, 41/0, 41/1, 41/8, 41/9, 41/16, 41/17, 41/24, 41/25, 41/32, 41/33, 41/40, 41/41, 41/48, 41/49, 42/0, 42/1, 42/8, 42/9, 42/16, 42/17, 42/24, 42/25, 42/32, 42/33, 42/40, 42/41, 42/48, 42/49, 43/0, 43/1, 43/8, 43/9, 43/16, 43/17, 43/24, 43/25, 43/32, 43/33, 43/40, 43/41, 43/48, 43/49, 44/0, 44/1, 44/8, 44/9, 44/16, 44/17, 44/24, 44/25, 44/32, 44/33, 44/40, 44/41, 44/48, 44/49, 45/0, 45/1, 45/8, 45/9, 45/16, 45/17, 45/24, 45/25, 45/32, 45/33, 45/40, 45/41, 45/48, 45/49, 46/0, 46/1, 46/8, 46/9, 46/16, 46/17, 46/24, 46/25, 46/32, 46/33, 46/40, 46/41, 46/48, 46/49, 47/0, 47/1, 47/8, 47/9, 47/16, 47/17, 47/24, 47/25, 47/32, 47/33, 47/40, 47/41, 47/48, 47/49, 48/0, 48/1, 48/8, 48/9, 48/16, 48/17, 48/24, 48/25, 48/32, 48/33, 48/40, 48/41, 48/48, 48/49, 49/0, 49/1, 49/8, 49/9, 49/16, 49/17, 49/24, 49/25, 49/32, 49/33, 49/40, 49/41, 49/48, 49/49}
\filldraw[color = black, opacity =.1] (\x, \y) rectangle (\x+1, \y+1);%
\draw (-0.7 ,0.35) node {$0$};
%\draw (4, -.5) node {$9$};

\draw (0.35, -0.7) node {$0$};
%\draw (-.5, 4.75) node {$9$};

\foreach \x/\y in {0/0,0/1,7/12,7/13,14/24,14/25,21/36,21/37,28/48,28/49}%,1/0,1/1,2/0,2/1}
 \filldraw[color = red!75!blue] (\x+.2, \y+.2) rectangle (\x+.8, \y+.8);

\foreach \x/\y in {6/8,6/9,8/12,8/13,12/16,12/17,14/20,14/21,13/20,13/21,15/24,15/25,20/28,20/29,19/28,19/29,21/32,21/33,20/32,20/33,18/24,18/25,22/36,22/37,25/36,25/37,26/36,26/37,24/32,24/33,31/44,31/45,32/44,32/45,26/40,26/41,29/48,29/49,30/40,30/41,36/48,36/49,27/40,27/41,32/48,32/49,33/48,33/49,27/44,27/45,28/44,28/45}
 \filldraw[color = blue, opacity =.7] (\x+.2, \y+.2) rectangle (\x+.8, \y+.8);

 \foreach \x/\y in {2/4,2/5,9/16,9/17,16/28,16/29,23/40,23/41}
 \filldraw[color = green!80!black] (\x+.2, \y+.2) rectangle (\x+.8, \y+.8);

\end{tikzpicture}
\end{center}
\caption{The picture illustrates the initial  outcomes (gray) of {\sc arithmetic additive}  $S =\{(1,2),(3,4),(4,6)\}$, together with its coloring for the \P-to-\P\ update rules of the middle segment. $\mathrm{purple}\rightarrow \mathrm{green}(2,4),\mathrm{blue}(6,8)$, $\mathrm{green}\rightarrow \mathrm{ purple}(5,8),\mathrm{blue}(6,8)$ and $\mathrm{blue}\rightarrow \mathrm{blue}(6,8)$. The lower segment has update rule {\rm blue}, with initial coloring $\{(x,0),(x,1)\mid x\ge 1\}$. The upper segment has two update rules: $\mathrm{red}\rightarrow\mathrm{yellow}(2,4)$, $\mathrm{yellow}\rightarrow\mathrm{red}(5,8)$ with initial red coloring $\{(0,y)\mid y\ge 2\}$. The middle segment is bounded away from the lower (upper) segment by a line of slope $4/3$ ($(4+8)/(2+5)=12/7$).}
\label{fig:coloring123446}
\end{figure}
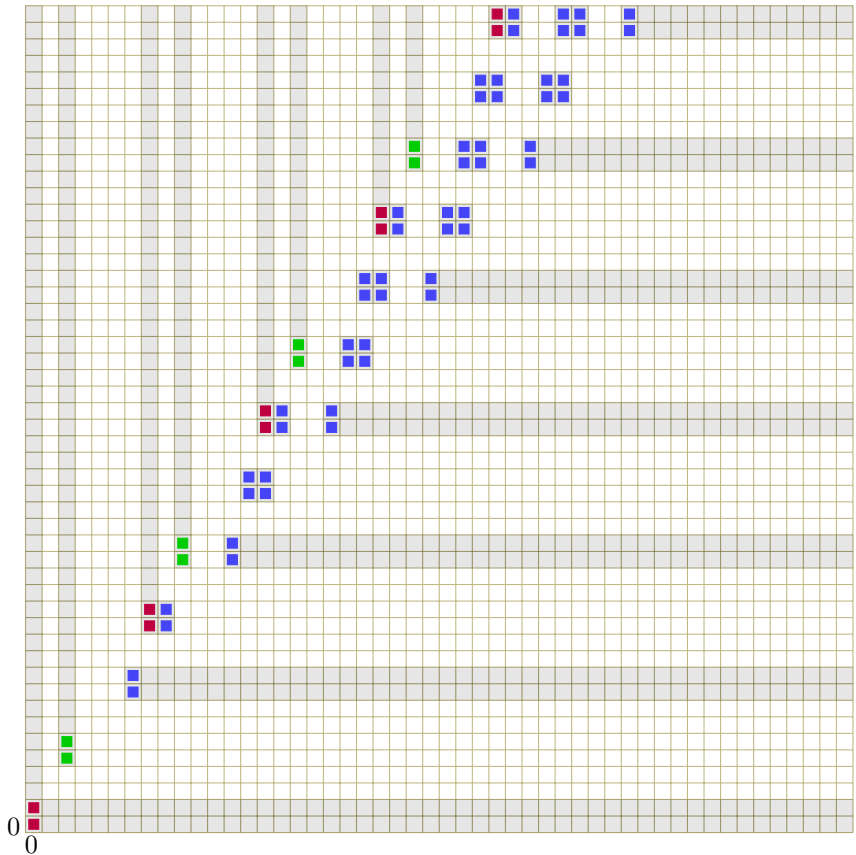

\section{Some open problems}
\label{sec:picopen}
A classification problem emerges from this experimental section. For a given $k\in \n$, which rulesets have exactly $k$ outcome segments? We have run some c-code and found instances for $k=1,\ldots ,6$. See the pictures in this section for $k>3$. Figure~\ref{fig:symasym} illustrates two 4-segmentations (a symmetric and an asymmetric) where the segmentation settles very early. In contrast, we display Figures~\ref{fig:3move4segm_thin} and \ref{fig:3move4segm_fat}, where the eventual segmentation settles a lot slower. In the latter case, a thick part of the middle region turns out not to shape a segment, but a semi-infinite (periodic) strip. We do not currently have any argument or intuition to point out why the latter of those two rulesets has a much wider (lower) middle segment.    Figures~\ref{fig:5segm} and \ref{fig:6segm} display a 5 and a 6-segmentation, respectively, that satisfy the hypothesis of Conjecture~\ref{conj:maxsegm}.  In Figures~\ref{fig:chaos} and \ref{fig:chaos2}, we display a ruleset for which we believe that ``outcome segmentation'' is not the correct concept.

Recall the pictures in Figure~\ref{fig:1551add2552}. We show a bit further computation of those pictures in Figure~\ref{fig:1525_200}. By visual inspection, it is easy to see that both settle in periodic patterns. In view of an early question: Yes, we do get some more complexity if $2a<b$, but the patterns  do not seem to transform into irregularity. 
\begin{figure}[htp!]
\begin{center}
\includegraphics[width=7cm]{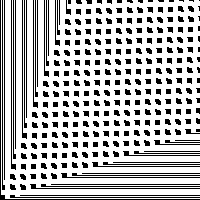}\hspace{2 mm}
\includegraphics[width=7cm]{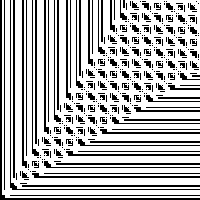}
\end{center}
\caption{Recall the variations of {\sc symmetric additive three-move} from Figure~\ref{fig:1551add2552}. This time we have computed the outcomes of the first 200 by 200 positions. Visual inspection proves periodic patterns. 
}
\label{fig:1525_200}
\end{figure}
These examples together with many other similar experimental results lead us to the following conjecture. 

\begin{conj}[{\sc three-move} Segmentation]\label{conj:threesegm}
If $|S|=3$, there is an outcome segmentation.
\end{conj}

\begin{conj}[Maximum Segmentation]\label{conj:maxsegm}
 For each $k\in \n$, $k\ne 2$, there exists a ruleset of size $k$ that has a $(k+1)$-segmentation. For each $k\in \n$, there is no ruleset of size $k$ that has a $(k+2)$-segmentation.
\end{conj}

\begin{figure}[htbp!]
 \begin{center}
 \includegraphics[width=9.9cm]{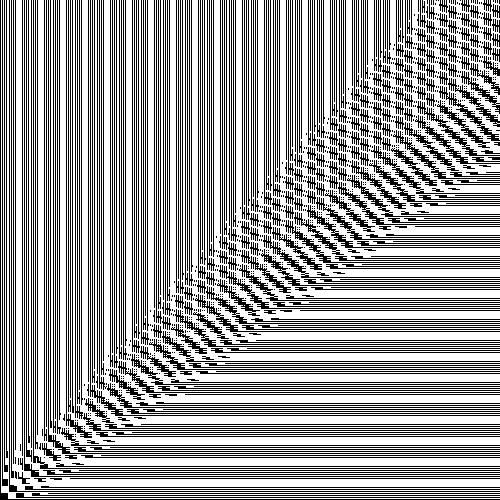}
 \vspace{2 mm}
  \includegraphics[width=9.9cm]{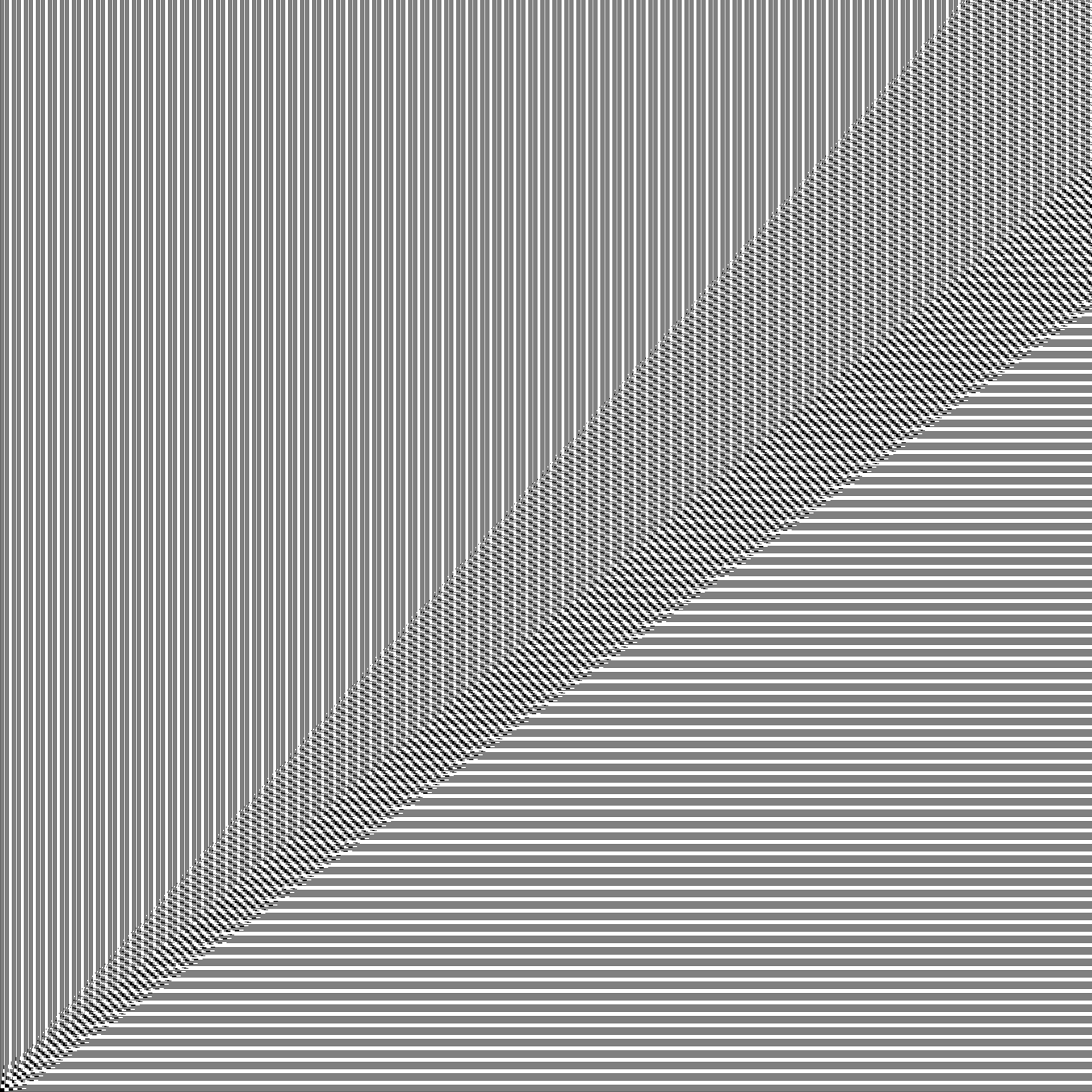}
 \end{center}
 \caption{A delayed 4-segmentation arising out of the three-move game $S=\{(1,7),(8,1),(14,14)\}$, 500 by 500 positions on top of 2000 by 2000 positions.}
 \label{fig:3move4segm_thin}
 \end{figure}
 
\begin{figure}[htbp!]
 \begin{center}
\includegraphics[width=9.9cm]{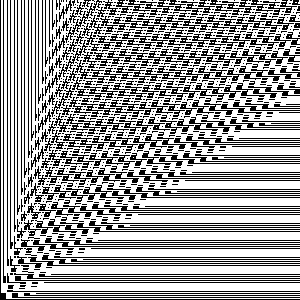}
\vspace{2 mm} \includegraphics[width=9.9cm]{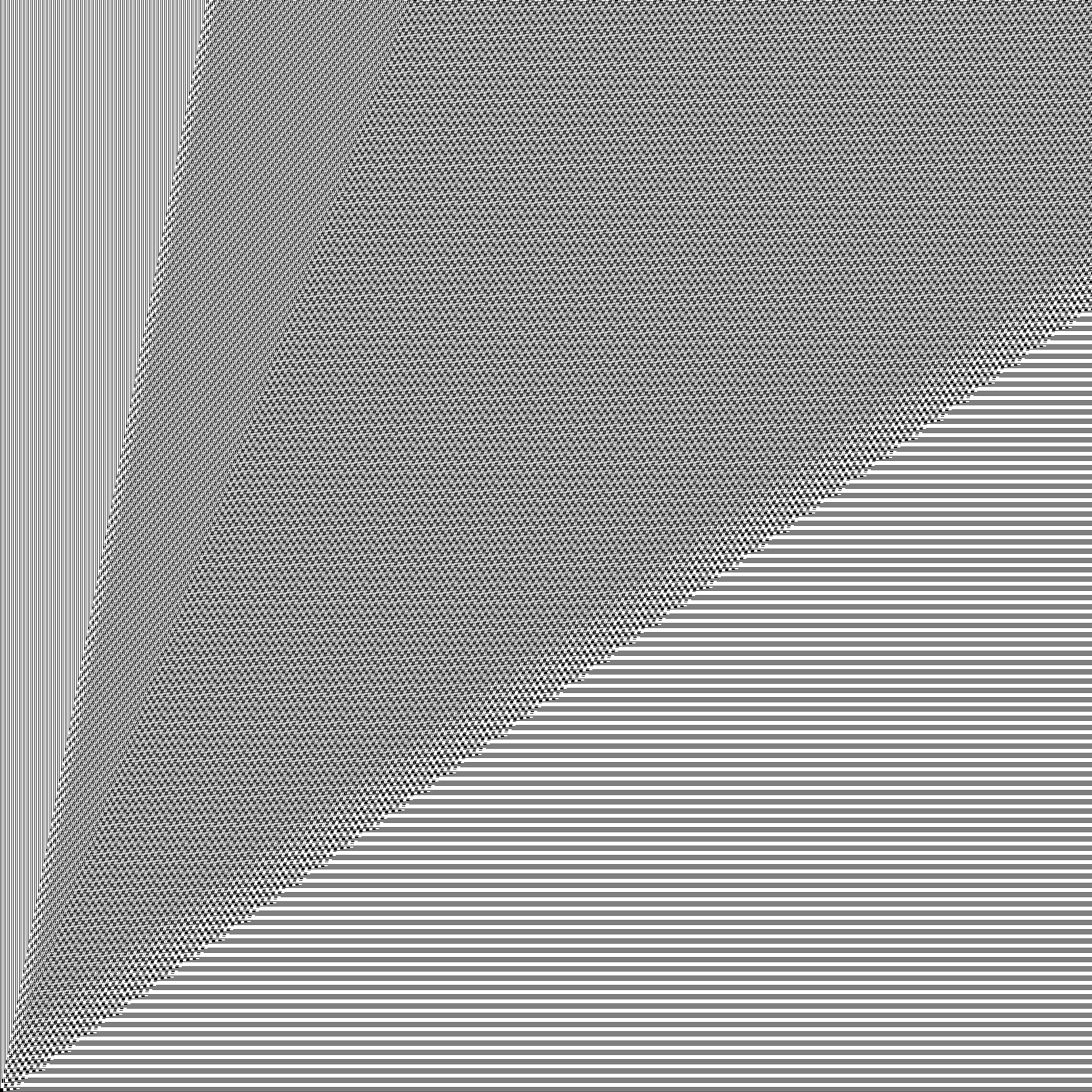}
 \end{center}
 \caption{A 4-segmentation arising out of the three-move game $S=\{(1,7),(2,10),(6,1)\}$, 300 by 300 positions on top of 2000 by 2000 positions.}
 \label{fig:3move4segm_fat}
 \end{figure}
 
 \begin{figure}[htbp!]
 \begin{center}
 \includegraphics[width=\textwidth]{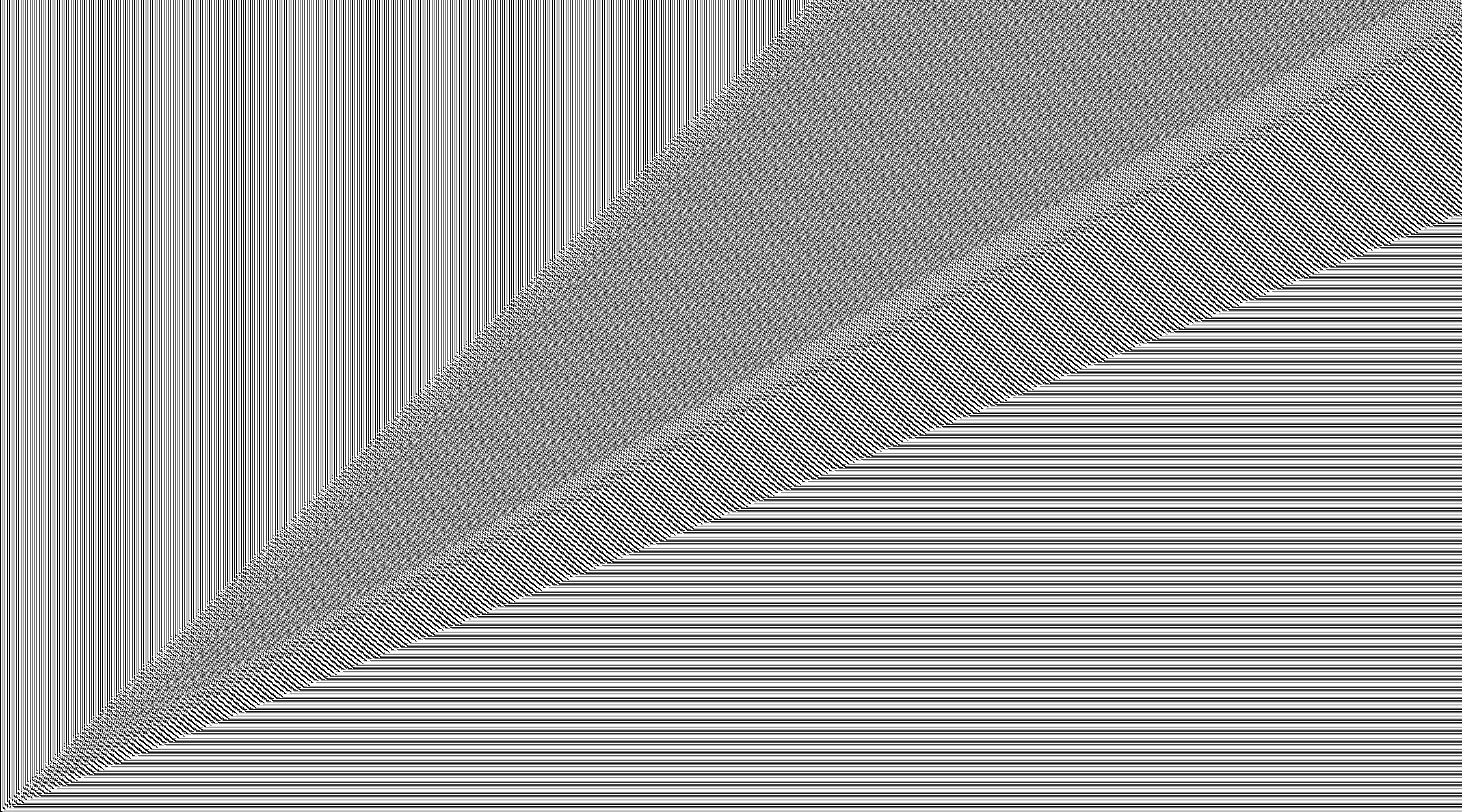}
 \end{center}
 \caption{A 5-segmentation arising out of the 4-move game $S=\{(2,6),(3,3),(6,1),(19,6)\}$ (3600 by 2000 positions).}
 \label{fig:5segm}
 \end{figure}
 \begin{figure}[htbp!]
 \begin{center}
 \includegraphics[width=\textwidth]{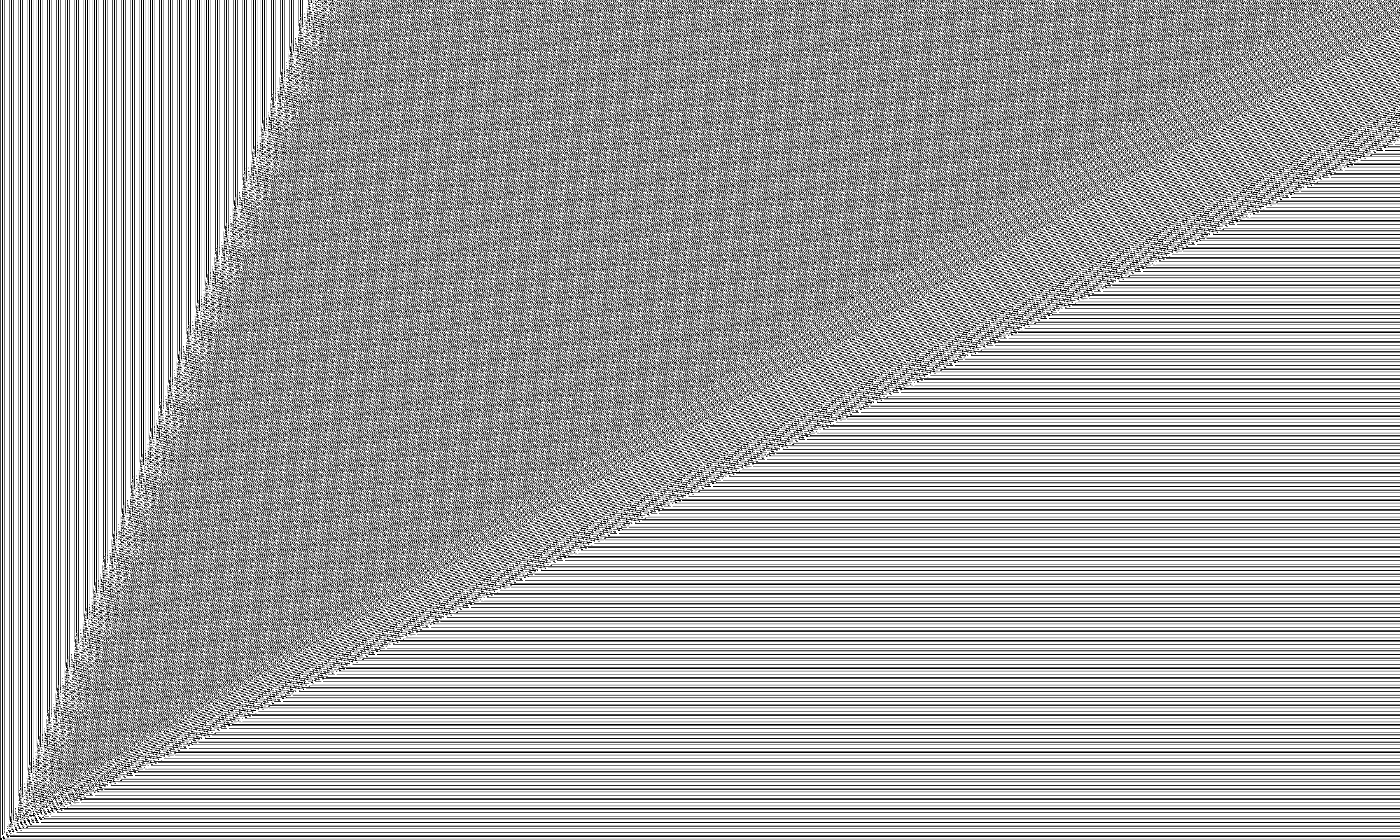}
 \end{center}
 \caption{A 6-segmentation arising out of the 5-move game $S=\{(2,6),(4,11),(6,1),(6,3),(19,5)\}$ (5000 by 3000 positions).}
 \label{fig:6segm}
 \end{figure}

Some rulesets seem harder than others. We have computed the outcomes of the game $$S=\{(0,1), (1,0), (1,1), (1,2), (2,2), (2,51), (4,3), (4,4), (13,1)\}$$ up to 20000 by 20000; see Figure~\ref{fig:chaos} for the first few outcomes. And the patterns do not seem to stabilize; see Figure~\ref{fig:chaos2}. Why does every {\em glider} in the picture have the same slope? If true, this would contribute to a simpler solution, than if the slopes of outgoing gliders were different (as for example in the rule 110 CA updates)? At least, a linear search (of bounded with) would suffice to identify the full two-dimensional pattern in the region of the gliders. Still, an arbitrary line of rational slope does not appear to define periodic outcomes, so the ruleset is not 2-d eventually periodic (see Definition~\ref{def:2devper}). This ruleset certainly seems more complicated than those that satisfy outcome segmentation.

The next problem concerns the outcome patterns. 
\begin{problem}\label{prob:percolation}
This is a question related to percolation. When does an outcome segment for given rulesets {\em \N-percolate} (defined as an infinite path of connected \N-positions)? For example, in Figure~\ref{fig:additive},  the middle segment of the right-most picture satisfies \N-percolation whereas the left-most does not. In Figure~\ref{fig:3move4segm_fat} both middle segments percolate, whereas in Figure~\ref{fig:3move4segm_thin} only the upper middle segment percolates.
\end{problem}

In all our examples where we define a coloring scheme, we were able to set the initial colors at the terminal \P-positions of the ruleset. We have not yet studied a more advanced case, such as in the 4-segmentations in Figures~ ~\ref{fig:3move4segm_thin} and~\ref{fig:3move4segm_fat}, where the patterns that (visibly) define the outcome segments start later.

\begin{problem}
    Study a coloring scheme for a ruleset where outcome segmentation does not start at the origin, e.g. Figures~\ref{fig:3move4segm_thin} and~\ref{fig:3move4segm_fat}.
\end{problem}

The following generalizes the notion of ``symmetric ruleset'' as defined in \cite{Flammenkamp_1997} to several dimensions. A ruleset $S$ is {\em max-symmetric} if there is a unique maximum $\s_{\max}$ such that, if $\s\in S\setminus\{\s_{\max}\}$, then $\s_{\max}-\s\in S$. Such rulesets have fewer \P-to-\P~ update rules, which seems to weaken tendencies to pattern formation.
See Figures~\ref{fig:maxsym1},~\ref{fig:maxsym2},~ \ref{fig:maxsym3} and \ref{fig:maxsym4} for some examples. Note that the same-slope-glider formation in Figure~\ref{fig:chaos2} does not seem to happen for {\sc five-move max-symmetric}. There appear to be two outer outcome segments. However, the middle segments do not stabilize to any visible regularity. 
We write $S_{\mathrm{ms}}$ for a max-symmetric ruleset, where max-symmetric moves have been omitted. For example $S=\{(1,7),(7,1),(9,3),(3,9),(10,10)\}=\{(1,7),(7,1),(10,10)\}_{\mathrm{ms}}$.
\begin{problem}
    Study max-symmetric rulesets.
\end{problem}

Guo and Miller \cite{guo2011lattice} develop a generalization of disjunctive sum called {\em square free} lattice games. Their heap monoids count the number of heaps of each size, say $x_i=$ `the number of heaps of size $0\le i\le n$'. By moving, at least one of the $x_i$s must decrease, and in particular the count of the maximum heap size that is affected by a move must decrease (the number of smaller heaps of each size may increase). A ruleset is square free if no $x_i$ can decrease by more than one unit in a single move. For example, if there are two heaps of size 7 then after moving, at least one of these heaps remains unaffected. 
Within their framework, this turns out to be a very successful restriction, when it comes to normal-play. (Their framework generalizes both mis\`ere and normal-play.) Namely in this case, they prove that the set of \P-positions is a finitely generated affine stratification. And in more generality, they conjecture that, provided the lattice game is square free, but the set of so-called `defeated positions' may vary, then this continues to hold. From our perspective, and within our (quite different) framework, the coloring schemes may be viewed as some type of `affine stratification' on the set of \P-positions. But we do not yet have any simple proviso that gives a fairly general class of rulesets for which the coloring schemes provide a finite `affine stratification'. Note however that lattice games, in our terms, are non-invariant rulesets. Namely, the removal from a heap may depend on whether there are other heaps of the same size. Moreover, their square free restriction is natural because they study heap monoids. Such a restriction may be imposed in our setting as well (if we pick a large dimension), but it feels quite artificial. Hence, we cannot hope to easily adapt any of their techniques. But there are also similarities. Let us phrase an open problem, which is inspired by this discussion.

\begin{problem}
\label{prob:proviso}
    Develop a simple proviso for {\sc two-dimensional finite vector subtraction} for which coloring schemes give well defined outcome segmentations.
\end{problem}
For example, could we have a proviso $|S|=3$?

\begin{problem}
\label{prob:OS}
Give necessary and/or sufficient conditions such that a finite ruleset $S$ forms an outcome segmentation.
\end{problem}

There are several other differences in our frameworks; in particular Guo and Miller \cite{guo2011lattice} allow splitting of heaps similar to octal games e.g. {\sc Kayles} and {\sc Dawson's Chess}. 
And they generalize both mis\`ere-  and normal-play, etc. Their heap monoids are capable to model any impartial ruleset, in terms of lattice game quotients (similar to nimbers that model any normal play impartial game). In this spirit, recall also the finitely generated mis\`ere quotients by Plambeck and Siegel \cite{PlambeckSiegel2008}. We do not aspire such generality, but we simply want to learn more about the outcomes of  {\sc vector subtraction}. For example, is there a difference between two-dimensional eventual periodicity of outcomes and outcome segmentation? 

%%%%%%%%%%%%%%%%%%%%%%%%%%
 \begin{defi}[2-d Eventual Periodicity]\label{def:2devper}
 \label{def:halflineperiodicity}
 Consider a two-dimensional ruleset $S$, let $p,q$ be given integers, with $p,q>0$, and let $m$ be a given rational. Let $f(x)=px/q+m$ be a function defined on the integers. If $(x, f(x))$ has infinitely many lattice points, then the outcomes on this (discrete) line are eventually  periodic if  there exist  $T >0$ and   $x'\ge 0$ with the following properties 
 \begin{itemize}
\item  $f(x+T)$ is an integer whenever $f(x)$ is an integer.  
\item for all $x\ge x'$, $o(x,f(x)) = o(x+T, f(x+T))$.
\end{itemize}
     If all such lines have eventually periodic outcomes, then the ruleset $S$ is eventually periodic.  If, for all $p,q,m$, both $T$ and $x'$ can be upper bounded, then $S$ is strongly eventually periodic.
\end{defi}

\begin{conj}
\label{conj:periodic}
Not all finite 2-d rulesets are eventually periodic.
\end{conj}

\begin{problem}
\label{Prob:regularitystrong}
    Are all 2-d eventually periodic rulesets strongly eventually periodic?
\end{problem}

\begin{problem}
\label{Prob:regularityprop}
    Prove some regularity property, such as 2-d eventual periodicity, on coloring schemes.
\end{problem}

\begin{problem}
\label{probcolorscheme}
Why does a coloring scheme give borders of rational slopes?
\end{problem}

%%%%%%%%%%%% Line Detection Discussion%%%%%%%%%%%%%%%%%%
%As mentioned, experimental observations suggest some rulesets have regular outcomes. 
Motivated by the many observations of outcome segmentation, Adhikari  \cite{adhikari2024} proposes an algorithm that determines the slopes of the boundaries of the outcome segments. The algorithm uses variance and entropy filters, and a Hough line transform, and parameters have to be set manually to explore any given boundary. 

%The algorithm first uses variance and entropy filters to detect edges in the images by appropriately choosing the parameters of filters. Then, it applies Hough line transform techniques to identify prominent straight lines and determine their slopes. We pose a open problem related to detection of outcome segments.

\begin{problem}
    Given an outcome segmentation, develop an algorithm that automatically determines the number of outcome segments together with the corresponding boundary slopes. %and slopes of the straight lines for a given image of two-dimensional rulesets outcome.
\end{problem}

\begin{problem}
    Develop an algorithm that gives a  yes/no answer to the question:  does the ruleset generate an outcome segmentation?
\end{problem}

Of course, there might not exist such an algorithm, which leads us to the final question. 

\begin{problem}
\label{prob:turingcomplete}
    Is {\sc two-dimensional finite subtraction} Turing complete?
\end{problem}

\begin{figure}[htbp!]
 \begin{center}
 \includegraphics[width=\textwidth]{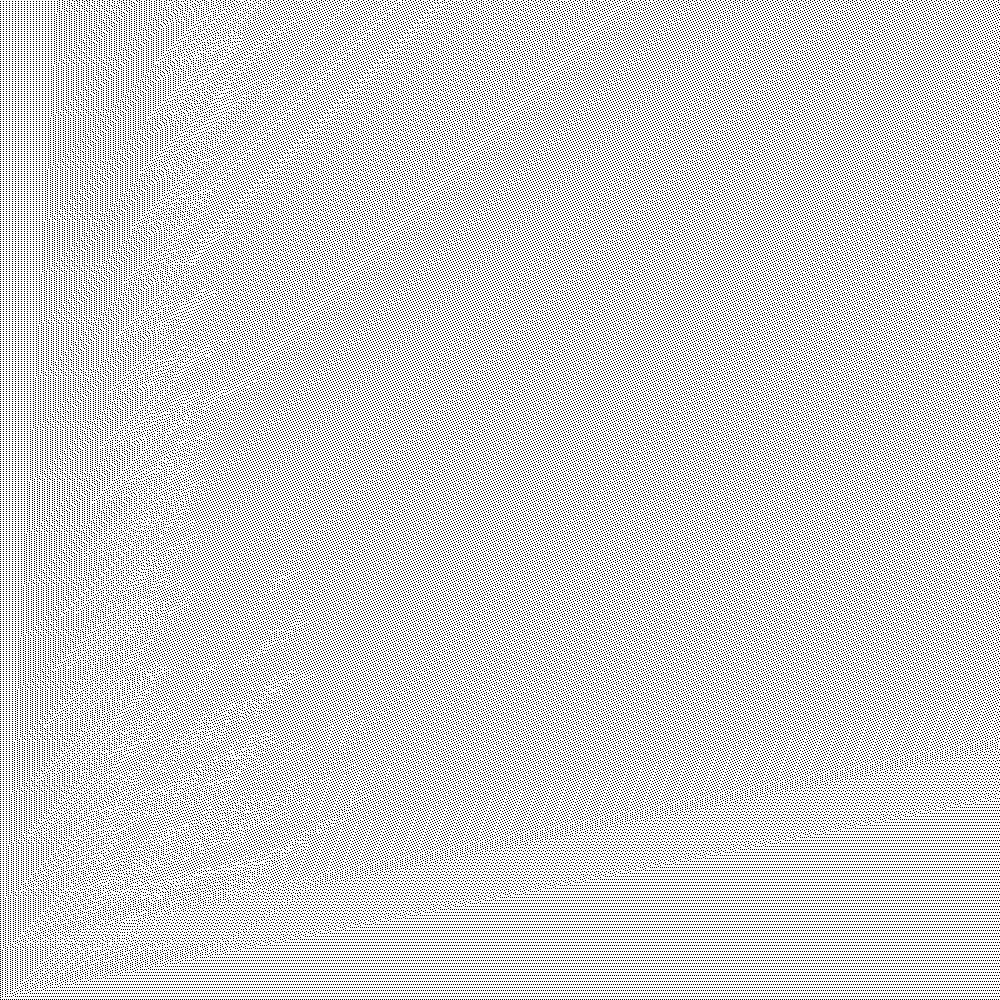}
 \end{center}
 \caption{The initial outcomes of the game $S=\{(0,1), (1,0), (1,1), (1,2), (2,2), (2,51), (4,3), (4,4), (13,1)\}$ %$S=\{(1,0), (0,1), (1,1), (2,2), (4,4), (2,51), (13,1), (4,3), (1,2)\}$ 
 (1000 by 1000 positions).}
 \label{fig:chaos}
 \end{figure}
 
 \begin{figure}[htbp!]
  \begin{center}
  %\includegraphics[width=14cm]%{somemorechaos.png}
  %\vfill{.5 cm}
  %\vspace{2.1 cm} 
  \includegraphics[width=\textwidth]{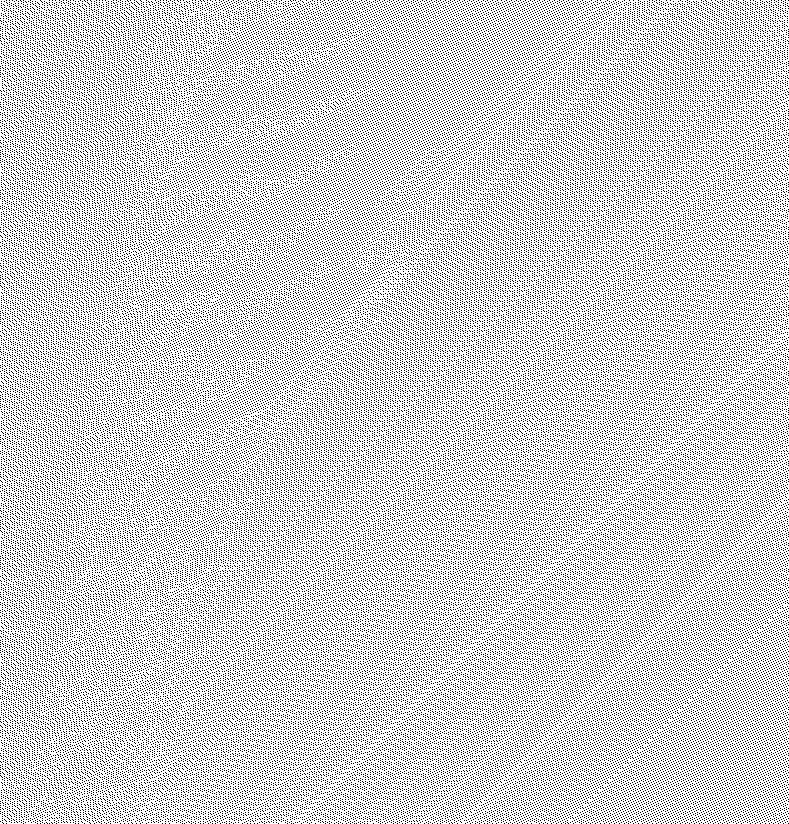}
  \end{center}
  \caption{Some later outcome patterns of the game $S$ in Figure~\ref{fig:chaos}; a chaotic emergence of a variety of `same-slope-gliders' appears to continue for larger board sizes.}
  \label{fig:chaos2}
  \end{figure}

\begin{figure}[htbp!]
  \begin{center}
  \includegraphics[width=\textwidth]
  {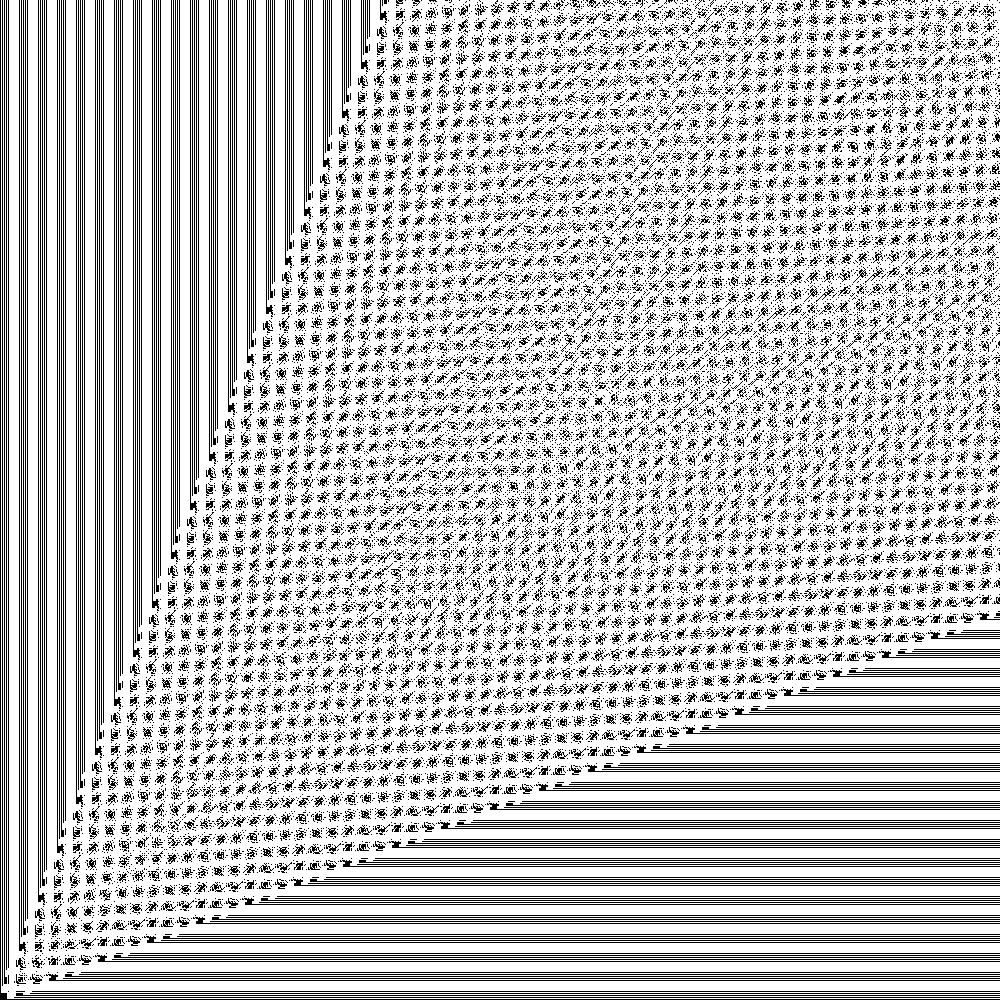}
  \end{center}
  \caption{The max-symmetric {\sc five-move} $\{(1,7),(7,1),(10,10)\}_{\mathrm{ms}}$. Max-symmetric rulesets often have unsettled outcome geometry; here visualized to 1000 by 1000.}
  \label{fig:maxsym1}
  \end{figure}
  
\begin{figure}[htbp!]
  \begin{center}
 \includegraphics[width=\textwidth]{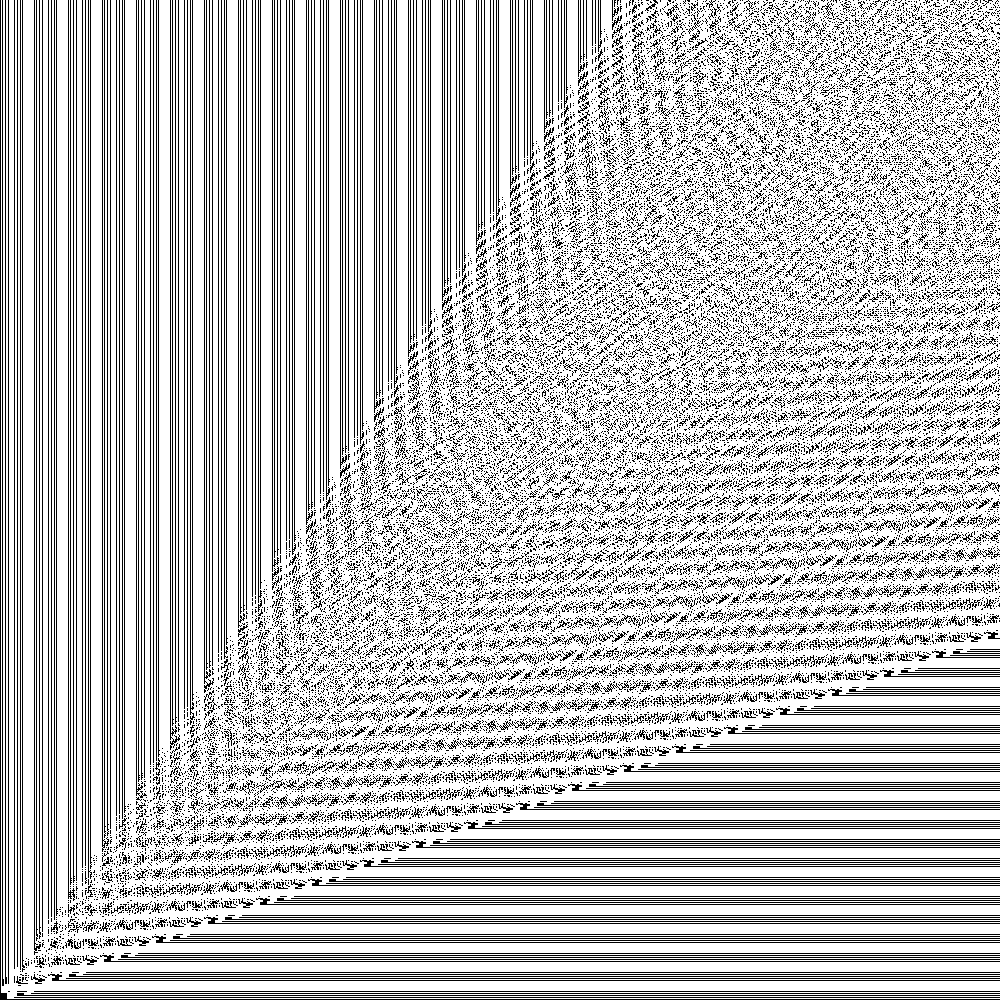}
 \end{center}
  \caption{The max-symmetric {\sc five-move} $\{(1,7),(7,1),(11,10)\}_{\mathrm{ms}}$ (1000 by 1000 positions).  }
  \label{fig:maxsym2}
  \end{figure}

\begin{figure}[htbp!]
  \begin{center}
  \includegraphics[width=12.5cm, clip=true, trim = 0 0 0 0]{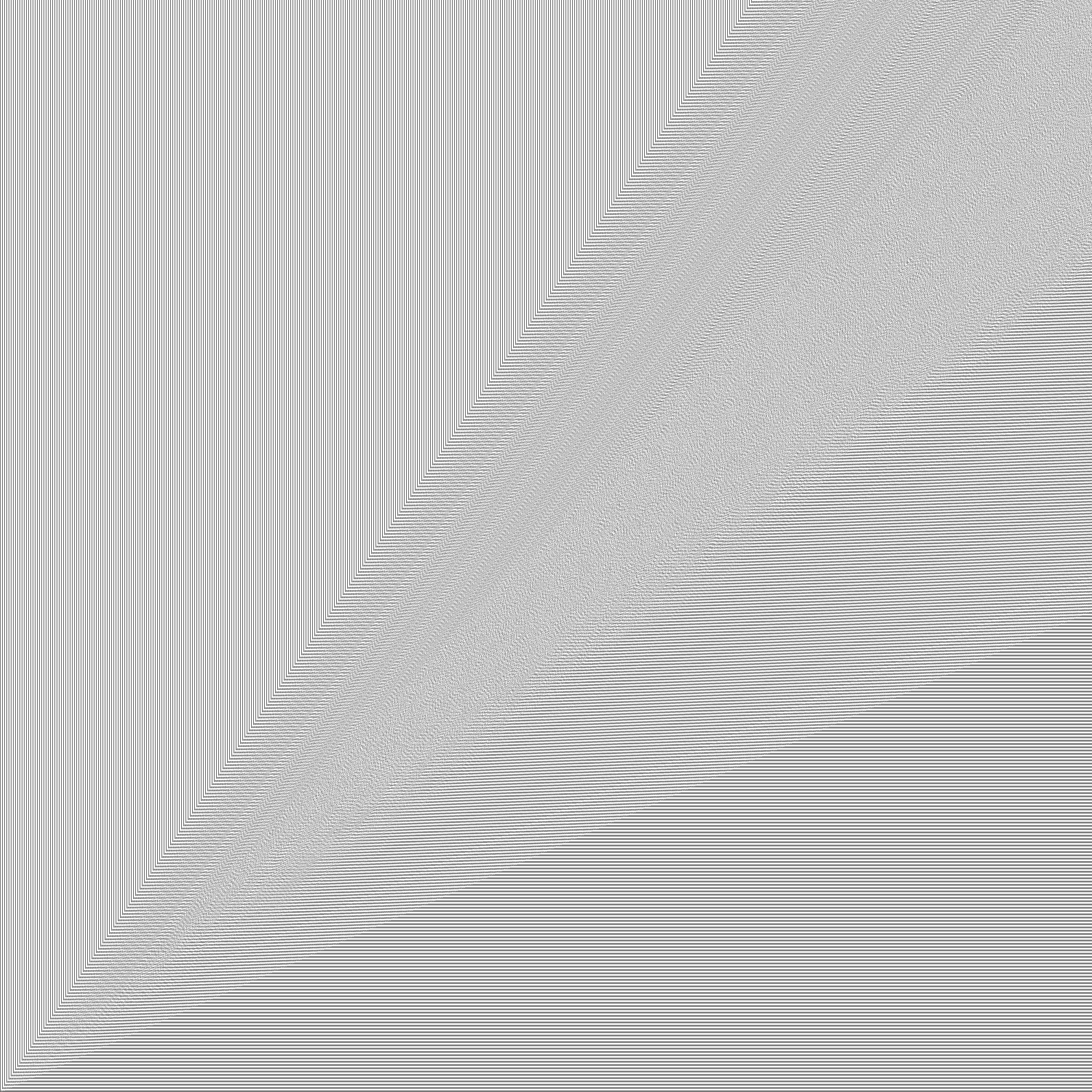}
  \includegraphics[width=6.2cm, clip=true, trim = 0 0 4850 4850]{215188maxSym5000.png}
   \includegraphics[width=6.2cm, clip=true, trim = 4850 4850 0 0]{215188maxSym5000.png}
  \end{center}
  \caption{The max-symmetric ruleset {\sc five-move}, $\{(2,1),(5,1),(8,8)\}_{\mathrm{ms}}$, 5000 by 5000 positions (below the first and last 150 by 150, respectively.} 
  \label{fig:maxsym3}
  \end{figure}

\begin{figure}[htbp!]
  \begin{center}
  \includegraphics[width=12.5cm, clip=true, trim = 0 0 0 0]{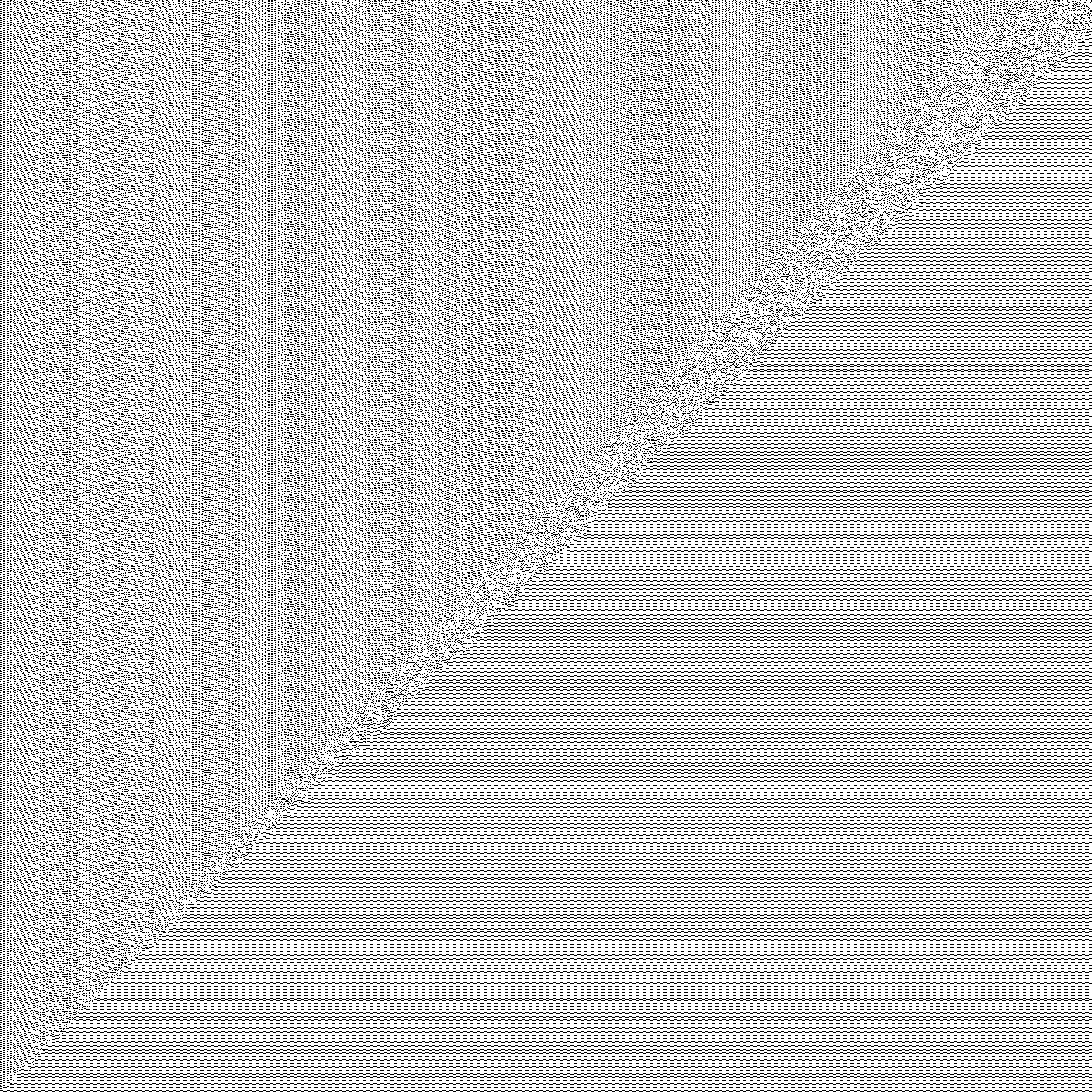}
  \includegraphics[width=6.2cm, clip=true, trim = 0 0 4850 4850]{025089Sym5000.png}
   \includegraphics[width=6.2cm, clip=true, trim = 4850 4850 0 0]{025089Sym5000.png}
  \end{center}
  \caption{The max-symmetric ruleset {\sc five-move}, $\{(0,2),(5,0),(8,8)\}_{\mathrm{ms}}$, 5000 by 5000 positions (below the first and last 150 by 150, respectively). Note the big variation of ornaments on the lower and upper regions that appear to be generated from the middle region. (Zoom in the picture to see this, and follow the edges of the middle region.) The next page illustrates a selection of this behavior.} 
  \label{fig:maxsym4}
  \end{figure}
  
  \begin{figure}
      \centering
     \includegraphics[width=14cm, clip=true, trim = 3400 1800 1400 2900]{025089Sym5000.png}
\end{figure}
  
  \clearpage
\noindent {\bf Acknowledgements.} We thank Raj Aryan Agrawal, Dr. Koki Suetsugu, Dr. Tomoaki Abuku, Dr. Hironori Kiya and Prof. Hideki Tsuiki for inspiring discussions regarding patterns and games. We thank Prof. Alex Fink for pointing out some important references.


\begin{thebibliography}{10}

\bibitem{abuku2019combination}
Tomoaki Abuku, Masanori Fukui, Ko~Sakai, and Koki Suetsugu.
\newblock On a combination of the cyclic nimhoff and subtraction games.
\newblock {\em Tsukuba J. Math.}, 43(2):241--249, 2019.

\bibitem{adhikari2024}
Tirthankar Adhikari.
\newblock {CGT}: Graph analysis.
\newblock {\em \url{https://github.com/Lhisoka/Line-detection}}, 2023.

\bibitem{althofer1995superlinear}
Ingo Alth{\"o}fer and J{\"o}rg B{\"u}ltermann.
\newblock Superlinear period lengths in some subtraction games.
\newblock {\em Theoret. Comput. Sci.}, 148(1):111--119, 1995.

\bibitem{austin1976impartial}
Richard~Bruce Austin.
\newblock {\em Impartial and partisan games (M.Sc. Thesis)}.
\newblock University of Calgary, Canada, 1976.

\bibitem{yakovurban}
Yakov Babichenko and Urban Larsson.
\newblock Golden games.
\newblock {\em Theoret. Comput. Sci.}, 891:50--58, 2021.

\bibitem{berlekamp2004winning}
Elwyn~R Berlekamp, John~H Conway, and Richard~K Guy.
\newblock {\em Winning Ways for Your Mathematical Plays}.
\newblock AK Peters/CRC Press, 2004 (1st edition Academic Press, 1982).

\bibitem{bouton1901nim}
Charles~L Bouton.
\newblock Nim, a game with a complete mathematical theory.
\newblock {\em Ann. of Math.}, 3(1/4):35--39, 1901.

\bibitem{cairns2010ultimately}
Grant Cairns and Nhan~Bao Ho.
\newblock Ultimately bipartite subtraction games.
\newblock {\em Australas. J. Combin}, 48:213--220, 2010.

\bibitem{carvalho2020combinatorics}
Alda Carvalho, Jo{\~a}o Neto, and Carlos Santos.
\newblock Combinatorics of jenga.
\newblock {\em Australas. J. Combin}, 76(1):87--104, 2020.

\bibitem{carvalho2012recursive}
Alda Carvalho, Carlos~P Santos, C{\'a}tia~Lente Dias, Francisco Coelho,
  Jo{\~a}o~Pedro Neto, and Sandra Vinagre.
\newblock A recursive process related to a partizan variation of wythoff.
\newblock {\em Integers}, 12(5):1029--1045, 2012.

\bibitem{cohensius2019cumulative}
Gal Cohensius, Urban Larsson, Reshef Meir, and David Wahlstedt.
\newblock Cumulative subtraction games.
\newblock {\em Electron. J. Combin.}, 26(P4.52), 2019.

\bibitem{coleman2020department}
Deidra Coleman, Jack Good, Michael Smith, Jennifer Travis, and Mark~Daniel
  Ward.
\newblock The periodicity of nim-sequences in two-element subtraction games.
\newblock {\em Integers}, 20:G5, 2020.

\bibitem{Cook2017}
Matthew Cook, Urban Larsson, and Turlough Neary.
\newblock A cellular automaton for blocking queen games.
\newblock {\em Nat. Comput.}, 16:397--410, 2017.

\bibitem{duchene2009another}
Eric Duchene, Aviezri~S Fraenkel, Sylvain Gravier, and Richard~J Nowakowski.
\newblock Another bridge between nim and wythoff.
\newblock {\em Australas. J. Combin}, 44:43--56, 2009.

\bibitem{duchene2019wythoff}
Eric Duch{\^e}ne, Aviezri~S. Fraenkel, Vladimir Gurvich, Nhan~Bao Ho, Clark
  Kimberling, and Urban Larsson.
\newblock Wythoff visions.
\newblock In {\em Games of No Chance 5}, volume~70, pages 35--87. MSRI,
  Cambridge Univ. Press, 2019.

\bibitem{duchene2022partizan}
Eric Duch{\^e}ne, Marc Heinrich, Richard~J Nowakowski, and Aline Parreau.
\newblock Partizan subtraction games.
\newblock {\em Integers}, 21B(A8), 2022.

\bibitem{DuRi2010}
Eric Duchêne and Michel Rigo.
\newblock Invariant games.
\newblock {\em Theoretical Computer Science}, 411(34):3169--3180, 2010.

\bibitem{eppstein2018faster}
David Eppstein.
\newblock Faster evaluation of subtraction games.
\newblock In {\em 9th International Conference on Fun with Algorithms (FUN
  2018)}, 2018.

\bibitem{fink2012lattice}
Alex Fink.
\newblock Lattice games without rational strategies.
\newblock {\em J. Combinat. Theory Ser. A}, 119(2):450--459, 2012.

\bibitem{Flammenkamp_1997}
Achim Flammenkamp.
\newblock {\em Lange Perioden in Subtraktions-Spielen}.
\newblock der Fakult\"at f\"ur Mathematik der Universit\"at Bielefeld, 1997.

\bibitem{fox2014aperiodic}
Nathan Fox.
\newblock On aperiodic subtraction games with bounded nim sequence.
\newblock {\em Preprint arXiv:1407.2823}, 2014.

\bibitem{FoxLar}
Nathan Fox and Urban Larsson.
\newblock An aperiodic subtraction game of nim-dimension two.
\newblock {\em J. Integer Seq.}, 18(15.7.4.), 2015.

\bibitem{fraenkel2004complexity}
Aviezri~S Fraenkel.
\newblock Complexity, appeal and challenges of combinatorial games.
\newblock {\em Theoret. Comput. Sci.}, 313(3):393--415, 2004.

\bibitem{fraenkel1987partizan}
Aviezri~S Fraenkel and Anton Kotzig.
\newblock Partizan octal games: partizan subtraction games.
\newblock {\em Internat. J. Game Theory}, 16(2):145--154, 1987.

\bibitem{MR4188748}
Aviezri~S. Fraenkel and Urban Larsson.
\newblock Games on arbitrarily large rats and playability.
\newblock {\em Integers}, 19:Paper No. G04, 29, 2019.

\bibitem{fraenkel1991nimhoff}
Aviezri~S Fraenkel and Mordeghai Lorberbom.
\newblock Nimhoff games.
\newblock {\em J. Combinatorial Theory Ser. A}, 58(1):1--25, 1991.

\bibitem{friedman2019geometric}
Eric Friedman, Scott~M. Garrabrant, Ilona~K. Phipps-Morgan, Adam.~S. Landsberg,
  and Urban Larsson.
\newblock Geometric analysis of a generalized {W}ythoff game.
\newblock In {\em Games of no chance 5}, volume~70, pages 343--372. MSRI,
  Cambridge Univ. Press, 2019.

\bibitem{golomb1966mathematical}
Solomon~W Golomb.
\newblock A mathematical investigation of games of “take-away”.
\newblock {\em J. Combinatorial Theory}, 1(4):443--458, 1966.

\bibitem{guo2011lattice}
Alan Guo and Ezra Miller.
\newblock Lattice point methods for combinatorial games.
\newblock {\em Advances in Applied Mathematics}, 46(1-4):363--378, 2011.

\bibitem{guo2013algorithms}
Alan Guo and Ezra Miller.
\newblock Algorithms for lattice games.
\newblock {\em Internat. J. Game Theory}, 42:777--788, 2013.

\bibitem{guo2009potential}
Alan Guo, Ezra Miller, and Michael Weimerskirch.
\newblock Potential applications of commutative algebra to combinatorial game
  theory.
\newblock {\em Kommutative Algebra, Oberwolfach Rep. 22}, pages 23--26, 2009.

\bibitem{Gurvich}
Vladimir Gurvich and Mikhail Vyalyi.
\newblock Computational hardness of multidimensional subtraction games.
\newblock In {\em Computer science theory and applications}, volume 12159 of
  {\em Lecture Notes in Comput. Sci.}, pages 237--249. Springer, Cham, 2020.

\bibitem{ho2015expansion}
Nhan~Bao Ho.
\newblock On the expansion of three-element subtraction sets.
\newblock {\em Theoret. Comput. Sci.}, 582:35--47, 2015.

\bibitem{larsson2011blocking}
Urban Larsson.
\newblock Blocking {W}ythoff nim.
\newblock {\em Electron. J. Combin.}, 18(1):P120, 2011.

\bibitem{larsson2012operator}
Urban Larsson.
\newblock The *-operator and invariant subtraction games.
\newblock {\em Theoret. Comput. Sci.}, 422:52--58, 2012.

\bibitem{Larsson2012GDWN}
Urban Larsson.
\newblock A generalized diagonal wythoff nim.
\newblock {\em Integers}, 12(5):1003--1027, 2012.

\bibitem{larsson2013impartial}
Urban Larsson.
\newblock {\em Impartial Games and Recursive Functions}.
\newblock Chalmers Tekn. H\"ogsk. (Sweden), 2013.

\bibitem{Larsson2014split}
Urban Larsson.
\newblock Wythoff nim extensions and splitting sequences.
\newblock {\em J. Integer Seq.}, 17(5):14.5.7, 2014.

\bibitem{LaHeFr2011}
Urban Larsson, Peter Hegarty, and Aviezri~S. Fraenkel.
\newblock Invariant and dual subtraction games resolving the duchêne–rigo
  conjecture.
\newblock {\em Theoretical Computer Science}, 412(8):729--735, 2011.

\bibitem{LarssonUrbanNeil}
Urban Larsson, Neil~A. McKay, Richard~J. Nowakowski, and Angela Siegel.
\newblock Wythoff partizan subtraction.
\newblock {\em Internat. J. Game Theory}, 47(2):613--652, 2018.

\bibitem{larsson2013heaps}
Urban Larsson and Johan W{\"a}stlund.
\newblock From heaps of matches to the limits of computability.
\newblock {\em Electron. J. Combin.}, 20(P41), 2013.

\bibitem{larsson2014maharaja}
Urban Larsson and Johan W{\"a}stlund.
\newblock Maharaja nim, {W}ythoff's queen meets the knight.
\newblock {\em Integers}, 14:G05, 2014.

\bibitem{PlambeckSiegel2008}
Thane~E. Plambeck and Aaron~N. Siegel.
\newblock Misère quotients for impartial games.
\newblock {\em Journal of Combinatorial Theory, Series A}, 115(4):593--622,
  2008.

\bibitem{siegel2013combinatorial}
Aaron~N Siegel.
\newblock {\em Combinatorial Game Theory}, volume 146.
\newblock Amer. Math. Soc., 2013.

\bibitem{siegel2005finite}
Angela Siegel.
\newblock {\em Finite Excluded Subtraction Sets and Infinite Modular Nim.}
\newblock Dalhousie University, 2005.

\bibitem{ward2016conjecture}
Mark~Daniel Ward.
\newblock A conjecture about periods in subtraction games.
\newblock {\em Preprint arXiv:1606.04029}, 2016.

\bibitem{mikeW}
Mike Weimerskirch.
\newblock Generalized mis\'ere play.
\newblock In {\em Games of No Chance 5}, volume~70. MSRI, Cambridge Univ.
  Press, 2019.

\bibitem{wythoff1907modification}
Willem~A Wythoff.
\newblock A modification of the game of nim.
\newblock {\em Nieuw Arch. Wisk.}, 7(2):199--202, 1907.

\bibitem{zermelo1913anwendung}
Ernst Zermelo.
\newblock {\"U}ber eine anwendung der mengenlehre auf die theorie des
  schachspiels.
\newblock In {\em Proceedings of the fifth international congress of
  mathematicians}, volume~2, pages 501--504. Cambridge Univ. Press, 1913.

\bibitem{zhang2021linearity}
Shenxing Zhang.
\newblock On linearity of the periods of subtraction games.
\newblock {\em Preprint \url{https://zhangshenxing.gitee.io}}, 2021.

\end{thebibliography}
\end{document}